\documentclass[english, usenames,dvipsnames,svgnames,table]{article}
\usepackage{preamble}

\title{Twisted Graded Categories}

\begin{document}

\begin{titlepage}
    \maketitle
    \begin{abstract}
        Given a presentably symmetric monoidal $\infty$-category $\cC$ and an $\EE_{\infty}$-monoid $M$, we introduce and classify twisted graded categories, which generalize the Day convolution structure on $\Fun(M, \cC)$. 
        These are governed by a braiding encoded in symmetric group actions on tensor powers of invertible elements, whose character depends only on the $\TT$-equivariant monoidal dimension. 
        We analyze the $\TT$-action on the dimension of invertible objects and identify it with the $\TT$-transfer map. 
        As applications, we compute braiding characters in examples arising from higher cyclotomic extensions, including the $(\SS, \chrHeight+1)$-oriented extension of $\ModEn$ at all primes and heights, and the cyclotomic closure of $\Vectn$ at low heights.        
        \begin{figure}[h]
            \centering
            \includegraphics[width=\linewidth]{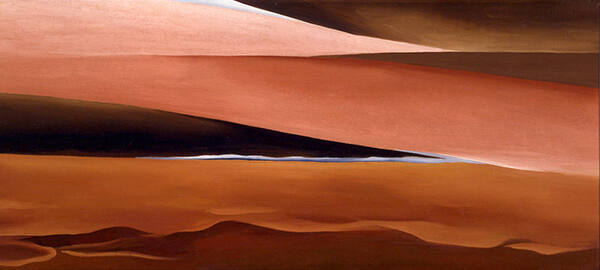}
            \caption*{Desert Abstraction, Georgia O'Keeffe, 1931}
        \end{figure}
    \end{abstract}
\end{titlepage}

\tableofcontents
\newpage

%%%%%%%%%%%%%%%%%%%%%%%%%%%%%%%%%%%%%%%%%%%%%%%%%%%%%%%%%%%%%%%%%%%%%%%%%%%%%%%%
%%%%%%%%%%%%%%%%%%%%%%%%%%%%%%%%%%%%%%%%%%%%%%%%%%%%%%%%%%%%%%%%%%%%%%%%%%%%%%%%
\section{Introduction}
\label{sec:intro}
%%%%%%%%%%%%%%%%%%%%%%%%%%%%%%%%%%%%%%%%%%%%%%%%%%%%%%%%%%%%%%%%%%%%%%%%%%%%%%%%
%%%%%%%%%%%%%%%%%%%%%%%%%%%%%%%%%%%%%%%%%%%%%%%%%%%%%%%%%%%%%%%%%%%%%%%%%%%%%%%%
    Given a space $X$ and a presentably symmetric monoidal category $\cC$\footnote{Throughout this article, we will use the term `category' to mean an `$(\infty,1)$-category.' We will use the term `space' to mean an '$(\infty,0)$-category' or an `$\infty$-groupoid'.}, one can consider the category of $X$-graded $\cC$-objects
    \begin{equation*}
        \Gr_X \cC \coloneqq \cC[X] \simeq \Fun(X, \cC).
    \end{equation*}
    Furthermore, when $X$ is equipped with the structure of a (commutative) monoid, the category $\Gr_X \cC$ naturally inherits a (symmetric) monoidal structure called the Day convolution.

    A key example of a graded category arises in algebra, where the category of $\ZZ$-graded abelian groups serves as a natural target for the homotopy (or homology) groups functor. 
    Since $\ZZ$ is a commutative monoid, one can equip this category with the Day symmetric monoidal structure. However, the functor $\pi_*$ (or $\H_*$) will not be symmetric monoidal, or even lax symmetric monoidal. 
    To remedy this, one instead considers a twisted version, in which the monoidal structure agrees with the Day convolution given by
    \begin{equation*}
        (A_{\bullet} \otimes B_{\bullet})_{\degree} = \bigoplus_{i+j = \degree} A_i \otimes B_j,
    \end{equation*}
    while the braiding is given by the Koszul sign rule.

    A similar phenomenon happens in the categorified context of super $\field$-linear categories, that is modules in $\PrL$ over the category $\sVect$ of super vector spaces. Given two $\ZZ$-graded super linear categories $\cC_{\bullet}$ and $\cD_{\bullet}$, we define their graded tensor product to be
    \begin{equation*}
        (\cC_{\bullet} \otimes \cD_{\bullet})_{\degree} = \bigoplus_{i + j = \degree} \cC_i \otimes \cD_j
    \end{equation*}
    with the braiding given by tensoring with the $(1|0)$- or $(0|1)$-dimensional super vector space, depending on the parity of the product of the degrees.

    This paper is devoted to the study of $\EE_{\infty}$-lifts of the Day convolution, generalizing the two examples discussed above. Specifically, given a commutative monoid $M$ and a presentably symmetric monoidal category $\cC$, we study symmetric monoidal structures on $\Fun(M,\cC)$ with an $\EE_1$-isomorphism to the Day convolution. We call categories equipped with such structures \emph{twisted $M$-graded categories}.
    
    A natural way to construct such categories is via the Thom construction. Let $\cU$ be a presentably symmetric monoidal category and $\Zchar \colon M \to \cU\units$.\footnote{Here, $\cU\units$ denotes the Picard spectrum of $\cU$. Since categorification plays a central and unavoidable role in this paper, we refrain from using the ambiguous notation $\pic(\cU)$, which could refer either to $\cU\units$ or to $\Mod_{\cU}(\PrL)\units$.} 
    The Thom construction of $\Zchar$ is defined to be its colimit $\Th(\Zchar) \coloneqq \colim_M \Zchar$ computed in $\cU$.
    If $\Zchar$ is moreover a map of $\EE_k$-monoids, then the Thom construction $\Th(\Zchar)$ is naturally an $\EE_k$-algebra in $\cU$. 
    When $\Zchar$ is nullhomotopic as an $\EE_k$-map, this colimit is isomorphic, as an $\EE_k$-algebra, to the monoid-algebra $\ounit_{\cU}[M]$. 
    More generally, the map $\Zchar \colon M \to \cU\units$ may be non-trivial as an $\EE_k$-map, yet become nullhomotopic when viewed as an $\EE_{\ell}$-map for some $\ell < k$. Any choice of such a nullhomotopy gives rise to an $\EE_{\ell}$-isomorphism $\Th(\Zchar) \isoto \ounit_{\cU}[M]$.
    
    We are therefore interested in the space of maps $M \to \cU\units$ with an $\EE_1$-nullhomotopy, which we denote by $\Map\Enull(M, \cU\units)$. We remark that when $M$ is grouplike, i.e.\ a connective spectrum, then
    \begin{equation*}
        \Map\Enull(M, \cU\units) \simeq \Map(M\Enull, \ounit_{\cU}\units)
    \end{equation*}
    where $M\Enull$ is the connective spectrum defined by
    \begin{equation*}
        M\Enull \coloneqq \Omega \cofib(\Sigma^{-1} \redSS[\B M] \to M) \qin \cnSp,
    \end{equation*}
    where $\redSS[-] \colon \spc_* \to \cnSp$ is the reduced suspension spectrum functor.

    Restricting to the case $\cU = \Mod_{\cC} \coloneqq \Mod_{\cC}(\PrL)$\footnote{
        The category $\Mod_{\cC}$ is not presentable, as it is too large. However, there exists a sufficiently large cardinal $\kappa$ such that any invertible $\cC$-linear category is $\kappa$-compactly generated. Since we are interested in the Thom construction, given as a colimit of a map to the Picard spectrum, we can replace $\Mod_{\cC}$ with $\Mod_{\cC}(\PrL_{\kappa})$, which is presentably symmetric monoidal. As the natural functor $\Mod_{\cC}(\PrL_{\kappa}) \to \Mod_{\cC}$ is both colimit preserving and symmetric monoidal, this replacement allows us to avoid the size issue. For brevity, we will continue to write $\Mod_{\cC}$ instead of $\Mod_{\cC}(\PrL_{\kappa})$.
    }, the monoid-algebra $\cC[M] = \Gr_M \cC$ is the category of $M$-graded objects with the Day convolution. We define:
    \begin{definition}
        Let $M$ be a commutative monoid, $\cC$ be a presentably symmetric monoidal category and $\Zchar \in \Map\Enull(M, \Mod_{\cC}\units)$. We let $\Gr^{\Zchar}_M \cC \in \CAlg_{\cC}(\PrL)$ be its Thom construction, equipped with the $\EE_1$-isomorphism to $\Gr_M \cC$.
    \end{definition}
    In particular, $\Gr^0_M \cC \simeq \Gr_M \cC$ in $\CAlg_{\cC}(\PrL)$.

    \begin{example}
        Let $\field$ be a field of characteristic different from 2. Then it admits a (non-trivial) minus one map $(-1) \colon \ZZ/2 \to \field\units$. Define the map
        \begin{equation*}
            \Kos \colon \ZZ \onto \ZZ/2 \xto{\Sq^2} \Sigma^2 \ZZ/2 \xto{\Sigma^2(-1)} \Sigma^2 \field\units \to \Mod_{\Vect}\units.
        \end{equation*}
        Since the second Steenrod square $\Sq^2$ is $\EE_1$-nullhomotopic, so is the composition. The twisted graded category $\Gr^{\Kos}_{\ZZ} \Vect$ is the usual Koszul-twisted $\ZZ$-graded category of vector spaces.

        Note that the map factors through $\ZZ/2$, and we have $\Gr^{\Kos}_{\ZZ/2} \Vect \simeq \sVect$.
    \end{example}
    
    An analogous construction gives the twisted braiding on super linear categories discussed above.

    % \begin{example}
    %     Let $\field$ be a field of characteristic different from 2, and consider $\cC = \Mod_{\sVect}$ --- the category of super $\field$-linear categories. 
    %     Consider the cofiber sequence corresponding to the connected cover of $\Mod_{\sVect}\units$:
    %     \begin{equation*}
    %         \Sigma \field\units \to \Mod_{\sVect}\units \to \pi_0 \Mod_{\sVect}\units \xto{\partial} \Sigma^2 \field\units.
    %     \end{equation*}
    %     Recall that $\pi_0 \Mod_{\sVect}\units \cong \ZZ/2$ is generated by the $(0|1)$-dimensional super space. We define
    %     \begin{equation*}
    %         \Zchar \colon \ZZ \onto \ZZ/2 \xto{\partial} \Sigma^2 \field\units \to \Mod_{\Mod_{\sVect}}\units.
    %     \end{equation*}
    %     The twisted graded category $\Gr^{\Zchar}_{\ZZ}\Mod_{\sVect}$ is the one described above where the braiding is twisted by tensoring with the $(0|1)$-dimensional super vector space.
    % \end{example}

    % \begin{definition}
    %     Let $M,N$ be commutative monoids. Define $\Map\Enull(M,N)$ as the space of maps of commutative monoids with an $\EE_1$-nullhomotopy. That is, as the fiber
    %     \begin{equation*}
    %         \Map\Enull(M,N) \coloneqq \fib\left(\Map_{\EE_{\infty}}(M,N) \to \Map_{\EE_1}(M,N)\right).
    %     \end{equation*}
    % \end{definition}

    % We are mostly interested in the case where $M$ is grouplike, in this case it is simple to see that $\Map\Enull(M,-)$ is representable.
    Restricting to the discrete case, the Thom categories are exactly all twisted graded categories:   
    \begin{alphtheorem}[\cref{thm:twFun-are-all-lifts}]\label{alphthm:twFun-are-all-lifts}
        Let $A \in \Ab$ and $\cC\in \CAlg(\PrL)$ be semiadditive. Then $\Gr^{(-)}_A \cC$ defines an isomorphism between the spaces $\Map\Enull(A,\Mod_{\cC}\units)$ and the space of twisted $A$-graded categories of objects in $\cC$.
    \end{alphtheorem}

    \begin{example}[$\ZZ$-graded categories]
        In the case $A = \ZZ$, one can simply compute that $\ZZ\Enull \simeq \tau_{\ge 1}\SS$. Therefore, for any symmetric monoidal category $\cC$
        \begin{equation*}
            \Map\Enull(\ZZ, \Mod_{\cC}\units) \simeq \Map(\tau_{\ge 1} \SS, \cC\units).
        \end{equation*}
        As $\tau_{\ge 1} \SS$ is connected, maps from $\tau_{\ge 1} \SS$ to $\cC\units$ factor through the connected cover map $\Sigma \ounit_{\cC}\units \to \cC\units$:
        \begin{equation*}
            \Map\Enull(\ZZ, \Mod_{\cC}\units) \simeq \Map(\Omega \tau_{\ge 1} \SS, \ounit_{\cC}\units) \simeq \Map(\Omega \SS, \ounit_{\cC}\units).
        \end{equation*}
        In particular, if $\ounit_{\cC}$ is discrete (e.g.\ $\cC = \Mod_R$ or $\cC = \Mod_R\heart$ for $R$ discrete), then
        \begin{equation*}
             \Map\Enull(\ZZ, \Mod_{\cC}\units) \simeq \Map_{\Ab}(\ZZ/2, \ounit_{\cC}\units).
        \end{equation*}
        For example, if $\field$ is a field of characteristic different from 2, there are exactly two twisted $\ZZ$-graded structures on $\Vect$ --- the usual, and the Koszul-twisted. If $\field$ is of characteristic 2 there are no non-trivial twisted $\ZZ$-graded structures on $\Vect$.

        Note that the same holds for $\sVect$, although even in characteristic $\neq 2$, $\Gr^{\Kos}_{\ZZ} \sVect \simeq \Gr_{\ZZ} \sVect$ as symmetric monoidal categories. This corresponds to the fact that $\Gr_{\ZZ}\sVect$ admits a non-trivial isomorphism (and therefore two different $\EE_1$-trivializations). 
    \end{example}

    A variant of twisted graded categories is that of \emph{homotopy graded categories}, where instead of an $\EE_1$-isomorphism we only require an $\EE_0$-isomorphism $\Gr_A \cC \isoto \cD$ together with a compatible $\EE_1$-isomorphism of homotopy $(k,1)$-categories
    \begin{equation*}
        \h_k \Gr_A \cC \isoto \h_k \cD.  
    \end{equation*}
    We show that this data is equivalent to a map of connective spectra $A \to \Mod_{\cC}\units$ together with an $\EE_1$-nullhomotopy of the composition
    \begin{equation*}
        A \to \Mod_{\cC}\units \to \tau_{\leq k+1}\Mod_{\cC}\units,  
    \end{equation*}
    see \cref{thm:homotopy-twFun-are-all-lifts}.
    We denote the space of such maps by $\Map^{\h,k}_{\Enull}(A,\Mod_{\cC}\units)$.
    
    In the process of proving \cref{alphthm:twFun-are-all-lifts}, we study the Picard spectrum of the Day convolution, which may be of independent interest:    
    \begin{proposition}[\cref{prop:picard-of-Day}]
        Let $\cC \in \Alg_{\EE_k}(\PrL)$ be semiadditive and connected\footnote{That is, $\cC$ does not decompose as a product of symmetric monoidal categories.}.
        Let $M$ be a discrete commutative monoid. Then there is an isomorphism of $\EE_k$-groups
        \begin{equation*}
           (\Gr_M \cC)\units \simeq \cC\units \times M\units.
        \end{equation*}
    \end{proposition}

    %%%%%%%%%%%%%%%%%%%%%%%%%%%%%%%%%%%%%%%%%%%%%%%%%%%%%%%%%%%%%%%%%%%%%%%%%%%%%%%%
    \subsection{Braiding}
    %%%%%%%%%%%%%%%%%%%%%%%%%%%%%%%%%%%%%%%%%%%%%%%%%%%%%%%%%%%%%%%%%%%%%%%%%%%%%%%%
        The construction of twisted graded categories gives rise to several different symmetric monoidal structures that are isomorphic as monoidal categories. The distinction between such symmetric monoidal structures lies in their braiding; that is, the different identifications of objects of the form
        \begin{equation*}
            X_1 \otimes \cdots \otimes X_{\degree}.
        \end{equation*}
        Taking all the $X_i$ to be equal, we obtain a natural $\Sm$-action on $\Tm X \coloneqq X\om$. Since $\bigsqcup_{\degree} \B \Sm$ is the free commutative monoid, the induced map
        \begin{equation*}
            \T X \colon \bigsqcup_{\degree} \B\Sm \to \cC
        \end{equation*}
        is symmetric monoidal. In a slight abuse of terminology, we refer to this map as the \emph{braiding of $X$}.
        
        For example, in the usual $\ZZ$-graded category $\Gr^{\Kos}_{\ZZ}\Vect$, the $\Sm$-braiding on $\field\shift{1}$ --- the one dimensional vector space in degree 1, is given by the sign representation of $\Sm$. In the case of $\Mod_{\sVect}$, the braiding on $\sVect\shift{1}$ is given by the higher sign representation, as studied in \cite{Ganter-Kapranov-2014-exterior-categories}.

        The braiding of invertible objects is much better understood: If $Z \in \cD\units$, the braiding map factors through the group-completion of the free commutative monoid, i.e.\ the sphere spectrum, and lands in the Picard spectrum
        \begin{equation*}
            \T Z \colon \bigsqcup_{\degree} \B\Sm \to \SS \to \cD\units \into \cD.
        \end{equation*}
        In particular, by taking connected covers, we get the map of spectra
        \begin{equation*}
            \hchar_{Z} \colon \tau_{\ge 1} \SS \to \Sigma \ounit_{\cD}\units,
        \end{equation*}
        which we also call the braiding of $Z$. 

        Given a (homotopy) twisted $A$-graded category, any relation $a_1 + \cdots + a_r = 0$ in $A$, gives rise to a natural isomorphism
        \begin{equation*}
            \hchar_{\ounit_{\cC}\shift{a_1}} \cdots \hchar_{\ounit_{\cC}\shift{a_r}} \simeq 1 \qin \Map(\tau_{\ge 1} \SS, \ounit_{\cC}\units).
        \end{equation*}

        As intuition suggests, we show that in the twisted graded case, the braiding of objects indeed determine the symmetric monoidal structure. Moreover, it suffices to consider the braiding of generators and their relations in the following sense:  
        Let $A \in \Ab$, $\cC \in \CAlg(\PrL)$, and $\Zchar \in \Map^{\h,k}\Enull(A, \Mod_{\cC}\units)$. For $V \in \cC$ and $a \in A$, let $V\shift{a} \in \Gr^{\Zchar}_A \cC$ denote the functor $A \to \cC$ sending $b \in A$ to $V$ if $b = a$, and to $\emptyset$ otherwise.

        \begin{proposition}[\cref{cor:braiding-for-general-A-determined}]
            The braiding of $\ounit_{\cC}\shift{a} \in \Gr^{\Zchar}_A \cC$ for a set of generators of $A$, together with the relations among them, determines $\Zchar \colon A \to \Sigma^2 \ounit_{\cC}\units$.\footnote{Any map $\Zchar \colon A \to \Mod_{\cC}\units$ equipped with an $\EE_1$-nullhomotopy (of some truncation) lifts uniquely to $\Sigma^2 \ounit_{\cC}\units$, see \cref{lem:homotopy-E1-null-factors-through-Sigma^2}.}
        \end{proposition}

        In the universal case $A = \ZZ$, this determines both the symmetric monoidal structure and the $\EE_1$-isomorphism to the Day monoidal structure. We can also identify when two such twisted graded categories are equivalent, i.e., when two maps $\Zchar_1, \Zchar_2 \in \Map\Enull(\ZZ, \Mod_{\cC}\units)$ agree after forgetting the $\EE_1$-nullhomotopies:
        % This can be done in the universal case $A = \ZZ$.
        Noting that the unit of any twisted graded category of $\cC$-objects is the unit of $\cC$, one can compare braiding of invertible objects in different twisted graded categories. 
        \begin{proposition}[\cref{lem:braiding-of-A-determined-by-Z}, \cref{cor:graded-trivial-iff-picard}]\label{prop:intro-graded-trivial-iff-picard}
            Assume that there exists $Z \in \cC\units$ and an isomorphism between the braiding of $Z \in \cC$ and that of $\ounit_{\cC}\shift{1} \in \Gr^{\Zchar}_{\ZZ} \cC$. Then $\Gr^{\Zchar}_{\ZZ} \cC \simeq \Gr_{\ZZ} \cC$ in $\CAlg_{\cC}(\PrL)$.
        \end{proposition}

        In order to prove this proposition, we notice that the braiding of invertible elements defines a map 
        \begin{equation*}
            \hchar_{(-)} \colon \cC\units \to \Map(\tau_{\ge 1} \SS, \Sigma \ounit_{\cC}\units) \simeq \Map\Enull(\ZZ, \Sigma \cC\units).
        \end{equation*}
        This in turn, admits a natural map to $\Map(\ZZ, \Sigma \cC\units)$ forgetting the $\EE_1$-nullhomotopy. \cref{prop:intro-graded-trivial-iff-picard} then follows from

        \begin{proposition}[\cref{rmrk:kernel-of-forgetting-the-E1-nullhomotopy}]\label{prop:intro-exact-sequence}
            There is a short exact sequence
            \begin{equation*}
                0 \to \Pic(\cC) / \Pic^{\mrm{str}}(\cC) \xto{\hchar_{(-)}} \pi_0 \Map\Enull(\ZZ, \Sigma \cC\units) \to \pi_0 \Map(\ZZ, \Sigma \cC\units) \to 0,
            \end{equation*}
            where $\Pic(\cC) \coloneq \pi_0\cC\units$ is the Picard group, and $\Pic^{\mrm{str}}(\cC) = \pi_0\hom(\ZZ, \cC\units)$ is the group of strict Picard elements.
        \end{proposition}

    %%%%%%%%%%%%%%%%%%%%%%%%%%%%%%%%%%%%%%%%%%%%%%%%%%%%%%%%%%%%%%%%%%%%%%%%%%%%%%%%
    \subsection{Braiding character}
    %%%%%%%%%%%%%%%%%%%%%%%%%%%%%%%%%%%%%%%%%%%%%%%%%%%%%%%%%%%%%%%%%%%%%%%%%%%%%%%%

        The braiding in twisted graded categories is a complete invariant of the symmetric monoidal structure, and as such, it can be intricate to describe or compute explicitly. 
        Drawing from representation theory, for a dualizable object $W \in \cD\dbl$, one can consider a simpler invariant, namely the character of the $\Sm$-braidings $\chi_{\Tm W}$.
        These characters assemble into a map of commutative monoids
        \begin{equation*}
            \begin{split}
                \mscr{X}_{\T W} \colon \bigsqcup_{\degree} \L\B\Sm \to \ounit_{\cD}[t^{\pm 1}], \\
                \res{\mscr{X}_{\T W}}{\L\B\Sm} \coloneqq \chi_{\Tm W} \, t^{\degree}. %\qquad \colon \B\Sm \to \ounit_{\cC}[t^{\pm 1}].
            \end{split}
        \end{equation*}
        % sending the component $\L\B\Sm$ to $\chi_{\Tm W}\, t^{\degree}$. 
        When $\cD = \Gr^{\Zchar}_\ZZ \cC$ is a twisted $\ZZ$-graded category and $W = V\shift{1}$ for $V \in \cC\dbl$, this character coincides (see \cref{lem:braiding-characters-agree}) with the image under $\THH_{\cC}$ of the braiding
        \begin{equation*}
            \T V\shift{1} \colon \cC[\bigsqcup_{\degree} \B\Sm] \to \Gr^{\Zchar}_{\ZZ} \cC.
        \end{equation*}
        This invariant turns out to be computationally simple, and unlike the braiding, it depends only on the monoidal dimension of $W$. 
        Recall that any dualizable object $W$ admits a monoidal dimension with a $\TT$-action, $\dim(W) \in \End(\ounit_{\cD})^{\B\TT}$. The homomorphisms $\Ck \to \TT$ assemble into a map $\vee_k \B\Ck \to \B\TT$, and the corresponding restriction map $\End(\ounit_{\cD})^{\B\TT} \to \End(\ounit_{\cD})^{\vee_k \B\Ck}$ encodes the action of all finite cyclic subgroups.
        
        \begin{alphtheorem}[\cref{thm:braiding-depends-only-on-dim}]
            Let $\cD \in \CAlg(\PrL)$ and $W \in \cD\dbl$.
            Then $\mscr{X}_{\T W}$ depends only on $\dim W \in \End(\ounit_\cD)^{\vee_k \B\Ck}$.
        \end{alphtheorem} 
        The dependence is entirely explicit, with \cref{lem:character-of-Tm} giving a formula for the braiding character in terms of $\dim(W)$. Since the first version of this paper appeared on the arXiv, we have learned that Maxime Ramzi independently obtained this result in \cite[Lemma~4.7]{Ramzi-2025-endomorphisms-of-THH}.

        \begin{remark}
            The maps $\B\Ck \to \B\TT$ assemble to a map $\B\QQ/\ZZ = \colim_k \B\Ck \to \B\TT$. The restriction map for $\cD \in \CAlg(\PrL)$ factors as
            \begin{equation*}
                \End(\ounit_{\cD})^{\B\TT} \to \End(\ounit_{\cC})^{\B\QQ/\ZZ} \to \End(\ounit_{\cD})^{\vee_k \B\Ck}.
            \end{equation*}
            The further restriction to $\vee_k \B\Ck$ forgets the compatibility data between different values of $k$.
        \end{remark}

        In our case of interest, when $W$ is invertible, we classify the $\TT$-action on the dimension using the universal case, and identify it with the $\TT$-transfer map (see \cref{prop:dim-is-transfer}).
        % which is the underlying groupoid of $\SS$. That is,
        % \begin{equation*}
        %     \cD\units \simeq \Map^{\otimes}(\SS, \cD).
        % \end{equation*}
        % When $Z$ is invertible, its dimension is also invertible and lands in $(\End(\ounit_{\cD})\units)^{\B\TT}$. This space is represented by $\Sigma \SS[\B\TT]$:
        % \begin{equation*}
        %     (\End(\ounit_{\cD})\units)^{\B\TT} \simeq \Map^{\otimes}(\Sigma \SS[\B\TT], \cD).
        % \end{equation*}
        % The dimension map on invertible objects is represented by a map $\Sigma \SS[\B\TT] \to \SS$. This map is known classically by a different name.
        % \begin{proposition}[\cref{prop:dim-is-transfer}]
        %     The dimension of invertible objects is represented by the $\TT$-transfer map (see e.g. \cite{Boardman-1966-duality-and-thom} for classical cobordism transfers, see e.g. \cite{Klein-2001-dualizing-object,Cnossen-2023-twisted-ambi} for norm maps)
        %     \begin{equation*}
        %         \Tr \colon \Sigma \SS[\B\TT] \xto{\Nm} \SS^{\B\TT} \to \SS.
        %     \end{equation*}
        % \end{proposition}

    %%%%%%%%%%%%%%%%%%%%%%%%%%%%%%%%%%%%%%%%%%%%%%%%%%%%%%%%%%%%%%%%%%%%%%%%%%%%%%%%
    \subsection{Twisted graded categories and orientability}
    \label{subsec:intro-galois-closed}
    %%%%%%%%%%%%%%%%%%%%%%%%%%%%%%%%%%%%%%%%%%%%%%%%%%%%%%%%%%%%%%%%%%%%%%%%%%%%%%%%
        The Koszul braiding on $\Gr^{\Kos}_{\ZZ}\Vect$ originates from the non-trivial braiding on $\Gr^{\Kos}_{\ZZ/2}\Vect = \sVect$. 
        The category $\sVect$ of super vector spaces is known to form a Galois extension (in the sense of Rognes \cite{Rognes-2008-Galois}) with Galois-group $\B \ZZ/2$ (see e.g.\ \cite{Johson-Freyd-2017-sVect}).
        Moreover, replacing $\Vect$ with $\sVect$ trivializes all symmetric monoidal structures on $\Gr_{\ZZ}\sVect$. 
        
        We generalize this phenomenon, demonstrating that the nontriviality of the braiding emerges precisely from the fact that $\sVect$ is the cyclotomic closure of $\Vect$. To accomplish this, we leverage the notion of orientability for $\infty$-semiadditive categories introduced in \cite{BCSY-Fourier}. Specifically, we construct a canonical Galois extension analogous to the extension $\Vect \to \sVect$ for a broad class of $\infty$-semiadditive categories, and we interpret it in terms of orientability.

        %%%%%%%%%%%%%%%%%%%%%%%%%%%%%%%%%%%%%%%%%%%%%%%%%%%%%%%%%%%%%%%%%%%%%%%%%%%%%%%%
        \medskip\noindent
        \textbf{$\infty$-semiadditivity.}
        The notion of $\infty$-semiadditive categories, introduced by Hopkins and Lurie \cite{Hopkins-Lurie-2013-ambi}, is a generalization of ordinary semiadditivity, replacing finite sets by $\pi$-finite spaces. Namely, a category is said to be $\infty$-semiadditive if limits and colimits along $\pi$-finite spaces coincide. 

        An important family of examples is provided by monochromatic categories. Namely, by \cite{Hopkins-Lurie-2013-ambi, CSY-teleambi} for any $R \in \Alg(\SpTn)$, $\LMod_R(\SpTn)$ is $\infty$-semiadditive.
        In particular, $\SpTn$, $\SpKn$ and $\ModEn \coloneqq \Mod_{\En}(\SpTn)$ are $\infty$-semiadditive.
        Another important example comes from categorification: If $\cC \in \CAlg(\PrL)$, then $\Mod_{\cC}$ is $\infty$-semiadditive (\cite[Corollary~5.3]{BCSY-Fourier}).

        %%%%%%%%%%%%%%%%%%%%%%%%%%%%%%%%%%%%%%%%%%%%%%%%%%%%%%%%%%%%%%%%%%%%%%%%%%%%%%%%
        \medskip\noindent
        \textbf{Orientability and the Fourier transform.}
        In \cite{BCSY-Fourier}, Barthel, Carmeli, Schlank, and Yanovski introduced the notion of \emph{orientations} for $\infty$-semiadditive categories, a concept that serves, in a precise sense explained below, as an analogue of roots of unity in the context of Fourier transforms.
        
        \begin{definition}
            Let $\Inprime \coloneqq \tau_{\ge 0} \Sigma^{\chrHeight} I_{\QQ_p / \Zp}$ be the truncated and shifted $p$-typical Brown--Comenetz dual of the sphere.
        \end{definition}
        $\Inprime$ is a higher analog of the group $\Inprime[0] = \mu_{p^{\infty}}$ of $p$-typical roots of unity. Its $n$-th homotopy group corresponds to higher roots of unity as in \cite{CSY-cyclotomic}. In particular, it defines a notion of height $\chrHeight$ Pontryagin duality:

        \begin{definition}
            Let $M$ be a $p$-local connective spectrum. Its height $\chrHeight$ Pontryagin dual is defined to be $\tau_{\ge 0}\hom(M, \Inprime)$.
        \end{definition}

        Let $\cC$ be presentably symmetric monoidal and let $R \in \CAlg(\cnSp)$. An $(R,\chrHeight)$-pre-orientation of $\cC$ is a map 
        \begin{equation*}
            \omega \colon \tau_{\ge 0}\hom(R, \Inprime) \to \ounit_{\cC}\units.
        \end{equation*}
        A pre-orientation defines, for any $M \in \Mod_R^{[0,\chrHeight]\hyphen\mrm{fin}}$, a Fourier transform (\cite[\textsection~3.2]{BCSY-Fourier})
        \begin{equation*}
            \mscr{F}_{\omega} \colon \ounit_{\cC}[M] \to \ounit_{\cC}^{\Map(M,\Inprime)} \qin \CAlg(\cC).
        \end{equation*}
        $\omega$ is called an orientation if the associated Fourier transform is an isomorphism for all $M$. 

        A $(\SS_{(p)}, \chrHeight)$-orientation\footnote{Equivalently, a $(\tau_{\le \chrHeight} \SS_{(p)}, \chrHeight)$-orientation.} of $\cC$ is a primitive map $\Inprime \to \ounit_{\cC}\units$, which we think of as the unit of $\cC$ having all spherical roots of unity (or that $\ounit_{\cC}$ is spherically-cyclotomically-closed).
        
        \begin{example}[{\cite[Theorem~7.8]{BCSY-Fourier}}]
            $\ModEn$ is $(\SS_{(p)},\chrHeight)$-orientable.
        \end{example}

        Orientation also behaves well under categorification
        \begin{theorem}[{\cite[Corollary~5.16]{BCSY-Fourier}}]\label{intro-thm:categorification-of-orinetations}
            Assume $R$ is $\chrHeight$-truncated. Then $\cC$ is $(R,\chrHeight)$-orientable if and only if $\Mod_{\cC}$ is $(R, \chrHeight+1)$-orientable.
        \end{theorem}

        \begin{example}
            $\Mod_{\ModEn}$ is $(\tau_{\le \chrHeight}\SS_{(p)}, \chrHeight+1)$-orientable.
        \end{example}

        Returning to our original example, let $\field$ be a cyclotomically-closed field of characteristic 0. Then $\Vect$ is $(\SS_{(p)},0)$-orientable, or equivalently, $(\ZZ_{(p)}, 0)$-orientable, for every prime $p$, and therefore $\Mod_{\Vect}$ is $(\ZZ_{(p)}, 1)$-orientable. 
        Since $\pi_1 \SS = \ZZ/2$, this orientability automatically extends to $(\SS_{(p)}, 1)$-orientability for any $p \neq 2$.

        However, $\Mod_{\Vect}$ fails to be $(\SS_{(2)}, 1)$-orientable, reflecting the fact that $\Vect$ does not admit the degree 0 roots of unity at the prime 2, corresponding to $\ZZ/2 \simeq \pi_0 \mu^{(1)}_{\SS_{(2)}}$.
        These missing roots can be adjoined by passing to the category $\sVect$.
    
        %%%%%%%%%%%%%%%%%%%%%%%%%%%%%%%%%%%%%%%%%%%%%%%%%%%%%%%%%%%%%%%%%%%%%%%%%%%%%%%%
        \medskip\noindent
        \textbf{Galois extensions.}
        In \cite{Rognes-2008-Galois}, Rognes extended the notion of $G$-Galois extensions to the setting of arbitrary $\EE_1$-groups and presentably symmetric monoidal categories. This generalizes classical Galois theory, allowing also split extensions.
        
        \begin{example}[The trivial Galois extension]
            Assume $G$ is $\cC$-dualizable. Then $\ounit_{\cC}^G$ is a $G$-Galois extension of $\ounit_{\cC}$.
        \end{example}

        % In the $\Kn$-local category, the Galois theory is completely known:
        % \begin{theorem}[{\cite[Theorem~5]{Devinatz-Hopkins-2004-Morava-stabilizer},~\cite{Baker-Richter-2008-Galois-En},~\cite{Rognes-2008-Galois},~\cite[Theorem~10.9]{Mathew-2016-Galois}}]
        %     The map $\SSKn \to \En$ exhibits $\En$ as the Galois-closure of $\SSKn$ and its Galois group is the Morava stabilizer group $\extMorStb$.
        % \end{theorem}

        In our original example, $\sVect$ is a $\B\ZZ/2$-Galois extension of $\Vect$, and provides a non-discrete example. Moreover, it is the Galois closure (\cite{Deligne-2002-tensor-categories}, \cite{Johson-Freyd-2017-sVect}).

        %%%%%%%%%%%%%%%%%%%%%%%%%%%%%%%%%%%%%%%%%%%%%%%%%%%%%%%%%%%%%%%%%%%%%%%%%%%%%%%%
        \medskip\noindent
        \textbf{Oriented extension.}
        We now specialize to the case where $\cC$ is $(\SS_{(p)},\chrHeight)$-oriented.
        By \cref{intro-thm:categorification-of-orinetations}, $\Mod_{\cC}$ is $(\tau_{\le \chrHeight} \SS_{(p)}, \chrHeight+1)$-oriented, which is detected by the orientation map of $\cC$:
        \begin{equation*}
            \tau_{\ge 1} \Inprime[\chrHeight+1] = \Sigma \Inprime \to \Sigma \ounit_{\cC}\units \to \cC\units,
        \end{equation*}
        identifying the roots of unity of degrees $1$ through $\chrHeight+1$ in $\cC$.

        % This orientability is detected by the orientation map $\Inprime \to \ounit_{\cC}\units$, which encodes the presence of spherical roots of unity of heights $0$ through $\chrHeight$ in $\ounit_{\cC}$. 
        % It induces a $(\tau_{\le \chrHeight} \SS_{(p)}, \chrHeight+1)$-orientation on $\Mod_{\cC}$
        % \begin{equation*} 
        %     \tau_{\ge 1} \Inprime[\chrHeight+1] = \Sigma \Inprime \to \Sigma \ounit_{\cC}\units \to \cC\units, 
        % \end{equation*}
        % identifying the roots of unity of heights $1$ through $\chrHeight+1$ in $\cC$.
        
        We construct a Galois extension $\cC[\omega^{(0)}_{\SS_{(p)}}]$ of $\cC$, which adds all roots of unity of degree 0 (i.e.\ $\pi_0 \Inprime[\chrHeight+1]$). That is, it is a universal extension such that $\Mod_{\cC[\omega^{(0)}_{\SS_{(p)}}]}$ is $(\SS_{(p)}, \chrHeight+1)$-oriented.
       
        Let $\pin \coloneqq \pi_{\chrHeight+1}(\SS_{(p)})$ be the $p$-local $(\chrHeight+1)$-st stable stem, and $\pinD$ be its Pontryagin dual. Using the cofiber sequence arising from $\Inprime[\chrHeight+1] \to \pi_0(\Inprime[\chrHeight+1]) = \pinD$
        \begin{equation*}
            \Sigma \Inprime \to \Inprime[\chrHeight+1] \to \pinD \to \Sigma^2 \In,
        \end{equation*}
        we get a map
        \begin{equation*}
            \zeta = \zeta_n \colon \pinD \to \Sigma^2\Inprime \to \Sigma^2\ounit_{\cC}\units \to \Mod_{\cC}\units.
        \end{equation*}

        \begin{definition}
            Let $\cC$ be a $(\SS_{(p)},\chrHeight)$-oriented category. We define $\cC[\omega^{(0)}_{\SS_{(p)}}]$ as the homotopy twisted $\pinD$-graded category induced by the map
            \begin{equation*}
                \zeta\colon \pinD \to \Mod_{\cC}\units.
            \end{equation*}
        \end{definition}

        This category is $(\SS_{(p)},\chrHeight+1)$-oriented, and it is the universal $(\SS_{(p)},\chrHeight+1)$-oriented category over $\cC$, in the following sense:
        \begin{proposition}[\cref{prop:cyc-is-Galois}, \cref{lem:cyc_is_cyc_closed}, \cref{prop:C-cyclotoimcally-closed-iff-gradedpi-is-trivial}]\label{prop-intro:C-cyclotoimcally-closed-iff-gradedpi-is-trivial}
           $\cC[\omega^{(0)}_{\SS_{(p)}}]$ is a $\B^{\chrHeight+1}\pin$-Galois extension of $\cC$, and the following are equivalent:
           \begin{enumerate}
               \item The category $\Mod_\cC$ is $(\SS_{(p)},\chrHeight+1)$-orientable.
               \item $\cC[\omega^{(0)}_{\SS_{(p)}}]$ is a trivial Galois extension, i.e.\ $\cC[\omega^{(0)}_{\SS_{(p)}}] \simeq \cC^{\B^{\chrHeight+1} \pin}$.
               \item $\cC[\omega^{(0)}_{\SS_{(p)}}] \simeq \Gr_{\pinD} (\cC)$, i.e.\ the braiding of $\cC[\omega^{(0)}_{\SS_{(p)}}]$ is trivial.
           \end{enumerate}
        \end{proposition}

        \begin{example}
            $\sVect = \Vect[\omega^{(0)}_{\SS_{(2)}}]$. As its braiding is non-trivial, it gives an alternative proof of the non-orientability of $\Mod_{\Vect}$.
        \end{example}

        We study these categories in low  heights. Specifically, we show that for $\chrHeight \le 4$ and any $\pichar\in\pinD$, the $\Ck$-action on the dimension of $\ounit_{\cC}\shift{\pichar} \in \cC[\omega^{(0)}_{\SS_{(p)}}]$ is trivial for all $k$ (\cref{lem:trivial-T-action-0-1}, \cref{lem:trivial-T-action-2}). This allows us to compute the corresponding braiding characters.

        Moreover, in the special case when $\cC=\ModEn$, we can compute the corresponding braiding characters in all primes and heights:

        \begin{alphtheorem}[\cref{thm:chromatic-braiding-character}]\label{alphthm:chromatic-braiding-character}
            Let $\pichar\in \pinD$. Then
            \begin{enumerate}
                \item If $p=2$, $n \le 2$ and $\pichar$ is not 2-divisible, then the braiding character of $\En\shift{\pichar} \in \ModEn{}[\omega^{(0)}_{\SS_{(p)}}]$ is the braiding character of $\Sigma\En \in \ModEn$;
                \item Otherwise, the braiding character of $\En\shift{\pichar} \in \ModEn{}[\omega^{(0)}_{\SS_{(p)}}]$ is the braiding character of $\En \in \ModEn$.
            \end{enumerate}
        \end{alphtheorem}

        As a corollary, we show:
        \begin{corollary}[\cref{cor:chromatic-T-action-zeta-trivial}]\label{cor:intro-chromatic-T-action}
            Let $p$ be a prime and $\chrHeight \ge 1$. Let $V \in \ModEn{}[\omega^{(0)}_{\SS_{(p)}}]\dbl$. Then the $\Ck$-action on $\dim(V)$ is trivial for every $k$.
        \end{corollary}
        
        These results, taking into account \cref{prop:intro-graded-trivial-iff-picard}, are compatible with the orientability conjecture of Barthel, Carmeli, Schlank and Yanvoski \cite[Conjecture~7.10]{BCSY-Fourier}, which conjectures that $\Mod_{\ModEn}$ is $(\SS_{(p)},\chrHeight+1)$-orientable. In particular, assuming this conjecture, using \cref{prop-intro:C-cyclotoimcally-closed-iff-gradedpi-is-trivial} and \cite[Proposition~10.11]{Mathew-2016-Galois}, it is simple to extend \cref{cor:intro-chromatic-T-action} to show that the $\TT$-action on $\dim(V)$ is trivial for any $V \in \ModEn{}[\omega^{(0)}_{\SS_{(p)}}]\dbl$. Moreover, by \cref{prop:intro-graded-trivial-iff-picard}, \cref{alphthm:chromatic-braiding-character} extends significantly to show that the braiding (and not only the braiding character) agrees with the braiding of a Picard element.

        % In the process of proving these results, we relate Hopkins--Kuhn--Ravanel's and Stapleton's transchromatic character theory \cite{Hopkins-Kuhn-Ravanel-2000-HKR,Stapleton-2013-HKR} and categorical character theory \cite{Hoyois-Scherozke-Sibilla-2017-traces,Carmeli-Cnossen-Ramzi-Yanovski-2022-characters}, in specific cases, using the chromatic Nullstellensatz \cite{Burklund-Schlank-Yuan-2022-Nullstellensatz}.

    %%%%%%%%%%%%%%%%%%%%%%%%%%%%%%%%%%%%%%%%%%%%%%%%%%%%%%%%%%%%%%%%%%%%%%%%%%%%%%%%
    \subsection{Organization}
    %%%%%%%%%%%%%%%%%%%%%%%%%%%%%%%%%%%%%%%%%%%%%%%%%%%%%%%%%%%%%%%%%%%%%%%%%%%%%%%%
        In \cref{sec:twisted-graded-categories} we introduce (twisted) graded categories. We start in \cref{subsec:graded-categories-and-day} with the study of the Day convolution, mainly with the Picard spectrum of graded categories. In \cref{subsec:Thom-cateogries} we construct twisted graded categories using the Thom construction, and show it exhausts all symmetric monoidal structures for discrete group-like monoids. Finally, in \cref{subsec:examples} we present some examples of interest.

        In \cref{sec:braiding} we discuss generally the braiding and its character in a symmetric monoidal category. In \cref{subsec:braiding} we introduce the braiding functor and discuss the braiding of invertible objects. In \cref{subsec:equivariant-trace} we study the monoidal trace and the monoidal dimension of invertible objects, and in \cref{subsec:braiding-character} we introduce the braiding character of a dualizable object, and show it only depends on the monoidal dimension of the object, remembering the $\TT\supseteq \Ck$-action for all $k$.

        In \cref{sec:graded-braiding}, we study the braiding and the braiding character in the special case of twisted graded categories. In \cref{subsec:graded-braiding} we show that in this case the collection of braidings of $\ounit_{\cC}\shift{a}$, determine the symmetric monoidal structure, and give a condition for triviality of the symmetric monoidal structure, forgetting the $\EE_1$-isomorphism. In \cref{subsec:graded-braiding-character} we show that the braiding character can be understood internally in twisted graded categories and study it in certain family of examples. Finally, in \cref{subsec:exterior-algebras}, we study the free commutative algebras of elements in degree 1, in $\ZZ$-twisted graded categories. These give an analog of (graded) exterior and symmetric algebras. We also relate these to 1-dimensional representations of the symmetric group.

        In \cref{sec:Galois-closed-unit}, we study the braiding of twisted graded categories over base categories $\cC$ that are $(\SS_{(p)},\chrHeight)$-oriented. We begin in \cref{subsec:graded-categories-and-orientability} by relating this structure to the $(\SS_{(p)},\chrHeight+1)$-orientability of $\Mod_{\cC}$, which enables us to interpret $\cC[\omega^{(0)}_{\SS{(p)}}]$ as the universal $(\SS_{(p)},\chrHeight+1)$-oriented category over $\cC$. In \cref{subsec:low-heights}, we analyze the braiding character of objects of the form $\ounit_{\cC}\shift{\alpha} \in \cC[\omega^{(0)}_{\SS{(p)}}]$ for $\pichar \in \pinD$ in heights $\chrHeight \le 4$. Finally, in \cref{subsec:chromatic-braiding-character}, we study the corresponding braiding character in the case $\cC = \ModEn$ for arbitrary prime and height, and show that in this case it trivializes in the correct sense.
    
    %%%%%%%%%%%%%%%%%%%%%%%%%%%%%%%%%%%%%%%%%%%%%%%%%%%%%%%%%%%%%%%%%%%%%%%%%%%%%%%%
    \subsection{Conventions}
    %%%%%%%%%%%%%%%%%%%%%%%%%%%%%%%%%%%%%%%%%%%%%%%%%%%%%%%%%%%%%%%%%%%%%%%%%%%%%%%%

         We use the following terminology and notation:
        \begin{enumerate}
            \item The category of spaces (or animae, or groupoids) is denoted by $\spc$.
            \item The category of spectra is denoted $\Sp$ and the its full subcategory of connective spectra is denoted $\cnSp$.
            \item We denote by $\cC\core\subseteq \cC$ the maximal subgroupoid of a category $\cC$.
            \item We denote the space of morphisms between two objects $X,Y$ in a category $\cC$ by $\Map_{\cC}(X,Y)$ and omit $\cC$ when it is clear from context. If $\cC$ is stable we denote the mapping spectrum of $X,Y$ by $\hom_{\cC}(X,Y)$ or by $\hom(X,Y)$ if $\cC$ is clear from context.
            \item The category of presentable categories with colimit-preserving functors is denoted by $\PrL$. For $\cC \in \CAlg(\PrL)$ we denote its category of modules by $\Mod_{\cC} \coloneqq \Mod_{\cC}(\PrL)$.
            \item For a category $\cC$ we denote its homotopy $(k,1)$-category by $\h_k \cC$. We use the convention $\h\cC = \h_1 \cC$ for the homotopy $(1,1)$-category.
            \item For a category $\cC$, an object $X \in \cC$ and a space $A \in \spc$, we denote the constant limit and colimit of $X$ along $A$ (if they exist) by $X^A$ and $X[A]$ respectively.
            \item We denote the free commutative monoid (in $\spc$) by $\MM = (\Fin\core, \sqcup) = \bigsqcup_{\degree}\B\Sm$.
            \item For a symmetric monoidal category $\cC$ we denote its full subcategory spanned by dualizable objects by $\cC\dbl$ and the maximal subgroupoid by $\cC\dblspace$.
            \item For a symmetric monoidal category $\cC$ we denote by $\cC\units$ its Picard spectrum and by $\Pic(\cC)$ its $\pi_0$. To avoid confusion, we do not use the common notation $\Pic(R)$ to mean $\Pic(\Mod_R)$.
            \item We denote the loops functor of connective spectra by $\Omega \colon \cnSp \to \cnSp$ and distinguish it from the desuspension of (not necessarily connective) spectra which we denote $\Sigma^{-1} \colon \Sp \to \Sp$.
            \item We denote the connective free loops functor by $\L$. We use it both for the functor $\Map(\TT,-) \colon \spc \to \spc$ and for $\tau_{\ge 0}\hom(\SS[\TT], -) \colon \cnSp \to \cnSp$.
            \item We denote the non-connective free loops functor of spectra by $\Lnc \coloneqq \hom(\SS[\TT], -) \colon \Sp \to \Sp$. For a connective spectrum $X$,  $\L X = \tau_{\ge 0} \Lnc X$.
            \item For a height $\chrHeight$ and a prime $p$ we denote by $\Kn$, $\Tn$ the corresponding Morava $K$-theory and any telescope of height $\chrHeight$. For a formal group $\mbb{G}$ of height $\chrHeight$ over $\Fpbar$ and an algebraically-closed field $L$ we denote by $\En(L) = \En(L, \mbb{G})$ the corresponding Morava $E$-theory. When $L$ does not play an important role we omit it from the notation.
            \item Inspired by \cite{Burklund-Schlank-Yuan-2022-Nullstellensatz}, we call maps to Nullstellensatzian objects \quotes{geometric points}.
            \item We denote by $C^{\chrHeight}_{\chrHeight-t}(L)$ the $\Kn[\chrHeight - t]$-localization of the splitting algebra of the $p$-divisible group on $L_{K(n-t)}\En(L)$, as constructed in \cite{Stapleton-2013-HKR, Lurie-2019-Elliptic3}. For any $\pi$-finite $p$-local space $A$, we denote the corresponding transchromatic character by
            \begin{equation*}
                \chi^{t,\HKR}_{(-)} \colon \En(L)^A \to C^{\chrHeight}_{\chrHeight-t}(L)^{\L^t A}.
            \end{equation*}

        \end{enumerate}

    %%%%%%%%%%%%%%%%%%%%%%%%%%%%%%%%%%%%%%%%%%%%%%%%%%%%%%%%%%%%%%%%%%%%%%%%%%%%%%%%
    \subsection{Acknowledgements}
    %%%%%%%%%%%%%%%%%%%%%%%%%%%%%%%%%%%%%%%%%%%%%%%%%%%%%%%%%%%%%%%%%%%%%%%%%%%%%%%%
        We thank Nathaniel Stapleton and our advisor Tomer Schlank for inspiring discussions that led to this project. We are especially grateful to Tomer Schlank for his patience, valuable insights, and many helpful conversations. We extend special thanks to Millie Rose for her insightful discussions and assistance in the early stages of the project. We also thank the entire Seminarak group, particularly Shay Ben-Moshe and Lior Yanovski for many helpful discussions and comments on previous drafts. We thank Achim Krause and Maxime Ramzi for many helpful comments on a previous version.
          
        The first author acknowledges the hospitality of the University of Chicago, where the project was primarily written. The second author expresses gratitude to the Hebrew University, where most of the project was developed.

%%%%%%%%%%%%%%%%%%%%%%%%%%%%%%%%%%%%%%%%%%%%%%%%%%%%%%%%%%%%%%%%%%%%%%%%%%%%%%%%
%%%%%%%%%%%%%%%%%%%%%%%%%%%%%%%%%%%%%%%%%%%%%%%%%%%%%%%%%%%%%%%%%%%%%%%%%%%%%%%%
\section{Twisted graded categories}
\label{sec:twisted-graded-categories}
%%%%%%%%%%%%%%%%%%%%%%%%%%%%%%%%%%%%%%%%%%%%%%%%%%%%%%%%%%%%%%%%%%%%%%%%%%%%%%%%
%%%%%%%%%%%%%%%%%%%%%%%%%%%%%%%%%%%%%%%%%%%%%%%%%%%%%%%%%%%%%%%%%%%%%%%%%%%%%%%%
    Let $\cC$ be a presentably symmetric monoidal category and $M$ a commutative monoid. In this section, we study different symmetric monoidal structures on the category $\Fun(M, \cC)$ of $M$-graded $\cC$-objects. The simplest example we consider is the Day convolution, which corresponds to the constant colimit $\cC[M]$ in $\CAlg_{\cC}(\PrL)$. 
    In \cref{subsec:graded-categories-and-day}, we study this symmetric monoidal structure and compute its Picard spectrum when $M$ is discrete and $\cC$ is semiadditive and connected. 
    
    In \cref{subsec:Thom-cateogries}, we employ the Thom construction to construct symmetric monoidal structures on $\Fun(M,\cC)$ that agree $\EE_1$-monoidally with the Day convolution. When $M = A$ is an abelian group, we prove that these exhaust all possible structures.
        
    Finally, in \cref{subsec:examples}, we explore examples of twisted graded categories, reconstructing in particular the Koszul twist in any category admitting a minus one.

    %%%%%%%%%%%%%%%%%%%%%%%%%%%%%%%%%%%%%%%%%%%%%%%%%%%%%%%%%%%%%%%%%%%%%%%%%%%%%%%%
    \subsection{Graded categories and Day convolution}
    \label{subsec:graded-categories-and-day}
    %%%%%%%%%%%%%%%%%%%%%%%%%%%%%%%%%%%%%%%%%%%%%%%%%%%%%%%%%%%%%%%%%%%%%%%%%%%%%%%%
        Let $\cC$ be an $\EE_k$-monoidal category and $M$ be an $\EE_k$-monoid for $1 \le k \le \infty$. In this subsection, we study the Picard $\EE_k$-group of the $\EE_k$-monoidal category $\Fun(M, \cC)$ equipped with the Day monoidal structure, in the case when $\cC$ is presentable and $M$ is discrete. 
        Under mild assumptions, we show it is isomorphic to $\cC\units \times M\units$.

        \begin{definition}
            Define the category of $M$-graded objects in $\cC$ as $\Gr_M \cC \coloneqq \Fun(M, \cC)$. We endow this category with the Day $\EE_k$-monoidal structure.

            Equivalently, $\Gr_M \cC$ is the Thom construction of the trivial map $M \xto{0} \Mod_{\cC}\units$ (see \cite{Antolin-Barthel-2019-Thom} or \cref{subsec:Thom-cateogries}).
        \end{definition}

        \begin{definition}\label{def:shift-and-yoneda}\ 
            \begin{enumerate}
                \item\label{item:<0>} 
                The map of commutative monoids $0 \to M$ over $\Mod_{\cC}\units$ induces an $\EE_k$-monoidal functor
                \begin{equation*}
                    (-)\shift{0} \colon \cC \to \Gr_M\cC.
                \end{equation*}
                It sends $X\in \cC$ to the graded object $X\shift{0} \colon M \to \cC$, which is $X[\Omega_0 M]$ at the connected component of $0 \in M$, and is $\emptyset$ at all other connected components.

                \item\label{item:1<->}
                The Yoneda embedding $\yo \colon M\op \to \FunDay(M, \cC) = \Gr_M \cC$ is $\EE_k$-monoidal (\cite[Corollary~4.8.1.12, Remark~4.8.1.13]{Lurie-HA}, \cite[Proposition~3.3]{Ben-Moshe-Schlank-2024-K-theory}). Denote the $\EE_k$-monoidal functor
                \begin{equation*}
                    \ounit_{\cC}\shift{-} \colon M \isoto M\op \xto{\yo} \Gr_M \cC.
                \end{equation*}
                It sends $m\in M$ to the graded object $\ounit_{\cC}\shift{m} \colon M \to \cC$, which is $\ounit_{\cC}[\Omega_m M]$ at the connected component of $m\in M$, and $\emptyset$ at all other connected components.

                \item\label{item:(-)<->}
                Taking the tensor product of the above two definitions, we define the $\EE_k$-monoidal functor
                \begin{equation*}
                    (-)\shift{-} \colon \cC \times M \xto{(-)\shift{0} \times \ounit_{\cC}\shift{-}} \Gr_M \cC \times \Gr_M \cC \xto{\otimes_{\Day}} \Gr_M \cC.
                \end{equation*}
                It sends $(X, m)$ to the graded object $X\shift{m} \colon M \to \cC$ which is $X[\Omega_m M]$ at the connected component of $m\in M$, and $\emptyset$ at all other connected components.
            \end{enumerate}
        \end{definition}
        
        We detour now for a quick discussion about the Picard spectrum, and more generally, the Picard $\EE_k$-group.        
        Let $\cU\in\Alg_{\EE_{k+1}}(\PrL)$. Then $\cU$ admits a unique $\EE_{k+1}$-monoidal, colimit preserving functor
        \begin{equation*}
            \ounit_{\cU}[-] \colon \spc \to \cU
        \end{equation*}
        and its right adjoint is $\Map(\ounit_{\cU}, -)$. Taking $\EE_k$-algebras we get an adjunction
        \begin{equation*}
            \ounit_{\cU}[-] \colon \Mon_{\EE_k}(\spc) \rightleftarrows \Alg_{\EE_k}(\cU) \cocolon \Map(\ounit_{\cU}, -).
        \end{equation*}
        The forgetful functor from group-like $\EE_k$-monoids admits a right adjoint $(-)\units$
        \begin{equation*}
            \fgt \colon \Grp_{\EE_k}(\spc) \rightleftarrows \Mon_{\EE_k}(\spc) \cocolon (-)\units
        \end{equation*}
        which takes a monoid to its collection of connected components spanned by invertible elements. Composing the two right adjoints we get a functor
        \begin{equation*}
            (-)\units \colon \Alg_{\EE_k}(\cU) \to \Grp_{\EE_k}(\spc).
        \end{equation*}
    
        When $\cU$ is symmetric monoidal we have a commutative diagram
        \begin{equation*}
            \begin{tikzcd}
                {\CAlg(\cU)} & {\Alg(\cU)} \\
                \cnSp & {\Grp(\spc).}
                \arrow["\fgt", from=1-1, to=1-2]
                \arrow["{(-)\units}", from=1-1, to=2-1]
                \arrow["{(-)\units}", from=1-2, to=2-2]
                \arrow["\fgt", from=2-1, to=2-2]
            \end{tikzcd}
        \end{equation*}
    
        Choosing $\cU = \Cat$, for any $\cC\in \Mon_{\EE_k}(\Cat)$ we call $\cC\units$ the Picard $\EE_k$group of $\cC$.
        The commutativity of the diagram above implies that when $\cC$ is symmetric monoidal, its Picard $\EE_1$-group agrees with the underlying $\EE_1$ structure of its Picard spectrum.
    
        We also denote $\Pic(\cC) \coloneqq \pi_0 \cC\units$.

        We recall the following definition:
        \begin{definition}\label{def:connected}
            Let $\cC$ be an $\EE_k$-monoidal category. We say that $\cC$ is connected if it can not be decomposed as a product of $\EE_k$-monoidal categories. Equivalently, the unit $\ounit_{\cC}$ does not admit central idempotents (\cite[Proposition~6.28]{BCSY-Fourier}).   
        \end{definition}
    
        The goal of this section is the following:
    
        \begin{theorem}\label{prop:picard-of-Day}
            Let $\cC \in \Alg_{\EE_k}(\PrL)$ be semiadditive and connected.
            Let $M$ be a discrete commutative monoid. Then the functor of \cref{def:shift-and-yoneda}(\labelcref{item:(-)<->}) induces an isomorphism of $\EE_k$-groups
            \begin{equation*}
               \cC\units \times M\units \isoto (\Gr_M \cC)\units.
            \end{equation*}
        \end{theorem}
    
        We will need the following two simple lemmas:
    
        \begin{lemma}\label{lem:sum-is-zero}
            Let $\cC$ be a category. Assume there is a collection of objects $\{X_i\}_i$ such that $\bigsqcup_{i} X_i \simeq \emptyset$. Then $X_i \simeq \emptyset$ for all $i$.
        \end{lemma}
        \begin{proof}
            Let $Y \in \cC$. For any $i$ we call the following map the trivial map
            \begin{equation*}
                \emptyset \colon X_i \to \bigsqcup_j X_j \simeq \emptyset \to Y.
            \end{equation*}
            Choose an index $i_0$. The trivial maps $X_i \xto{\emptyset} Y$ for $i\neq i_0$, provide a retraction
            \begin{equation*}
                \Map(X_{i_0}, Y) \to \Map(\bigsqcup_j X_j, Y)
            \end{equation*}
            to the post-compisition map with $X_{i_0} \to \bigsqcup_j X_j$.
            Therefore $\Map(X_{i_0}, Y)$ is a retract of $\Map(\bigsqcup_j X_j, Y) \simeq \Map(\emptyset, Y) \simeq \pt$ and therefore contractible. Thus $X_{i_0} \simeq \emptyset$.
        \end{proof}
    
        \begin{lemma}\label{lem:sum-cannot-be-unit}
            Let $0 \neq \cC\in \Alg_{\EE_k}(\PrL)$ be semiadditive and connected. Let $E,F\in \cC$ be orthogonal, that is $E \otimes F \simeq 0$. Assume that $E \oplus F \simeq \ounit_{\cC}$. Then either $E = \ounit_{\cC}$ and $F = 0$ or $E = 0$ and $F = \ounit_{\cC}$.
        \end{lemma}
        \begin{proof}
            Consider the composition $u \colon \ounit \simeq E \oplus F \xto{\pr_E} E$. Then the map $\id_E \otimes u \colon E \to E\otimes E$ factors as
            \begin{equation*}
                \id_E \otimes u \colon E \simeq E \otimes (E \oplus F) \simeq (E \otimes E) \oplus (E\otimes F) \simeq (E\otimes E).
            \end{equation*}
            That is, $E$ is an idempotent algebra in $\cC$, in contradiction to the connectedness of $\cC$.

        \end{proof}

        \begin{proof}[Proof of \cref{prop:picard-of-Day}]
            Consider the $\EE_k$-monoidal map
            \begin{equation*}
                (-)\shift{-} \colon \cC \times M \to \Gr_M \cC
            \end{equation*}
            of \cref{def:shift-and-yoneda}(\labelcref{item:(-)<->}). It is fully faithful, as for each $m\in M$ the map $(-)\shift{m} \colon \cC \to \Gr_M \cC$ is fully faithful. Taking Picard $\EE_k$-groups, it induces
            \begin{equation*}
                \cC\units \times M\units \to (\Gr_M \cC)\units.
            \end{equation*}
    
            It is left showing it is essentially surjective.
            Let $X\in (\Gr_M \cC)\units$ and call its inverse $Y$. As $(-)\shift{0} \colon \cC \to \Gr_M \cC$ is $\EE_k$-monoidal, the unit of $\Gr_M \cC$ is $\ounit_{\cC}\shift{0}$. That is, it is $\ounit_{\cC}$ at $0\in M$ and $0\in \cC$ everywhere else. The Day convolution is given by
            \begin{equation*}
                (X \otimes_{\Day} Y)_m \simeq \bigoplus_{a+b = m} (X_a \otimes Y_b).
            \end{equation*}
            Therefore,
            \begin{equation*}
                \begin{split}
                    & \bigoplus_{a + b = 0} (X_a \otimes Y_b) \simeq \ounit_{\cC} \\
                    & \bigoplus_{a + b = m} (X_a \otimes Y_b) \simeq 0 \text{ for any $m \neq 0$}.
                \end{split}
            \end{equation*}
            By \cref{lem:sum-cannot-be-unit}, there exists a unique $m_0 \in M$ (necessarily invertible) for which $X_{m_0} \otimes Y_{-m_0} \simeq \ounit_{\cC}$, and together with \cref{lem:sum-is-zero}, for any $(a,b)\neq (m_0, -m_0)$ 
            \begin{equation*}
                X_a \otimes Y_b \simeq 0.
            \end{equation*}
            In particular, for any $a \neq m_0$, $X_a \otimes Y_{-m_0} \simeq 0$, but as $Y_{-m_0}$ is invertible, $X_a \simeq 0$. Therefore $X \simeq X_{m_0}\shift{m_0}$.
        \end{proof}
    
        \begin{remark}
            \cref{prop:picard-of-Day} does not hold if we replace $M$ by a non-discrete $\EE_k$-monoid. For example, if we take a connective spectrum of the form $\Sigma^2 A$ for $A\in \Ab$, and $\cC$ a strict 1-category, then any functor $\Sigma^2 A \to \cC$ is trivial and $\FunDay(\Sigma^2 A, \cC) \simeq \cC$. Therefore
            \begin{equation*}
                (\Gr_{\Sigma^2 A} \cC)\units = \cC\units \neq \cC\units \times \Sigma^2 A.
            \end{equation*}
        \end{remark}

    %%%%%%%%%%%%%%%%%%%%%%%%%%%%%%%%%%%%%%%%%%%%%%%%%%%%%%%%%%%%%%%%%%%%%%%%%%%%%%%%
    %%%%%%%%%%%%%%%%%%%%%%%%%%%%%%%%%%%%%%%%%%%%%%%%%%%%%%%%%%%%%%%%%%%%%%%%%%%%%%%%
    \subsection{Thom categories}
    \label{subsec:Thom-cateogries}
    %%%%%%%%%%%%%%%%%%%%%%%%%%%%%%%%%%%%%%%%%%%%%%%%%%%%%%%%%%%%%%%%%%%%%%%%%%%%%%%%
    %%%%%%%%%%%%%%%%%%%%%%%%%%%%%%%%%%%%%%%%%%%%%%%%%%%%%%%%%%%%%%%%%%%%%%%%%%%%%%%%
        
        \begin{definition}\label{def:twisted-Fun}
            Let $\cC \in \Alg_{\EE_{k}}(\PrL)$ for $1 \le k \le \infty$. Let $M$ be a discrete commutative monoid and $\Zchar \colon M \to \Sigma \cC\units$ be an $\EE_k$-monoidal map, equipped with an $\EE_1$-nullhomotopy. 
            Define the category $\Gr^{\Zchar}_M \cC$ of \emph{$\Zchar$-twisted $M$-graded objects} to be the Thom construction (i.e.\ the colimit in $\Mod_{\cC}$) of the map
            \begin{equation*}
                M \xto{\Zchar} \Sigma \cC\units \to \Mod_{\cC}\units.
            \end{equation*}
            It is an $\EE_k$-algebra in $\Mod_{\cC}$ (see e.g. \cite{Antolin-Barthel-2019-Thom}).
        \end{definition}

        \begin{lemma}\label{lem:twFun-Day-E1}
            The underlying $\EE_1$-algebra of $\Gr^{\Zchar}_M(\cC)$ is identified with $\Gr_M \cC \in \Alg_{\cC}(\PrL)$ via the $\EE_1$-nullhomotopy of $M\to \Sigma \cC\units$.
        \end{lemma}
        
        \begin{proof}
            The map of $\EE_1$-groups
            \begin{equation*}
                \Zchar \colon M \to \Sigma \cC\units
            \end{equation*}
            comes with a null homotopy by assumption. Therefore, the colimit along $\Zchar$ is identified with the trivial colimit $\cC[M] \simeq \Gr_M \cC$.
        \end{proof}

        \begin{corollary}
            Let $m_1, m_2 \in M$. Then for any  \begin{equation*}
                \Zchar \colon M \to \Sigma \cC\units
            \end{equation*}
            with an $\EE_1$-nullhomotopy we have $\ounit_{\cC}\shift{m_1} \otimes \ounit_{\cC}\shift{m_2} \simeq \ounit_{\cC}\shift{m_1 + m_2}$ in $\Gr_M^\phi(\cC)$. 
        \end{corollary}
    
        \begin{definition}
            Using the identification of \cref{lem:twFun-Day-E1}, we can define the $\EE_1$-map as in \cref{def:shift-and-yoneda}(\labelcref{item:(-)<->})
            \begin{equation*}
                (-)\shift{-} \colon \cC \times M \to \Gr^{\Zchar}_M(\cC).
            \end{equation*}
        \end{definition}
    
        We will be interested specifically in the case where $k = \infty$ and $M$ is grouplike. In this case $\EE_k$-groups are identified with connective spectra.
        
        \begin{definition}
            Let $X,Y \in \cnSp$. Define the space $\Map\Enull(X,Y)$ of maps $X \to Y$ that are $\EE_1$-null as the fiber
            \begin{equation*}
                \Map\Enull(X,Y) \coloneqq \fib \big( \Map_{\cnSp}(X,Y) \to \Map_{\EE_1}(X,Y) \big).
            \end{equation*}
        \end{definition}

        Using the fully faithful functor $\B \colon \Grp(\spc) \into \spc_*$ and the $\redSS[-] \dashv \fgt$ adjunction, where $\redSS[-]\colon \spc_* \to \cnSp$ is the reduced suspension spectrum functor, one can write
        \begin{equation*}
            \Map_{\EE_1}(X, Y) \simeq \Map_{\spc_*}(\B X, \B Y) \simeq \Map_{\cnSp}(\redSS[\B X], \Sigma Y) \simeq \Map_{\Sp}(\Sigma^{-1} \redSS[\B X], Y).
        \end{equation*}
        Therefore
        \begin{equation*}
            \Map\Enull(X,Y) \simeq \Map_{\Sp}(\cofib(\Sigma^{-1} \redSS[\B X] \to X), Y).
        \end{equation*}
        We moreover notice that the cofiber $\cofib(\Sigma^{-1} \redSS[\B X] \to X)$ is connected.

        \begin{definition}
            Let $X \in \cnSp$. Define
            \begin{equation*}
                X\Enull \coloneqq \Omega \cofib(\Sigma^{-1} \redSS[\B X] \to X) \in \cnSp.
            \end{equation*}
        \end{definition}

        \begin{corollary}\label{cor:E1-null}
            For any $X,Y \in \cnSp$
            \begin{equation*}
                \Map(X\Enull, \Omega Y) \simeq \Map\Enull(X,Y).
            \end{equation*}
        \end{corollary}
            
        \begin{definition}
            Let $\cC \in \CAlg(\PrL)$ and $\cD \in \Alg_{\cC}(\PrL)$. Define the space $\Lift(\cD)$, of $\EE_{\infty}$-lifts of $\cD$, as the pullback
            \begin{equation*}
                \begin{tikzcd}
                    {\Lift(\cD)} & {{\CAlg_{\cC}(\PrL)}\core} \\
                    {\{\cD\}} & {{\Alg_{\cC}(\PrL)}\core.}
                    \arrow[from=1-1, to=1-2]
                    \arrow[from=1-1, to=2-1]
                    \arrow["\lrcorner"{anchor=center, pos=0.125}, draw=none, from=1-1, to=2-2]
                    \arrow[from=1-2, to=2-2]
                    \arrow[from=2-1, to=2-2]
                \end{tikzcd}
            \end{equation*}
        \end{definition}

        \begin{theorem}\label{thm:twFun-are-all-lifts}
            Let $A \in \Ab$ and $\cC\in \CAlg(\PrL)$ be semiadditive. The map 
            \begin{equation*}
                \Gr^{(-)}_A \cC \colon \Map(A\Enull, \cC\units) \simeq \Map\Enull(A, \Sigma \cC\units) \to \Lift(\Gr_A \cC)
            \end{equation*}
            is an isomorphism of spaces.
        \end{theorem}
    
        Before proving this theorem, we will need the following result from \cite{Antolin-Barthel-2019-Thom}:
        \begin{proposition}[\cite{Antolin-Barthel-2019-Thom}]\label{prop:universal-property-thom}
            Let $\cU \in \CAlg(\PrL)$ and $f\colon G \to \cU\units$ be a map of $\EE_k$-groups. Then for any $R\in \Alg_{\EE_{k+1}}(\cU)$ there is an isomorphism between the space $\Map_{\Alg_{\EE_{k}}(\cU)}(\Th(f), R)$ of algebra maps from the Thom construction of $f$ to $R$ and the space of nullhomotopies of the composition $G \to \cU\units \to \LMod_R(\cU)\units$ in $\EE_k$-groups.
        \end{proposition}
        
        \begin{proof}
            The proofs in \cite{Antolin-Barthel-2019-Thom} are done for $\cU$ being a category of modules over a commutative ring spectrum, but they hold generally. Now following the proof of \cite[Proposition~8.11]{Burklund-Schlank-Yuan-2022-Nullstellensatz}: By \cite[Theorem~3.5, Definition~3.12, Definition~3.14, Proposition~3.15]{Antolin-Barthel-2019-Thom} the space of maps of $\EE_k$-algebras $\Th(f) \to R$ is isomorphic to the space of lifts
            \begin{equation*}
                \begin{tikzcd}
                    & {B(R;\cU)} \\
                    G & {\cU\units,}
                    \arrow[from=1-2, to=2-2]
                    \arrow[dashed, from=2-1, to=1-2]
                    \arrow["f", from=2-1, to=2-2]
                \end{tikzcd}
            \end{equation*}
            where $B(R;\cU)$ is the space of maps $Z \to R$ where $Z\in \cU\units$ and when induced to $R$ it is an isomorphism $R\otimes Z \isoto R$. In the proof of \cite[Proposition~3.16]{Antolin-Barthel-2019-Thom} it is written as the pullback
            \begin{equation*}
                \begin{tikzcd}
                    {B(R; \cU)} & {B(\ounit_{\cU}; \cU)} \\
                    {\cU\units} & {\LMod_R(\cU)\units,}
                    \arrow[from=1-1, to=1-2]
                    \arrow[from=1-1, to=2-1]
                    \arrow["\lrcorner"{anchor=center, pos=0.125}, draw=none, from=1-1, to=2-2]
                    \arrow[from=1-2, to=2-2]
                    \arrow["{R\otimes-}", from=2-1, to=2-2]
                \end{tikzcd}
            \end{equation*}
            i.e.\ lifts to $B(R;\cU)$ are in correspondence with nullhomotopies of $G \to \cU\units \to \LMod_R(\cU)\units$.
        \end{proof}
    
        % \begin{corollary}[{\cite[Corollary~3.18]{Antolin-Barthel-2019-Thom}}]\label{cor:maps-from-Thom}
        %     Let $\cU \in \CAlg(\PrL)$ and $f\colon G \to \cU\units$ be a map of $\EE_k$-groups. Then for any $R \in \Alg_{\EE_k}(\cC)$ the space of $\EE_k$-algebras from $\Th(f)$ to $R$ is either empty or
        %     \begin{equation*}
        %         \Map_{\Alg_{\EE_k}(\cC)}(\Th(f), R) \simeq \Map_{\Alg_{\EE_k}}(\ounit_{\cC}[G], R).
        %     \end{equation*}
        % \end{corollary}

        \begin{definition}
            Let $\cD, \cE \in \CAlg_{\cC}(\PrL)$. A functor $F \colon \cD \to \cE$ is called homotopy multiplicative if for any $X, Y \in \cD$ there is an isomorphism
            \begin{equation*}
                F(X) \otimes F(Y) \simeq F(X \otimes Y).
            \end{equation*}
            In particular, $F$ defines a map of groups
            \begin{equation*}
                F \colon \Pic(\cD) \to \Pic(\cE).
            \end{equation*}
        \end{definition}

        \begin{lemma}\label{lem:cofiber-sequence-of-picard-general}
            Assume that $\cC$ is connected. Let $\cD \in \CAlg_{\cC}(\PrL)$ with a homotopy multiplicative functor $\cD \to \Gr_A \cC$ which is an isomorphism on underlying categories. Then there exists a cofiber sequence of connective spectra
            \begin{equation*}
                \cC\units \to \cD\units \to A.
            \end{equation*}
        \end{lemma}

        \begin{proof}
            $\cD$ admits a unique $\cC$-linear symmetric monoidal functor $\cC \to \cD$ which induces a map of connective spectra $\cC\units \to \cD\units$. Denote its cofiber by $X$. 
            Since $\cC\units$ is connective, the cofiber sequence 
            \begin{equation*}
                \cC\units \to \cD\units \to X
            \end{equation*}
            is also a fiber sequence. 

            As the isomorphism $\cD \isoto \Gr_A \cC$ preserves the unit, the map $\cC\units \to \cD\units$ on connective covers identifies with the identity map $\Sigma \ounit_{\cC}\units \to \Sigma \ounit_{\cC}\units$. Therefore, by the long exact sequence of homotopy groups, $X$ is discrete and is equal to the cofiber of the map $\Pic(\cC) \to \Pic(\cD)$.

            By our assumption, $\Pic(\cD) \isoto \Pic(\Gr_A \cC)$ and therefore, by \cref{prop:picard-of-Day}, $\Pic(\cD) \cong \Pic(\cC) \times A$, and the map $\Pic(\cC) \to \Pic(\cD)$ is the inclusion. In particular, $X \simeq A$.
        \end{proof}
    
        \begin{corollary}\label{lem:cofiber-of-pic-Day}
            Let $A \in \Ab$ and $\cD\in \Lift(\Gr_A \cC)$ and assume that $\cC$ is connected. Then there is a cofiber sequence of connective spectra
            \begin{equation*}
                \cC\units \to \cD\units \to A.
            \end{equation*}
        \end{corollary}

        \begin{proof}[Proof of \cref{thm:twFun-are-all-lifts}]
            As both $\Map\Enull(A,-)$ and $\Lift[-](\Gr_A (-))$ send products to products, we may assume $\cC$ is connected.
            
            First we show that $\Gr^{(-)}_A \cC$ is an an isomorphism on $\pi_0$. Let \linebreak $\cD\in\Lift(\Gr_A \cC)$. By \cref{lem:cofiber-of-pic-Day}, $A$ sits in a cofiber sequence
            \begin{equation*}
                \cC\units \to \cD\units \to A.
            \end{equation*}
            Rolling the cofiber sequence, we get a natural map $\Zchar_{\cD} \colon A \to \Sigma \cC\units$. 
            The canonical nullhomotopy
            \begin{equation*}
                A \xto{\Zchar_{\cD}} \Sigma \cC\units \to \Sigma \cD\units
            \end{equation*}
            induces a null homotopy
            \begin{equation*}
                A \xto{\Zchar_{\cD}} \Sigma \cC\units \to \Mod_{\cC}\units \to \Mod_{\cD}\units.
            \end{equation*}
            By \cref{prop:universal-property-thom} this induces a map of $\cC$-linear symmetric monoidal categories $\Gr^{\Zchar_{\cD}}_A \cC \to \cD$, which is an $\EE_1$-isomorphism by assumption. Therefore it is an isomorphism in $\CAlg_{\cC}(\PrL)$. This proves that $\Gr^{(-)}_A \cC$ is surjective on $\pi_0$. To show it is injective, it is enough to verify that for any $\Zchar \in \Map\Enull(A, \Sigma \cC\units)$, the fiber of $\Zchar \colon A \to \Sigma \cC\units$ is $(\Gr^{\Zchar}_A \cC)\units$. Denote by $F_{\Zchar}$ the fiber of $\Zchar$. 
            Then there exists a symmetric monoidal map
            \begin{equation*}
                j_{\Zchar} \colon F_{\Zchar} \simeq F_{\Zchar}\op \xto{\yo} \FunDay(F_{\Zchar}, \cC) \to \Gr^{\Zchar}_A \cC,
            \end{equation*}
            where $\yo$ is the Yoneda embedding, which is symmetric monoidal by \cite[Theorem~3.2]{Glasman-2016-Day} or \cite[Corollary~4.8.1.12, Remark~4.8.1.13]{Lurie-HA}, and the last map is the one induced on Thom constructions from the commutative diagram
            \begin{equation*}
                \begin{tikzcd}
                    {F_{\Zchar}} & A \\
                    & {\Sigma\cC\units.}
                    \arrow[from=1-1, to=1-2]
                    \arrow["0"', from=1-1, to=2-2]
                    \arrow["\Zchar", from=1-2, to=2-2]
                \end{tikzcd}
            \end{equation*}
            In particular we get a map of connective spectra $F_{\Zchar} \to (\Gr^{\Zchar}_A \cC)\units$. Moreover, this map is an isomorphism as $\EE_1$-groups, as $\Zchar$ is trivial as an $\EE_1$-map. Thus it is an isomorphism of connective spectra, as needed.

            It is left now to show, that for any $\Zchar \in \Map\Enull(A, \Sigma \cC\units)$, the induced map
            \begin{equation*}\label{eqn:map-of-aut}
                \Aut_{\Map\Enull}(\Zchar) \to \Aut_{\Lift(\Gr_A \cC)}(\Gr^{\Zchar}_A \cC) \tag{$\star$}
            \end{equation*}
            is an isomorphism.
            $\Aut_{\Map\Enull}(\Zchar)$ is computed as the fiber
            \begin{equation*}
                \begin{split}
                    \Aut_{\Map\Enull}(\Zchar) 
                    & = \fib \big( \Omega_{\Zchar}\Map_{\cnSp}(A, \Sigma \cC\units)  \to \Omega_{\Zchar}\Map_{\EE_1}(A, \Sigma\cC\units) \big) \\
                    & \simeq \fib \big( \Map_{\cnSp}(A, \cC\units) \to \Map_{\EE_1}(A, \cC\units) \big) \\
                    & = \Map\Enull(A, \cC\units).
                \end{split}
            \end{equation*}
            On the other hand, $\Aut_{\Lift(\Gr_A \cC)}(\Gr^{\Zchar}_A \cC)$ is computed as the fiber
            \begin{equation*}
                \Aut_{\Lift(\Gr_A \cC)}(\Gr^{\Zchar}_A \cC) = \fib \big( \Aut_{\EE_{\infty}}(\Gr^{\Zchar}_A \cC) \to \Aut_{\EE_1}(\Gr^{\Zchar}_A \cC) \big).
            \end{equation*}
            The pullback square
            \begin{equation*}
                \begin{tikzcd}
                    {\Aut_{\EE_{\infty}}(\Gr^{\Zchar}_A \cC)} & {\Aut_{\EE_1}(\Gr^{\Zchar}_A \cC)} \\
                    {\End_{\EE_\infty}(\Gr^{\Zchar}_A \cC)} & {\End_{\EE_1}(\Gr^{\Zchar}_A \cC)}
                    \arrow[from=1-1, to=1-2]
                    \arrow[from=1-1, to=2-1]
                    \arrow["\lrcorner"{anchor=center, pos=0.125, rotate=0}, draw=none, from=1-1, to=2-2]
                    \arrow[from=1-2, to=2-2]
                    \arrow[from=2-1, to=2-2]
                \end{tikzcd}
            \end{equation*}
            implies that
            \begin{equation*}
                \Aut_{\Lift(\Gr_A \cC)}(\Gr^{\Zchar}_A \cC) = \fib \big( \End_{\EE_{\infty}}(\Gr^{\Zchar}_A \cC) \to \End_{\EE_1}(\Gr^{\Zchar}_A \cC) \big).
            \end{equation*}
            Using \cref{prop:universal-property-thom}, $\End_{\EE_k}(\Gr^{\Zchar}_A \cC)$ is identified with the space of $\EE_k$-nullhomotopies of the composition
            \begin{equation*}
                A \to \Sigma \cC\units \to \Sigma (\Gr^{\Zchar}_A \cC)\units.
            \end{equation*}
            The identity map $\Gr^{\Zchar}_A \cC \to \Gr^{\Zchar}_A \cC$ induces a specific nullhomotopy, therefore, the space of $\EE_k$-nullhomotopies is equivalent to the space of $\EE_k$-maps
            \begin{equation*}
                A \to (\Gr^{\Zchar}_A \cC)\units.
            \end{equation*}
            We get that $\Aut_{\Lift(\Gr_A \cC)}(\Gr^{\Zchar}_A \cC)$ is identified with
            \begin{equation*}
                 \fib \big( \Map_{\cnSp}(A, (\Gr^{\Zchar}_A \cC)\units) \to \Map_{\EE_1}(A, (\Gr^{\Zchar}_A \cC)\units) \big),
            \end{equation*}
            i.e.\ $\Aut_{\Lift(\Gr_A \cC)}(\Gr^{\Zchar}_A \cC) \simeq \Map\Enull(A, (\Gr^{\Zchar}_A \cC)\units)$.
    
            Following the above identifications, the map $\labelcref{eqn:map-of-aut}$, is identified with the map
            \begin{equation*}
                \Map\Enull(A, \cC\units) \to \Map\Enull(A, (\Gr^{\Zchar}_A \cC)\units)
            \end{equation*}
            induced by the symmetric monoidal functor $\cC \to \Gr^{\Zchar}_A \cC$.
            As this functor induces the isomorphism $\Omega \cC\units \simeq \Omega (\Gr^{\Zchar}_A \cC)\units \simeq \ounit_{\cC}\units$, the claim follows from \cref{cor:E1-null}.
        \end{proof}

        % \begin{remark}
        %     Note that the isomorphism of \cref{thm:twFun-are-all-lifts} does not hold when $A$ is not grouplike. Indeed, if $M$ is a discrete commutative monoid then $\EE_{\infty}$-maps $M \to \Sigma \cC\units$ factor through the group-completion of $M$, so the left hand side of the isomorphism corresponds to $\Map\Enull(M\gp, \Sigma \cC\units)$. The right hand side \todo{finish this sentence I don't know what you meant}
        % \end{remark}

    %%%%%%%%%%%%%%%%%%%%%%%%%%%%%%%%%%%%%%%%%%%%%%%%%%%%%%%%%%%%%%%%%%%%%%%%%%%%%%%%
    %%%%%%%%%%%%%%%%%%%%%%%%%%%%%%%%%%%%%%%%%%%%%%%%%%%%%%%%%%%%%%%%%%%%%%%%%%%%%%%%
    \subsection{Homotopy graded categories}
    %%%%%%%%%%%%%%%%%%%%%%%%%%%%%%%%%%%%%%%%%%%%%%%%%%%%%%%%%%%%%%%%%%%%%%%%%%%%%%%%
    %%%%%%%%%%%%%%%%%%%%%%%%%%%%%%%%%%%%%%%%%%%%%%%%%%%%%%%%%%%%%%%%%%%%%%%%%%%%%%%%
        A variant of twisted graded categories that naturally arises, and will be used in \cref{sec:Galois-closed-unit}, is the notion of \emph{homotopy graded categories}, which admit a monoidal trivialization only after passing to homotopy categories.
        \begin{definition}
            Let $A \in \Ab$ and $k \ge 1$. A $k$-homotopy $A$-graded category is a symmetric monoidal category $\cD \in \CAlg_{\cC}(\PrL)$ together with an $\EE_0$-isomorphism $\Gr_A \cC \isoto \cD$ and an extension of it to an $\EE_1$-isomorphism of $k$-homotopy categories $\h_k \Gr_A \cC \isoto \h_k \cD$.

            We denote the space of $A$-graded $\cC$-linear categories by $\Lift^{\h,k}(\Gr_A \cC)$.
        \end{definition}

        \begin{corollary}\label{cor:cofiber-sequence-for-Lift^h}
            Let $\cD \in \Lift^{\h,k}(\Gr_A \cC)$ and assume that $\cC$ is connected. Then there is a cofiber sequence of connective spectra
            \begin{equation*}
                \cC\units \to \cD \units \to A.
            \end{equation*}
        \end{corollary}
        \begin{proof}
            This follows from \cref{lem:cofiber-sequence-of-picard-general}.
        \end{proof}

        \begin{definition}\label{def:Map-Enull-homotopy}
            Let $X, Y$ be connective spectra. Let $\Map^{\h,k}\Enull(X,Y)$ be the space of maps of connective spectra $X \to Y$ with an $\EE_1$-nullhomotopy of the composition $X \to Y \to \tau_{\le k+1} Y$.
        \end{definition}

        \begin{notation}
            By an abuse of notation, for an abelian group $A$ and $\Zchar \in \Map^{\h,k}\Enull(A, \Mod_{\cC}\units)$, we denote the Thom construction of $\Zchar$ also by $\Gr_A^{\Zchar} \cC$.
        \end{notation}

        Although for $\Zchar \in \Map^{\h}\Enull(A, \Mod_{\cC}\units)$ it is generally not true that $\Gr_A^{\Zchar} \cC$ is twisted graded, we will now show it is $k$-homotopy graded. We will need the following lemmas:

        \begin{lemma}\label{lem:MaphEnull-fiber-sequence}
            Let $A \in \Ab$ and $X \to Y \to Z$ be an exact sequence of connective spectra. Assume that $Z$ is $k$-truncated. Then
            \begin{equation*}
                \Map^{\h,k}\Enull(A, X) \to \Map^{\h,k}\Enull(A, Y) \to \Map^{\h,k}\Enull(A, Z)
            \end{equation*}
            is a fiber sequence of spaces.
        \end{lemma}
        \begin{proof}
            The sequence of $(k+1)$-truncations
            \begin{equation*}
                \tau_{\le k+1} X \to \tau_{\le k+1} Y \to \tau_{\le k+1} Z
            \end{equation*}
            is a cofiber sequence in the category of $(k+1)$-truncated spectra. As $Z \simeq \tau_{\le k+1} Z$ is $k$-truncated, it is also a cofiber sequence in spectra. As the forgetful to $\EE_1$-groups commutes with limits, this is also a fiber sequence of $\EE_1$-groups. Therefore, we have a commutative diagram where the rows are fiber sequences
            \begin{equation*}
                \begin{tikzcd}
                	{\Map_{\Sp}(A, X)} & {\Map_{\Sp}(A, Y)} & {\Map_{\Sp}(A, Z)} \\
                	{\Map_{\EE_1}(A, \tau_{\le k+1} X)} & {\Map_{\EE_1}(A, \tau_{\le k+1} Y)} & {\Map_{\EE_1}(A, \tau_{\le k+1} Z).}
                	\arrow[from=1-1, to=1-2]
                	\arrow[from=1-1, to=2-1]
                	\arrow[from=1-2, to=1-3]
                	\arrow[from=1-2, to=2-2]
                	\arrow[from=1-3, to=2-3]
                	\arrow[from=2-1, to=2-2]
                	\arrow[from=2-2, to=2-3]
                \end{tikzcd}
            \end{equation*}
            The result follows by taking fibers.
        \end{proof}

        \begin{lemma}\label{lem:MaphEnull-of-1-truncated-is-zero}
            Let $A$ be an abelian group and let $X$ be a $1$-truncated, connective spectrum. Then $\Map\Enull(A, X) \simeq \Map^{h, k}\Enull(A, X)$ is contractible.
        \end{lemma}
        \begin{proof}
            Consider the exact sequence
            \begin{equation*}
                \Sigma \pi_1 X \to X \to \pi_0 X.
            \end{equation*}
            Then by \cref{lem:MaphEnull-fiber-sequence}, it gives rise to a fiber sequence of spaces
            \begin{equation*}
                \Map^{\h, k}\Enull(A, \Sigma \pi_1 X) \to \Map^{\h, k}\Enull(A, X) \to \Map^{\h, k}\Enull(A, \pi_0 X).
            \end{equation*}
            Therefore it is enough to prove the claim when $X$ is concentrated in a single degree. If $X$ is concentrated in degree 0, i.e.\ an abelian group, the claim is clear. Assume therefore $X = \Sigma A'$ for an abelian group $A'$. We want to show that the fiber of
            \begin{equation*}
                \Map_{\Sp}(A, \Sigma A') \to \Map_{\EE_1}(A, \B A')
            \end{equation*}
            at 0 is contractible. As both spaces are 1-truncated, it is enough to verify that $\pi_0, \pi_1$ of $\Map^{\h, k}\Enull(A, \Sigma A')$ are trivial. On $\pi_1$, the above map is
            % \begin{equation*}
            %     \begin{split}
            %         \pi_1 \Map_{\Sp}(A, \Sigma A')
            %         & \simeq \pi_0 \Map_{\Ab}(A, A') \\
            %         & \simeq \pi_0\Map_{\Grp\heart}(A, A') \\
            %         & \simeq \pi_0 \Map_{\EE_1}(A, \Omega \B A') \\
            %         & \simeq \pi_1 \Map_{\EE_1}(A, \B A').
            %     \end{split}
            % \end{equation*}
            \begin{equation*}
                \pi_1 \Map_{\Sp}(A, \Sigma A')
                \simeq \pi_0 \Map_{\Ab}(A, A')
                \simeq \pi_0\Map_{\Grp\heart}(A, A') 
                \simeq \pi_0 \Map_{\EE_1}(A, \Omega \B A') 
                \simeq \pi_1 \Map_{\EE_1}(A, \B A').
            \end{equation*}
            So it is left showing that a map of spectra $A \to \Sigma A'$ is null if and only if it is $\EE_1$-null. Let $F$ be the fiber of $A \to \Sigma A'$, then it is also the fiber of $\EE_1$-groups. In particular, the map is $\EE_1$-nullhomotopic if and only if $F \simeq A \times B$ and the map is the projection, which is true if and only if the map is zero as a map of spectra.
        \end{proof}
        
        \begin{corollary}\label{lem:homotopy-E1-null-factors-through-Sigma^2}
            Let $A \in \Ab$. Then 
            \begin{equation*}
                \Map^{\h, k}\Enull(A, \Mod_{\cC}\units) \simeq \Map^{\h,k}\Enull(A, \Sigma \cC\units) \simeq \Map^{\h,k}\Enull(A, \Sigma^2 \ounit_{\cC}\units).
            \end{equation*}
        \end{corollary}
        \begin{proof}
            Consider the fiber sequences
            \begin{equation*}
                \begin{split}
                    & \Sigma \cC\units \to \Mod_{\cC}\units \to \pi_0 \Mod_{\cC}\units, \\
                    & \Sigma^2 \ounit_{\cC}\units \to \Mod_{\cC}\units \to \tau_{\le 1} \Mod_{\cC}\units.
                \end{split}
            \end{equation*}
            By \cref{lem:MaphEnull-fiber-sequence}, it gives rise to fiber sequences of spaces
            \begin{equation*}
                \begin{split}
                    & \Map^{\h,k}\Enull(A, \Sigma \cC\units) \to \Map^{\h,k}\Enull(A, \Mod_{\cC}\units) \to \Map^{\h,k}\Enull(A, \pi_0 \Mod_{\cC}\units), \\
                    & \Map^{\h,k}\Enull(A, \Sigma^2 \ounit_{\cC}\units) \to \Map^{\h,k}\Enull(A, \Mod_{\cC}\units) \to \Map^{\h,k}\Enull(A, \tau_{\le 1} \Mod_{\cC}\units).
                \end{split}
            \end{equation*}
            By \cref{lem:MaphEnull-of-1-truncated-is-zero} $\Map^{\h,k}\Enull(A, \pi_0 \Mod_{\cC}\units)$ and $\Map^{\h,k}\Enull(A, \tau_{\le 1} \Mod_{\cC}\units)$ are contractible.
        \end{proof}

        \begin{lemma}\label{lem:fiber-sequence-of-homotopy-Enull}
            Let $\cC$ be connected. Let $\Zchar \in \Map^{\h,k}\Enull(A, \Mod_{\cC}\units) \simeq \Map^{\h,k}\Enull(A, \Sigma \cC\units)$. Then the fiber of $\Zchar \colon A \to \Sigma \cC\units$ is identified with $(\Gr_A^{\Zchar} \cC)\units$.
        \end{lemma}

        \begin{proof}
            Let $F'_{\Zchar}$ be the fiber of $\Zchar \colon A \to \Sigma^2\ounit\units$, which easily seen to have $\pi_0 F'_{\Zchar} \simeq A$.
            The commutative diagram
            \begin{equation*}
                \begin{tikzcd}
                    {F'_{\Zchar}} & A \\
                    & {\Mod_{\cC}\units,}
                    \arrow[from=1-1, to=1-2]
                    \arrow["0"', from=1-1, to=2-2]
                    \arrow["\Zchar", from=1-2, to=2-2]
                \end{tikzcd}
            \end{equation*}
            gives rise to a symmetric monoidal functor
            \begin{equation*}
                \Gr_{F'_{\Zchar}} \cC \to \Gr_A^{\Zchar} \cC.
            \end{equation*}
            Composing with the Yoneda embedding we get a map of connective spectra
            \begin{equation*}
                F'_{\Zchar} \to (\Gr_{F'_{\Zchar}} \cC)\units \to (\Gr_A^{\Zchar} \cC)\units.
            \end{equation*}
            Choosing a homotopy multiplicative section $A \simeq \pi_0 F'_{\Zchar} \to F'_{\Zchar}$ we get that the map
            \begin{equation*}
                \ounit_{\cC} \shift{-}\colon A \to (\Gr_A^{\Zchar} \cC)\units
            \end{equation*}
            is a homotopy multiplicative, i.e.\ $\ounit_{\cC}\shift{a+b} \simeq \ounit_{\cC}\shift{a} \otimes \ounit_{\cC}\shift{b}$ for any $a, b \in A$. Therefore the $\EE_0$-isomorphism $\Gr_A^{\Zchar} \cC \simeq \Gr_A \cC$ is homotopy multiplicative: for any $X \simeq \bigsqcup_a X_a \shift{a}, Y \simeq \bigsqcup_a Y_a\shift{a} \in \Gr_A^{\Zchar} \cC$, 
            \begin{equation*}
                X \otimes Y \simeq \bigsqcup_{a,b} X_a\shift{a} \otimes Y_b\shift{b} \simeq \bigsqcup_{a,b} (X_a \otimes Y_b)\shift{a+b} \simeq X \otimes_{\Day} Y.
            \end{equation*}
            It follows that $\Pic(\Gr_A^{\Zchar} \cC) \simeq \Pic(\Gr_A \cC) \simeq \Pic(\cC) \times A$.
            
            Now, let $F_{\Zchar}$ be the fiber of $\Zchar \colon A \to \Sigma \cC\units$. 
            The nullhomotopy of $F_{\Zchar} \to A \to \Mod_{\cC}\units$ once again gives rise to a symmetric monoidal functor
            \begin{equation*}
                \Gr_{F_{\Zchar}} \cC \to \Gr_A^{\Zchar} \cC.
            \end{equation*}
            Composing with the Yoneda embedding we get a map of connective spectra
            \begin{equation*}
                F_{\Zchar} \to (\Gr_{F_{\Zchar}} \cC)\units \to (\Gr_A^{\Zchar} \cC)\units.
            \end{equation*}
            By construction, this map is an isomorphism on $\pi_0$. The connected cover of $F_{\Zchar}$ is the fiber, in connected spectra, of the map $0 \to \Sigma \cC\units$ which is $\Sigma \ounit_{\cC}\units \simeq \tau_{\ge 1} (\Gr_A^{\Zchar} \cC)\units$. Therefore $F_{\Zchar} \to (\Gr_A^{\Zchar} \cC)\units$ is an isomorphism    . 
        \end{proof}

        Note that for $\cC \in \PrL$, its homotopy category $\h_k \cC$ admits all coproducts, and left adjoint functors $\cC \to \cD$ preserve them after passage to homotopy categories. Therefore, the homotopy category functor lands in $\Catcoprod$--- the (huge) category of large categories with all small coproducts and functors preserving them.
        By \cite[\textsection~4.8.1]{Lurie-HA}, the category $\Catcoprod$ admits a bilinear symmetric monoidal structure $\otimes^{\sqcup}$. 

        \begin{lemma}
            Let $M$ be a discrete monoid and $\cE \in \Alg(\Catcoprod)$. Then the monoid algebra $\cE[M] \in \Alg(\Catcoprod)$ is isomorphic to $\FunDay(M, \cE)$. 
        \end{lemma}
        \begin{proof}
            Let $(-)^{\sqcup} \colon \Cat \to \Catcoprod$ be the functor freely adding coproducts. That is, for any $\cI \in \Cat$,
            \begin{equation*}
                \cI^{\sqcup} \subseteq \PSh(\cI)
            \end{equation*}
            is generated by coproducts under the image of the Yoneda embedding and the functoriality is given by left Kan extensions, noting that it preserves coproducts of representable objeects. That is, for any functor $f \colon \cI \to \cJ$ the following diagram commutes
            \begin{equation*}
                \begin{tikzcd}
                	{\cI^{\sqcup}} & {\PSh(\cI)} \\
                	{\cJ^{\sqcup}} & {\PSh(\cJ).}
                	\arrow[hook, from=1-1, to=1-2]
                	\arrow["{f^{\sqcup}}", from=1-1, to=2-1]
                	\arrow["{f_!}", from=1-2, to=2-2]
                	\arrow[hook, from=2-1, to=2-2]
                \end{tikzcd}
            \end{equation*}
            By \cite[\textsection~4.8.1]{Lurie-HA}, the functor $(-)^{\sqcup}$ is symmetric monoidal. In particular, it lifts to a functor 
            \begin{equation*}
                (-)^{\sqcup} \colon \Mon(\Cat) \to \Alg(\Catcoprod).
            \end{equation*}
            Following the above discussion, for $X \in \Set$, 
            \begin{equation*}
                X^{\sqcup} \simeq \Fun(X, \Set)
            \end{equation*}
            and for $M \in \Mon(\Set)$ it is equipped with the Day monoidal structure
            \begin{equation*}
                M^{\sqcup} \simeq \FunDay(M, \Set) = \Gr_M \Set.
            \end{equation*}
            The claim now follows as
            \begin{equation*}
                \cE[M] \simeq \cE \otimes^{\sqcup} M^{\sqcup} \simeq \cE \otimes^{\sqcup} \FunDay(M, \Set) \simeq \FunDay(M, \cE).
            \end{equation*}
        \end{proof}

        \begin{corollary}\label{cor:homotopy-category-of-graded-is-graded}
            Let $A \in \Ab$. Then 
            \begin{equation*}
                \h_k \Gr_A \cC \simeq (\h_k\cC)[A] \qin \Alg(\Catcoprod).
            \end{equation*}
        \end{corollary}

        \begin{lemma}\label{lem:E1-iso-on-homotopy-categories}
            Let $A \in \Ab$ and assume $\cC$ is connected.  
            Let $\cD \in \CAlg_{\cC}(\PrL)$ together with an $\EE_0$-isomorphism $\Gr_A(\cC) \isoto \cD$.
            Then equipping the isomorphism of $(k,1)$-categories
            \begin{equation*}
                \h_k\Gr_A(\cC) \isoto \h_k\cD
            \end{equation*}
            with an $\EE_1$-structure is equivalent to a factorization of the map of spaces
            \begin{equation*}
                \begin{tikzcd}
                	&& {\tau_{\le k}\cD\units} \\
                	{\tau_{\le k}(\Gr_A \cC)\units} & {\tau_{\le k}(\Gr_A \cC)\core} & {\tau_{\le k}\cD\core}
                	\arrow[from=1-3, to=2-3]
                	\arrow[dashed, from=2-1, to=1-3]
                	\arrow[from=2-1, to=2-2]
                	\arrow[from=2-2, to=2-3]
                \end{tikzcd}
            \end{equation*}
            as an isomorphism of $\EE_1$-groups.
        \end{lemma}
        
        \begin{proof}
            The first implication follows as $\tau_{\le k} \cD\units \simeq (\h_k \cD)\units$ and $\tau_{\le k} (\Gr_A \cC)\units \simeq (\h_k \Gr_A \cC)\units$. 

            Assume the isomorphism $\Gr_A \cC \isoto \cD$ is equipped with an $\EE_1$-isomorphism of groups
            \begin{equation*}
                \tau_{\le k} (\Gr_A \cC)\units \isoto \tau_{\le k} \cD\units
            \end{equation*}
            factoring the above map. Consider the composition
            \begin{equation*}
                A \to \tau_{\le k} (\Gr_A \cC)\units \isoto \tau_{\le k} \cD\units .
            \end{equation*}
            This is an $\EE_1$-map, thus by \cref{cor:homotopy-category-of-graded-is-graded}, it induces a unique coproduct preserving $\h_k\cC$-linear $\EE_1$-functor
            \begin{equation*}
                \h_k \Gr_A \cC \simeq (\h_k \cC)[A] \to \h_k \cD,
            \end{equation*}
            in particular it lifts the given $\EE_0$-isomorphism $\Gr_A \cC \isoto \cD$ on $k$-homotopy categories.
        \end{proof}

        \begin{corollary}\label{cor:Thom-of-homotopy-null-map}
            Let $\cC$ be connected. The functor $\Gr_A^{(-)} \colon \Map^{\h,k}\Enull(A, \Mod_{\cC}\units) \to \CAlg_{\cC}(\PrL)$ lifts naturally to
            \begin{equation*}
                \Gr_A^{(-)} \Map^{\h,k}\Enull(A, \Mod_{\cC}\units) \to \Lift^{\h,k}(\Gr_A \cC).
            \end{equation*}
        \end{corollary}
        \begin{proof}
            Considering \cref{lem:E1-iso-on-homotopy-categories}, it is enough to show that the map of spaces 
            \begin{equation*}
                \tau_{\le k}(\Gr_A^{\Zchar} \cC)\units \to \tau_{\le k}(\Gr_A^{\Zchar} \cC)\core \simeq \tau_{\le k}(\Gr_A \cC)\core
            \end{equation*}
            factors through an isomorphism of $\EE_1$-groups
            \begin{equation*}
                \tau_{\le k}(\Gr_A^{\Zchar} \cC)\units \isoto \tau_{\le k}(\Gr_A \cC)\units.
            \end{equation*}
            Applying $k$-truncation to the rolled-back fiber sequence of \cref{lem:fiber-sequence-of-homotopy-Enull}, and using that $A$ is discrete we get the exact sequence of spectra
            \begin{equation*}
                \tau_{\le k}\cC\units \to \tau_{\le k}(\Gr_A^{\Zchar}\cC)\units \to A .
            \end{equation*}
            Rolling the sequence forward then gives the map
            \begin{equation*}
                A \xto{\Zchar} \Sigma \cC\units \to \tau_{\le k+1}\Sigma \cC\units \simeq \Sigma \tau_{\le k}\cC\units.
            \end{equation*}
            As this maps is $\EE_1$-nullhomotopic, the above fiber sequence gives an $\EE_1$-isomorphism
            \begin{equation*}
                \tau_{\le k}(\Gr_A^{\Zchar} \cC)\units \simeq \tau_{\le k} \cC\units \times A \simeq \tau_{\le k}(\Gr_A \cC)\units
            \end{equation*}
            as needed.
        \end{proof}

        \begin{theorem}\label{thm:homotopy-twFun-are-all-lifts}
            Let $\cC$ be connected. The Thom construction lifts to an isomorphism of sets
            \begin{equation*}
                \Gr_A^{(-)} \cC \colon \pi_0 \Map^{\h,k}\Enull(A, \Mod_{\cC}\units) \isoto \pi_0 \Lift^{\h,k}(\Gr_A \cC).
            \end{equation*}
        \end{theorem}

        \begin{proof}
            As both $\Map^{\h, k}\Enull(A,-)$ and $\Lift^{\h,k}[-](\Gr_A (-))$ send products to products, we may assume $\cC$ is connected.
            
            Let $\cD\in\Lift^{\h,k}(\Gr_A \cC)$. By \cref{cor:cofiber-sequence-for-Lift^h}, $A$ sits in a cofiber sequence
            \begin{equation*}
                \cC\units \to \cD\units \to A.
            \end{equation*}
            Rolling the cofiber sequence, we get a natural map $\Zchar_{\cD} \colon A \to \Sigma \cC\units$. 
            The canonical nullhomotopy
            \begin{equation*}
                A \xto{\Zchar_{\cD}} \Sigma \cC\units \to \Sigma \cD\units
            \end{equation*}
            induces a null homotopy
            \begin{equation*}
                A \xto{\Zchar_{\cD}} \Sigma \cC\units \to \Mod_{\cC}\units \to \Mod_{\cD}\units.
            \end{equation*}
            By \cref{prop:universal-property-thom} this induces a map of $\cC$-linear symmetric monoidal categories $\Gr^{\Zchar_{\cD}}_A \cC \to \cD$, which is an $\EE_0$-isomorphism and agrees with the monoidal trivialization of homotopy categories. Therefore it is an isomorphism in $\Lift^{\h,k}(\Gr_A \cC)$. This proves that $\Gr^{(-)}_A \cC$ is surjective on $\pi_0$. It is injective by \cref{cor:cofiber-sequence-for-Lift^h}, \cref{lem:MaphEnull-fiber-sequence}.
        \end{proof}

        \begin{remark}
            Similarly to \cref{thm:twFun-are-all-lifts}, one can promote this isomorphism to an isomorphism of spaces. Carrying this out requires an analysis of the Thom construction in non-presentable categories, which lies beyond the scope of this paper.
        \end{remark}

        \begin{remark}
            One can phrase a similar statement for the case $k=0$, where $\h_0 \cD \coloneqq \pi_0 \cD\core$. In this case, by essentially the same argument, the Thom construction yields a bijection between maps of spectra $A \to \Sigma^2 \ounit_{\cC}\units$ and presentable symmetric monoidal $\cC$-linear categories $\cD$ equipped with an $\EE_0$-isomorphism $\Gr_A \cC \isoto \cD$ that is homotopy commutative.
        \end{remark}

    %%%%%%%%%%%%%%%%%%%%%%%%%%%%%%%%%%%%%%%%%%%%%%%%%%%%%%%%%%%%%%%%%%%%%%%%%%%%%%%%
    \subsection{Examples}
    \label{subsec:examples}
    %%%%%%%%%%%%%%%%%%%%%%%%%%%%%%%%%%%%%%%%%%%%%%%%%%%%%%%%%%%%%%%%%%%%%%%%%%%%%%%%
        We show some interesting examples of twisted graded categories. Our examples will be focused on the case $A = \ZZ$, in this case $\ZZ\Enull = \tau_{\ge 1} \SS$. The most famous twisted $\ZZ$-graded category is the category of graded objects with the Koszul twist. 
        
        We start with the trivial example.
        
        \begin{example}\label{exm:twFun-0-is-Day}
            Let $0 \colon A \to \Sigma \cC\units$ be the trivial map. Then $\Gr^{0}_{A} \cC \simeq \Gr_A \cC$.
        \end{example}

        When $\cC$ admits a minus one, one can reconstruct the Koszul twist. In the proof of this we will consider the universal category admitting a minus one.
    
        \begin{lemma}\label{lem:universal-case}
            Let $X \in \cnSp$. The universal category $\cC\in\CAlg(\PrL)$ with a morphism of connective spectra $X \to \ounit_{\cC}\units$ is 
            \begin{equation*}
                \spc[\Sigma X] \simeq \Mod_{X}(\spc) \in \CAlg(\PrL).
            \end{equation*}
            That is,
            \begin{equation*}
                \Fun_{\CAlg(\PrL)}(\Mod_X (\spc), \cC) \simeq \Map_{\cnSp}(X, \ounit_{\cC}\units).
            \end{equation*}
        \end{lemma}
    
        \begin{proof}
            For any $\cC\in \CAlg(\PrL)$, by Morita theory (\cite[Theorem~4.8.4.6, Corollary~4.8.5.22]{Lurie-HA}) and using the $\ounit_{\cC}[-] \dashv (-)\units$ adjunction:
            \begin{equation*}
                \begin{split}
                    \Map_{\CAlg(\PrL)}(\Mod_{X}(\spc), \cC) 
                    & \simeq \Map_{\CAlg_{\cC}(\PrL)}(\cC \otimes \Mod_{X}(\spc), \cC) \\
                    & \simeq \Map_{\CAlg_{\cC}(\PrL)}(\Mod_{\ounit_{\cC}[X]}(\cC), \Mod_{\ounit_{\cC}} (\cC)) \\
                    & \simeq \Map_{\CAlg(\cC)}(\ounit_{\cC}[X], \ounit_{\cC}) \\
                    & \simeq \Map_{\cnSp}(X, \ounit_{\cC}\units).
                \end{split}
            \end{equation*}
    
        \end{proof}

        \begin{definition}\label{def:Koszul}
            Let $\cC \in \CAlg(\PrL)$ and assume $\cC$ admits a minus one, i.e.\ a group homomorphism
            \begin{equation*}
                (-1) \colon \ZZ/2 \to \ounit_{\cC}\units.
            \end{equation*}
            Define $\Kos \colon \ZZ \to \Sigma \cC\units$ as the composition
            \begin{equation*}
                \Kos \colon \ZZ \onto \ZZ/2 \xto{\Sq^2} \Sigma^2 \ZZ/2 \xto{\Sigma^2 (-1)} \Sigma^2 \ounit_{\cC}\units \to \Sigma \cC\units.
            \end{equation*}
            The factorization through $\Sigma^2 \ounit_{\cC}\units$ gives a natural $\EE_1$-nullhomotopy.

            Equivalently, it is the map
            \begin{equation*}
                \ZZ\Enull = \tau_{\ge 1} \SS \to \Sigma \ZZ/2 \xto{\Sigma (-1)} \Sigma \ounit_{\cC}\units \to \cC\units.
            \end{equation*}
        \end{definition}

        \begin{proposition}\label{prop:Kos-is-Kos}
            Let $\cC\in\CAlg(\PrL)$ admitting a minus one $(-1) \colon \ZZ/2 \to \ounit_{\cC}\units$. 
            Then the braiding in $\Gr^{\Kos}_{\ZZ} \cC$ is given by a Koszul sign:
            For any $X \in \cC$, and integers $n,m\in\ZZ$, the isomorphism 
            \begin{equation*}
                X\shift{n}\otimes X\shift{m} \isoto X\shift{m} \otimes X\shift{n}
            \end{equation*}
            is given by multiplication with $(-1)^{nm} \in \ounit_{\cC}\units$.
        \end{proposition}
        
        \begin{proof}
            It suffices to prove the claim for the universal presentably symmetric monoidal category admitting a minus one $\spc[\Sigma \ZZ/2] \simeq \Mod_{\ZZ/2}(\spc)$ (\cref{lem:universal-case}). We first consider the non-presentable analog. The universal symmetric monoidal category admitting a minus one is $\B \ZZ/2$ and its endomorphism monoid unit is the connective spectrum $\ZZ/2 = \End(\ounit_{\B\ZZ/2})$.
            The Koszul map is the composition
            \begin{equation*}
                \ZZ \onto \ZZ/2 \xto{\Sq^2} \Sigma^2 \ZZ/2 \simeq \Sigma (\B\ZZ/2)\units,
            \end{equation*}
            which is known to have a fiber $\tau_{\le 1} \SS$ at each $n\in \ZZ$.            
            Rolling the fiber sequence further we get the fiber sequence
            \begin{equation*}
                \Sigma \ZZ/2 \to \tau_{\le 1}\SS \to \ZZ.
            \end{equation*}
            By unstraightening, $\tau_{\ge 1}\SS \simeq \colim_{\ZZ} \Sigma{\ZZ/2}$ (which is the non-presentable analog of $\Gr^{\Kos}_{\ZZ} \B\ZZ/2$). By commutation of colimits
    
            \begin{equation*}
                \Gr^{\Kos}_{\ZZ} \spc[\B \ZZ/2] \simeq \spc[\colim_{\ZZ} \B\ZZ/2] \simeq \spc[\tau_{\le 1}\SS].
            \end{equation*}
            Now, as it suffices to prove the claim for $X = \ounit_{\cC}$, we can prove it for $\tau_{\le 1} \SS$.
            This follows by observing that the canonical symmetric monoidal map 
            \begin{equation*}
                \bigsqcup_{\degree} \B\Sm \simeq \Fin\core \to \SS \to \tau_{\le 1} \SS
            \end{equation*}
            is the sign map on $\pi_1$.
        \end{proof}
        
        \begin{remark}\label{rmrk:2-structurs-on-GrVect}
            Note that for $\cC = \Vect$, symmetric monoidal structures on $\Gr_{\ZZ}\cC$ with an $\EE_1$-trivialization are the same as maps of spectra $\tau_{\ge 1} \SS \to  \Vect\units$  which are the same as maps of abelian groups $\ZZ/2 \to k\units$. So we get that there are exactly two symmetric monoidal structures on $\Gr_{\ZZ}\Vect$ with an $\EE_1$-trivialization: The usual graded and the Koszul-graded.
        \end{remark}

        %%%%%%%%%%%%%%%%%%%%%%%%%%%%%%%%%%%%%%%%%%%%%%%%%%%%%%%%%%%%%%%%%%%%%%%%%%%%%%%%
        \subsubsection*{Higher minus one}

        This example can be extended to a higher notion of minus one:
        In \cite[Definition~4.2]{CSY-cyclotomic} Carmeli, Schlank and Yanovski define higher roots of unity, and in particular one can define a higher minus one:
        \begin{definition} 
            Let $\cC\in\CAlg(\PrL)$ of height $\chrHeight$. A minus one of height $\chrHeight$ in $\cC$ is a map of connective spectra
            \begin{equation*}
                \minusone \colon \Sigma^{\chrHeight}\ZZ/2 \to \ounit_{\cC}\units.
            \end{equation*}
            We say that it is primitive if the only commutative algebra $S\in\CAlg(\cC)$ rendering the diagram of connective spectra
            \begin{equation*}
                \begin{tikzcd}
            	{\Sigma^{\chrHeight} \ZZ/2} & \ounit_{\cC}\units \\
            	0 & S\units
            	\arrow[from=1-1, to=2-1]
            	\arrow["\minusone", from=1-1, to=1-2]
            	\arrow[from=1-2, to=2-2]
            	\arrow[from=2-1, to=2-2]
                \end{tikzcd}
            \end{equation*}
            commutative is $S=0$.
        \end{definition}

        In \cite[Definition~4.7]{CSY-cyclotomic} the higher cyclotomic extension $\ounit_{\cC}[\minusone]$ is defined and in \cite[Proposition~4.8]{CSY-cyclotomic} it is proved that $\Mod_{\ounit_{\cC}[\minusone]}(\cC)$ admits a primitive $(-1)$ of height $\chrHeight$. In the case where $\cC$ is stable and virtually $(\Fq{2}, \chrHeight)$-oriented (\cite[\S~6]{BCSY-Fourier}), $\ounit_{\cC}[\minusone]$ is a $e = (\ZZ/2)\units$-Galois extension \cite[Proposition~6.11]{BCSY-Fourier}, and thus canonically isomorphic to $\ounit_{\cC}$. In particular such $\cC$ admits a primitive higher $(-1)$.

        \begin{example}
            For any $R\in\CAlg(\SpTn)$ at $p=2$, $\Mod_R(\SpTn)$ is virtually $(\Fq{2},\chrHeight)$-orientable (\cite[Proposition~7.27, Corollary~4.21]{BCSY-Fourier}) thus admits a minus one of height $\chrHeight$.
            In particular, the categories $\SpTn$, $\SpKn=\Mod_{\SSKn}(\SpTn)$ and $\ModEn=\Mod_{\En}(\SpTn)$ admit height $\chrHeight$ minus ones.
        \end{example}

        Other examples come from categorification.
        \begin{definition}\label{def:Vectn}
            Let $\field$ be a field of characteristic 0. 
            Define $\Vectn[1] \coloneqq \Vect \in \CAlg(\PrL)$ and $\Vectnp \coloneqq \Mod_{\Vectn} \in \CAlg(\PrL)$.

            To be more formal, we choose a sequence of inaccessible cardinals $\kappa_1 \subseteq \kappa_2 \subseteq \cdots$, and consider $\Vectn[1]$ as lying in $\CAlg(\PrL_{\kappa_1})$. We then define $\Vectnp \coloneqq \Mod_{\Vectn}(\PrL_{\kappa_{\chrHeight}}) \in \CAlg(\PrL_{\kappa_{\chrHeight+1}})$. We will usually omit this careful description and always assume we work in $\PrL_{\kappa_{\chrHeight}}$ for large enough $\chrHeight$.
        \end{definition}
        
        One can construct by hands\footnote{It is actually possible to extend the results of Carmeli, Schlank and Yanovski to show that any (not necessarily stable) category which is virtually $(\Fq{2},\chrHeight)$-orientable has a higher minus one, as one can show that in this case as well, $\ounit_{\cC}[\omega_{p^r}^{(n)}]$ is $e$-Galois.} a higher minus one in $\Vectn$:
        \begin{example}
            For a category $\cC\in\CAlg(\PrL)$ we have a map of connective spectra 
            \begin{equation*}
                \Sigma\ounit_{\cC}\units \to \cC\units = \cC\units
            \end{equation*}
            corresponding to the connected component of the unit.
            By induction we get a map of connective spectra
            \begin{equation*}
                \Sigma^{\chrHeight} \field\units \to (\Vectn)\units.
            \end{equation*}
            $\Vectn$ is the unit of $\Vectnp$ and thus $(-1)\in \field\units$ defines a height $\chrHeight$ minus one of $\Vectnp$:
            \begin{equation*}
                \SnCt \xto{\Sigma^{\chrHeight} (-1)} \Sigma^{\chrHeight} \field\units \to (\Vectn)\units.
            \end{equation*}
        \end{example}

        The Koszul twist required a map $\ZZ \to \Sigma^2 \ZZ/2$ (of which there are exactly 2). A higher Koszul twist can be achieved by a map $\hchar \colon \ZZ \to \Sigma^{\chrHeight+2} \ZZ/2$, or equivalently an element in $\H^{\chrHeight+2}(\ZZ, \Fq{2}) \coloneqq \pi_{-n-2}\hom(\ZZ, \Fq{2})$.
        \begin{definition}\label{def:higher-minus-one}
            Let $\cC \in \CAlg(\PrL)$ and $\minusone \colon \Sigma^{\chrHeight} \ZZ/2 \to \ounit_{\cC}\units$ be a minus one of height $\chrHeight$. Let $\hchar \in \H^{\chrHeight+2}(\ZZ, \Fq{2})$. Then define 
            \begin{equation*}
                \Zchar_{\hchar} \colon \ZZ \xto{\hchar} \Sigma^{\chrHeight+2} \ZZ/2 \xto{\Sigma^2 \minusone} \Sigma^2 \ounit_{\cC}\units \to \Sigma \cC\units,
            \end{equation*}
            and
            \begin{equation*}
                \Gr^{\hchar}_{\ZZ} \cC \coloneqq \Gr^{\Zchar_{\hchar}}_{\ZZ} \cC.
            \end{equation*}
        \end{definition}

        %%%%%%%%%%%%%%%%%%%%%%%%%%%%%%%%%%%%%%%%%%%%%%%%%%%%%%%%%%%%%%%%%%%%%%%%%%%%%%%%
        \subsubsection*{Dual stable stems in the oriented case}  
        \label{Galois-closed}
        Assume that $\field$ is a cyclotomically-closed field. One can think of the Koszul map $\ZZ \to \Sigma^2 \ZZ/2$ as a map $\ZZ \to \Sigma^2 \mu_{\infty}$. This can be extended in the $\infty$-semiadditive case when the unit is \quotes{cyclotomically-closed}:
        In \cite{BCSY-Fourier}, the authors introduce the notion of orientability for higher semiadditive categories, which is an analog of primitive roots of unity. 
        First we extend the group $\mu_{\infty}$ of roots of unity.
        \begin{definition}
            Let $\chrHeight \ge 0$.
            \begin{enumerate}
                \item Let $\In \coloneqq \tau_{\ge 0} \Sigma^{\chrHeight} I_{\QQ/\ZZ}$ be the truncated and shifted Brown--Comenetz dual of the sphere.
                \item For a prime $p$, let $\Inprime \coloneqq \tau_{\ge 0} \Sigma^{\chrHeight} I_{\QQ_p/\Zp}$ be the $p$-typical truncated and shifted Brown--Comenetz dual of the sphere.
            \end{enumerate}
        \end{definition}

        \begin{remark}
            As $\QQ / \ZZ \simeq \bigoplus_p \QQ_p / \Zp$ and the homotopy groups of the sphere are finitely generated, $\In \simeq \bigoplus_p \Inprime$.
        \end{remark}

        \begin{remark}\label{rmrk:maps-to-In}
            Let $X$ be a connective spectrum. Then, as $\Inprime$ is $p$-local and $X$ is connective
            \begin{equation*}
                \Map_{\cnSp}(X, \Inprime) 
                \simeq \Map_{\cnSp}(X_{(p)}, \Inprime) 
                \simeq \Map_{\cnSp}(X_{(p)}, \tau_{\ge \chrHeight}\Sigma^{\chrHeight} I_{\QQ_p/\Zp})
                \simeq \Map_{\Sp}(\Sigma^{-\chrHeight} X_{(p)}, I_{\QQ_p/\Zp}).
            \end{equation*}
            Therefore,
            \begin{equation*}
                \pi_0 \Map_{\cnSp}(X, \Inprime) \simeq \widehat{\pi}_{\chrHeight}(X_{(p)}) \coloneqq \Hom_{\Ab}(\pi_{\chrHeight} X_{(p)}, \QQ_p/\Zp).
            \end{equation*}
            Similarly,
            \begin{equation*}
                \pi_0 \Map_{\cnSp}(X, \In) \simeq \widehat{\pi}_{\chrHeight}(X).
            \end{equation*}
        \end{remark}

        We work in both the integral and $p$-typical settings, and will generally write $\In$ to denote either $\In$ (in the integral case) or $\Inprime$ (in the $p$-typical case), depending on context. We trust the reader to distinguish between the two when necessary.

        The connective spectrum $\In$ defines a height $\chrHeight$ analog of the Pontryagin dual, sending $M \in \Sp$ to $\tau_{\ge 0}\hom(M,\In)$.
        
        \begin{definition}
            Let $\cC$ be a presentably symmetric monoidal, $\infty$-semiadditive category. Let $R \in \CAlg(\cnSp)$. An $(R,\chrHeight)$-pre-orientation of $\cC$ is a map of connective spectra
            \begin{equation*}
                \omega \colon \tau_{\ge 0} \hom(R, \In) \to \ounit_{\cC}\units.
            \end{equation*}
        \end{definition}

        \begin{definition}
            An $R$-module $M$ is $[0,\chrHeight]$-finite if it is connective, $\chrHeight$-truncated and all its homotopy groups are finite. We denote the full subcategory of $[0,\chrHeight]$-finite $R$-modules by $\Mod_R^{[0,\chrHeight]\hyphen\mrm{finite}}$
        \end{definition}

        \begin{proposition}[{\cite[Proposition~3.10]{BCSY-Fourier}}]
            The space of $(R,\chrHeight)$-pre-orientations of $\cC$ is isomorphic to the space of natural transformations 
            \begin{equation*}
                \ounit_{\cC}[-] \to \ounit_{\cC}^{\Map(-,\In)} \qin \Fun(\Mod_R^{[0,\chrHeight]\hyphen\mrm{finite}}, \CAlg(\cC)).
            \end{equation*}
        \end{proposition}
        For a pre-orientation $\omega$, we denote the corresponding natural transformation by $\mscr{F}_{\omega}$ and call it the associated Fourier transform.

        \begin{definition}
            A pre-orientation $\omega$ is an orientation if the associated Fourier transform is a natural isomorphism. If the space of $(R,\chrHeight)$-orientation of $\cC$ is non-empty we say that $\cC$ is $(R,\chrHeight)$-orientable.
        \end{definition}

        For any $R$, $\tau_{\ge 0}\hom(R,\In) \simeq \tau_{\ge 0}\hom(\tau_{\le \chrHeight} R,\In)$. Therefore an $(R,\chrHeight)$-orientation is equivalent to a $(\tau_{\le \chrHeight} R, \chrHeight)$-orientation.

        \begin{example}
            A $(\ZZ/N, 0)$-orientation of $\Vect$ is a choice of a primitive $N$-th root of unity in $\field$. $\Vect$ is $(\SS,0)$-orientable if and only if $\field$ is cyclotomically-closed.
        \end{example}

        \begin{example}
            Let $\cC$ be $\infty$-semiadditive. A $(\ZZ/p^r,\chrHeight)$-orientation of $\cC$ is equivalent to a primitive $p^r$-th root of unity of height $\chrHeight$, as in \cite{CSY-cyclotomic}.
        \end{example}

        We interpret the $(\SS, \chrHeight)$-orientability of $\cC$ as expressing that the unit $\ounit_{\cC}$ is spherically cyclotomically closed.       
        
        \begin{example}[{\cite[Theorem~7.8]{BCSY-Fourier}}]
            $\ModEn$ is $(\SS,\chrHeight)$-orientable.
        \end{example}

        % \begin{corollary}
        %     Let $\cC \in \CAlg(\PrL)$ be $\infty$-semiadditive. Assume that $\cC$ is $(\SS,\chrHeight)$-orientable and that $\ounit_{\cC}$ is $\chrHeight$-truncated. Then $\cC$ is of height $\le \chrHeight$ at all primes.
        % \end{corollary}
        % \begin{proof}
        %     As $\ounit_{\cC}$ is Galois-closed at $p$, $(\ounit_{\cC}\units)\pitor_{(p)} \simeq \Inprime[\chrHeight_p]$, where $\chrHeight_p$ is the height of $\cC$ at $p$. As $\ounit_{\cC}\units$ is $\chrHeight$-truncated, $\chrHeight_p \le \chrHeight$. 
        % \end{proof}

        \begin{definition}\label{def:cyc-closure}
            A (height $\chrHeight$) cyclotomic-closure of $\ounit_{\cC}$ is an algebra $\ounit_{\cC}\cyc \in \CAlg(\cC)$ such that:
            \begin{enumerate}
                \item $\Mod_{\ounit_{\cC}\cyc}(\cC)$ is $(\SS,\chrHeight)$-orientable.
                \item $\ounit_{\cC} \to \ounit_{\cC}\cyc$ is a Galois extension in the sense of Rognes \cite{Rognes-2008-Galois}.
                \item It is the minimal such extension: For any Galois extension $\ounit \to R$ such that $\Mod_R(\cC)$ is $(\SS,\chrHeight)$-orientable, there exists a map $\ounit_{\cC}\cyc \to R$.
            \end{enumerate}
        \end{definition}
        We do not prove the existence of Galois closures in this paper but instead assume it. The general existence of cyclotomic extensions will be addressed in a forthcoming paper by the first author. In particular we assume the existence and fix a cyclotomic closure $\cVectn$ of $\Vectn$. We make no assumptions about its uniqueness.

        % In this case, we have a map\footnote{Here when working in the $p$-typical case we replace $\In$ with $\Inprime$, similarly one should replace $\QQ/\ZZ$ with $\QQ_p/\Zp$.}
        % \begin{equation*}
        %     \In \simeq (\ounit_{\cC}\units)\pitor \to \ounit_{\cC}\units.
        % \end{equation*}

        \begin{example}\label{exm:higher-minus-one-for-galois-closed}
            Any $\cC$ which is $(\SS_{(2)},\chrHeight)$-oriented admits a higher minus one. It is defined as the composition
            \begin{equation*}
                \Sigma^{\chrHeight} \ZZ/2 \to \Sigma^{\chrHeight} \QQ_2/\ZZ_2 \to \Inprime[2] \to \ounit_{\cC}\units.
            \end{equation*}
        \end{example}

        \begin{remark}
            The map $\Vectn \to \cVectn$, together with the higher minus one map \linebreak$\minusone \colon \Sigma^{\chrHeight} \ZZ/2 \to (\Vectn)\units$ induces a higher minus one of $\cVectn$. It agrees with the one of \cref{exm:higher-minus-one-for-galois-closed}.
        \end{remark}

        \begin{notation}
            Denote $\pinD \coloneqq \widehat{\pi}_{\chrHeight+1}(\SS)$. When working $p$-locally we would mean $\pinD = \widehat{\pi}_{\chrHeight+1}(\SS_{(p)})$.
        \end{notation}

        \begin{definition}\label{def:dual-stable-stems}
            Let $\cC \in \CAlg(\PrL)$ be $(\SS,\chrHeight)$-orientable. Let $\pichar \in \pinD = \pi_0(\Inplus)$. Then $\pichar$ induces a map
            \begin{equation*}
                \pichar \colon \SS \to \Inplus.
            \end{equation*}
            Taking connected covers, we get a map
            \begin{equation*}
                \tau_{\ge 1} \pichar \colon \tau_{\ge 1} \SS \to \Sigma \In \to \Sigma \ounit_{\cC}\units \to \cC\units.
            \end{equation*}
            By \cref{cor:E1-null}, it is equivalent to a map
            \begin{equation*}
                \Zchar_{\pichar} \colon \ZZ \to \Sigma^2 \In \to \Sigma^2 \ounit_{\cC}\units.
            \end{equation*}
            with an $\EE_1$ null homotopy.
            We define
            \begin{equation*}
                \Gr^{\pichar}_{\ZZ} \cC \coloneqq \Gr^{\Zchar_{\pichar}}_{\ZZ} \cC.
            \end{equation*}
        \end{definition}

        \begin{example}
            Let $\field$ be a cyclotomically-closed field of characteristic 0. Then $\Vect$ is $(\SS,0)$-orientable. $\pinD[1] = \{0, \hat{\eta}\}$ has 2 elements, and
            \begin{equation*}
                \Gr^0_{\ZZ} \Vect \simeq \Gr_{\ZZ}\Vect, \qquad \Gr^{\hat{\eta}}_{\ZZ} \Vect \simeq \Gr^{\Kos}_{\ZZ} \Vect.
            \end{equation*}
        \end{example}

        \begin{remark}\label{rmrk:height-0-dual-stems}
            The same is true for any $(\SS_{(2)}, 0)$-orientable category. The orientation $\Fq{2} \to \ounit_{\cC}\units$ is a choice of a minus one, and with respect to it $\Zchar_{\hat{\eta}} = \Kos$ (as in \cref{def:Koszul}).
        \end{remark}

%%%%%%%%%%%%%%%%%%%%%%%%%%%%%%%%%%%%%%%%%%%%%%%%%%%%%%%%%%%%%%%%%%%%%%%%%%%%%%%%
%%%%%%%%%%%%%%%%%%%%%%%%%%%%%%%%%%%%%%%%%%%%%%%%%%%%%%%%%%%%%%%%%%%%%%%%%%%%%%%%    
\section{Braiding}
\label{sec:braiding}
%%%%%%%%%%%%%%%%%%%%%%%%%%%%%%%%%%%%%%%%%%%%%%%%%%%%%%%%%%%%%%%%%%%%%%%%%%%%%%%%
%%%%%%%%%%%%%%%%%%%%%%%%%%%%%%%%%%%%%%%%%%%%%%%%%%%%%%%%%%%%%%%%%%%%%%%%%%%%%%%%
    In \cref{sec:twisted-graded-categories} we constructed many symmetric monoidal categories that agree monoidally. In this section we introduce an invariant to help distinguish these monoidal structures. Namely, we study what we call the braiding of an element in the category. Given an object $X \in \cD$ we consider the sequence of $\Sm$-representations $\Tm X \coloneqq X\om$. This sequence is multiplicative, in the sense that
    \begin{equation*}
        \Tm[k] X \otimes \Tm[\ell] X \simeq \Res_{\Sm[k]\times\Sm[\ell]}^{\Sm[k+\ell]} \Tm[k+\ell] X,
    \end{equation*}
    giving it a lot more structure. 
    For most categories though, it is hard to completely classify the braiding, for example, in \cref{sec:graded-braiding}, we show that in the case of twisted graded categories, the braidings of $\ounit_{\cC}\shift{a}$ completely determine the symmetric monoidal structure.

    As in usual representation theory, if $V \in \cD\dbl$, we can consider the sequence of characters $\chi_{\Tm V}$ (in the sense of \cite{Ponto-Schulman-2014-induced-character}, \cite{Hoyois-Scherozke-Sibilla-2017-traces} or \cite{Carmeli-Cnossen-Ramzi-Yanovski-2022-characters}). We introduce an invariant called the braiding character, which is a map
    \begin{equation*}
        \mscr{X}_{\T V} \colon \bigsqcup_{\degree} \L\B\Sm \to \ounit_{\cD}[t^{\pm 1}],
    \end{equation*}
    sending $\L\B\Sm$ to $\chi_{\Tm V} \, t^{\degree}$. This invariant is also multiplicative in a similar sense.
    This grading of the character on its own, does not provide with more information than the usual character $\chi_{\T V}$ (remembering the monoidality), but it will turn out to be important when applying it to twisted graded categories $\cD = \Gr^{\Zchar}_A \cC$, as we will see in the next section. 

    In \cref{subsec:braiding}, we introduce and analyze the braiding functor. In \cref{subsec:equivariant-trace}, we recall key properties of the monoidal trace and monoidal dimension, and use them to study and classify the $\TT$-action on the dimension of invertible objects. In \cref{subsec:braiding-character}, we define the braiding character and show that it depends only on $\dim(V) \in (\ounit_{\cD})^{\B\TT}$, or more precisely, on the restriction of the $\TT$-action to finite subgroups.

    %%%%%%%%%%%%%%%%%%%%%%%%%%%%%%%%%%%%%%%%%%%%%%%%%%%%%%%%%%%%%%%%%%%%%%%%%%%%%%%%    
    \subsection{Braiding of objects}
    \label{subsec:braiding}
    %%%%%%%%%%%%%%%%%%%%%%%%%%%%%%%%%%%%%%%%%%%%%%%%%%%%%%%%%%%%%%%%%%%%%%%%%%%%%%%% 
        For $V \in \cD$, we define its $\degree$-th permutation representation as $\Tm V \coloneqq V \om \in \cD^{\B\Sm}$. These $\Sm$-actions hold a lot of the information of the higher monoidal structure of $\cD$. In this subsection we begin our study of these representations.

        \begin{notation}
            We denote by $\MM = \Fin\core = \bigsqcup_{\degree} \B\Sm$ the free commutative monoid.\footnote{The first author strongly believes we should have a common blackboard letter for the free commutative monoid, which is the higher analog of the natural numbers. He argues that $\mathbb{M}$ is great choice for a few reasons: M is close to N, asserting the relation between $\MM$ and $\NN$; M can stand for \quotes{Monoid}; and most importantly, the letter Z (as in $\ZZ$) is the first letter in the German word for a number --- \quotes{Zahl}. The letter M is the first letter in the German word for a set --- \quotes{Menge}. The second author does not care and would have preferred we would not have presented a new notation.} 
        \end{notation}

        % \begin{definition}
        %     Let $\cC \in \CMon(\Cat)$. Let $X \in \cD$. The braiding of $X$ is the unique symmetric monoidal map $\T X \colon \MM \to \cC$ choosing $X$.
        % \end{definition}

        \begin{definition}
            Let $\cD \in \CAlg(\PrL)$. We define the braiding functor as the composition
            \begin{equation*}
                \T \colon \cD \xto{(-)\shift{1}} \cD[\MM] \xto{\mrm{Fr}_{\EE_{\infty}}} \CAlg(\cD[\MM]) \xto{\fgt} \cD[\MM].
            \end{equation*}
        \end{definition}
        Recall that $\cD[\MM] \simeq \FunDay(\MM, \cD)$, therefore $\CAlg(\cD[\MM]) \simeq \Funolax(\MM, \cD)$ is the category of lax symmetric monoidal functors. The braiding of $X \in \cD$ is the free lax symmetric monoidal functor $\MM \to \cD$ generated by $X\shift{1}$. We show it is actually symmetric monoidal:
        
        \begin{lemma}
            Let $X\in \cD$. Then $\T X$ is the unique symmetric monoidal functor $\MM \to \cD$ choosing $X$.
        \end{lemma}

        \begin{proof}
            By definition
            \begin{equation*}
                \T X = \bigsqcup_{k} (X\shift{1}^{\otimes_{\Day} k})_{h\Sm[k]},
            \end{equation*}
            where the unit is the $0$-th summand map and the multiplication is induced by
            \begin{equation*}
                (X\shift{1}^{\otimes_{\Day} k})_{h\Sm[k]} \otimes_{\Day} (X\shift{1}^{\otimes_{\Day} \ell})_{h\Sm[\ell]} \to (X\shift{1}^{\otimes_{\Day} (k+\ell)})_{h\Sm[k+\ell]}.
            \end{equation*}
            $X\shift{1}$ is the functor $X\shift{0}$ tensored with the image of $1\in \MM$ under the Yoneda embedding $\MM \to \cD[\MM]$. Therefore, as the Yoneda map is symmetric monoidal, $X\shift{1}^{\otimes_{\Day} k}$ is $X\om[k]\shift{0}$ tensored with the image of $k \in \MM$ under the Yoneda embedding, i.e.\
            \begin{equation*}
                X\shift{1}^{\otimes_{\Day} k} = X\om[k][\Sm[k]]\{k\},
            \end{equation*}
            where, for $Y\in \cC^{\B\Sm[k]}$, $Y\{k\}$ is the functor sending $k\in\MM$ to $Y$ and $m\neq k$ to $\emptyset$.
            
            In particular
            \begin{equation*}
                (X\shift{1}^{\otimes_{\Day} k})_{h\Sm[k]} = X\om[k]\{k\}.
            \end{equation*}
            The multiplication is then induced by the counit map
            \begin{equation*}
                X\om[k]\{k\} \otimes_{\Day} X\om[\ell]\{\ell\} \simeq \Ind_{\Sm[k] \times \Sm[\ell]}^{\Sm[k+\ell]}X\om[(k+\ell)] \{k+\ell\} \xto{c} X\om[(k+\ell)] \{k+\ell\},
            \end{equation*}
            which is the mate of the isomorphism
            \begin{equation*}
                X\om[k]\{k\} \boxtimes X\om[\ell]\{\ell\} \isoto X\om[k+\ell]\{k,\ell\} \qin \Fun(\MM \times \MM, \cD).
            \end{equation*}
            The mate of the Day convolution is exactly the structure of the lax monoidaility, which in our case is an isomorphism. That is, $\T X\shift{1}$ is multiplicative. It is also easily seen to be unital, therefore it is symmetric monoidal.

            The claim now follows as $\Tm[1] X\shift{1} = X$, and as $\MM$ is the free commutative monoid, there exists a unique symmetric monoidal functor $\MM \to \cD$ that evaluates to $X$ in 1.
        \end{proof}
        
        \begin{remark}
            The map of spaces
            \begin{equation*}
                \B\Sm \into \MM \xto{\T X} \cD
            \end{equation*}
            chooses the permutation representation $\Tm X = X\om$ with the corresponding $\Sm$-action. Therefore the braiding of $X$ contains the information of the commutation of $X$ with itself. The braiding of $X\sqcup Y$ contains information of the commutation of $X$ and $Y$. 
        \end{remark}

        \begin{corollary}\label{cor:braiding-sends-coproducts-to-tensors}
            The braiding sends finite coproducts to tensor products:
            \begin{equation*}
                \T (X \sqcup Y) \simeq \T X \otimes_{\Day} \T Y \qin \cD[\MM].
            \end{equation*}
        \end{corollary}

        \begin{proof}
            The braiding functor was defined as the composition
            \begin{equation*}
                \T \colon \cD \xto{(-)\shift{1}} \cD[\MM] \xto{\mrm{Fr}_{\EE_{\infty}}} \CAlg(\cD[\MM]) \xto{\fgt} \cD[\MM].
            \end{equation*}
            The first two functors are left adjoint, thus preserve coproducts. Coproducts in the commutative algebra category $\CAlg(\cD[\MM])$ is the tensor product, and $\fgt$ preserves tensor products.
        \end{proof}

        When $X$ is a Picard element we consider a simpler invariant, which by abuse of notation we also call the braiding:
        \begin{definition}
            Let $Z \in \cD\units$. Then the braiding $\T Z \colon \MM \to \cD$ factors through the map of connective spectra $\T Z \colon \SS \to \cD\units$, choosing $Z$. Taking connected covers we get the map
            \begin{equation*}
                \hchar_{Z} \colon \tau_{\ge 1} \SS \to \Sigma \ounit_{\cD}\units.
            \end{equation*}
        \end{definition}

        \begin{remark}\label{rmrk:Sm-action-from-braiding}
            Consider the map of spaces 
            \begin{equation*}
                j_m \colon \B\Sm \into \MM \to \SS \xto{-m} \SS.
            \end{equation*}
            Then this is a pointed map from a connected space, therefore it factors through the connective cover
            \begin{equation*}
                j_m \colon \B\Sm \to \tau_{\ge 1} \SS.
            \end{equation*}
            Composing the braiding of $Z$ with $j_m$ gives the $\Sm$-action on $Z\om$ in the following sense:
            \begin{equation*}
                \B\Sm \xto{j_m} \tau_{\ge 1} \SS \xto{\hchar_Z} \B\ounit_{\cD}\units \simeq \B\Aut(Z\om).
            \end{equation*}
        \end{remark}
    
    %%%%%%%%%%%%%%%%%%%%%%%%%%%%%%%%%%%%%%%%%%%%%%%%%%%%%%%%%%%%%%%%%%%%%%%%%%%%%%%%
     \subsection{Equivariant monoidal trace and dimension}
    \label{subsec:equivariant-trace}
    %%%%%%%%%%%%%%%%%%%%%%%%%%%%%%%%%%%%%%%%%%%%%%%%%%%%%%%%%%%%%%%%%%%%%%%%%%%%%%%%
   
        Let $\cD\in \CMon(\Cat)$. Recall that an object $X \in \cD$ is dualizable if there exists $X\dual \in \cD$ and evaluation and coevaluation maps 
        \begin{equation*}
            \ev \colon X \otimes X\dual \to \ounit_{\cD}, \qquad \coev \colon \ounit_{\cD} \to X\dual \otimes X,
        \end{equation*}
        satisfying the so-called zigzag identities --- the compositions
        \begin{equation*}
            X \xto{\id \otimes \coev} X \otimes X\dual \otimes X \xto{\ev \otimes \id} X, \qquad 
            X\dual \xto{\coev \otimes \id} X\dual \otimes X \otimes X\dual \xto{\id \otimes \ev} X\dual
        \end{equation*}
        are isomorphic to the identity (in a coherent way). We denote the subcategory spanned by dualizable objects by $\cD\dbl$ and its core by $\cD\dblspace$. If $X \in \cD\dbl$ and $f \colon X \to X$ is an endomorphism, its trace $\tr(f) \in \End(\ounit_{\cD})$ is defined to be
        \begin{equation*}
            \ounit_{\cD} \xto{\coev} X \otimes X\dual \simeq X \otimes X\dual \xto{f \otimes \id} X \otimes X\dual \xto{\ev} \ounit_{\cD}.
        \end{equation*}
        The trace of the identity endomorphism of $X$ is called the dimension of $X$ and denoted $\dim(X) = \tr(\id_{X}) \in \End(\ounit_{\cD}) \simeq \ounit_{\cD}$. It admits a natural $\TT$-action. This can be seen either by the cobordism hypothesis at dimension $1$ or using that the trace is cyclic-invariant and using the cyclic bar construction.
        
        The space $(\cD\dblspace)^{\TT}$ is the space of a dualizable objects in $\cD$ with an automorphism, and therefore admits a monoidal trace map
        \begin{equation*}
            \tr \colon (\cD\dblspace)^{\TT} \to \ounit_{\cD}.
        \end{equation*}
    
        \begin{lemma}\label{lem:equivariant-trace}
            Let $\cD\in\CMon(\Cat)$ in which every object is dualizable. Then the trace map
            \begin{equation*}
                \tr \colon (\cD^\simeq)^{\TT} \to \ounit_{\cD}
            \end{equation*}
            is $\TT$-invariant. Its $\TT$-fixed points induces the dimension map, remembering the $\TT$-action,
            \begin{equation*}
                \dim \simeq \tr^{h\TT} \colon \cD^\simeq \to (\ounit_{\cD})^{\B\TT}.
            \end{equation*}
            Furthermore if $\cD \in \CAlg(\Catst)$ is stable, then this map factors as\footnote{Here $\TCm(\cD)$ is defined to be $\THH(\cD)^{h\TT}$.}
            \begin{equation*}
                \cD\core \to \TCm(\cD) \to (\ounit_{\cD})^{\B\TT}.
            \end{equation*}
        \end{lemma}
        
        \begin{proof}
            Write $(\cD^\simeq)^{\TT} \in \spc^{\B\TT}$ as the geometric realization of the cyclic space $(\cD^\simeq)^{\TT}_{\bullet} \in \Fun(\Lambda\op, \spc)$, where
            \begin{equation*}
                (\cD^\simeq)^{\TT}_n = \bigsqcup_{X_0,\cdots, X_n \in \cD} \Iso(X_0,X_1)\times \cdots \times \Iso(X_{n-1}, X_n) \times \Iso(X_n,X_0).
            \end{equation*}
            Note that this admits a map to the cyclic bar construction
            \begin{equation}\label{eqn:THH_spc}
                \THH_\spc(\PSh(\cD))_n = \bigsqcup_{X_0,\cdots, X_n \in \cD} \Map(X_0,X_1)\times \cdots \times \Map(X_{\chrHeight-1}, X_n) \times \Map(X_n,X_0) \tag{$\star$}
            \end{equation}
            whose geometric realization is $\THH_{\spc}(\PSh(\cD)) \in \spc^{\B\TT}$ --- the space valued $\THH$ (i.e.\ the dimension of $\PSh(\cD)$ in $\PrL$). The geometric realization functor induces a $\TT$-equivariant map
            \begin{equation*}
                (\cD^\simeq)^{\TT} \to \THH_{\spc}(\PSh(\cD)).
            \end{equation*}
            By base change, this admits a $\TT$-equivariant map to $\THH_{\PSh(\cD)}(\PSh(\cD)) \simeq \ounit_\cD$ with the trivial $\TT$-action (see e.g. \cite{Hoyois-Scherozke-Sibilla-2017-traces} or \cite{Carmeli-Cnossen-Ramzi-Yanovski-2022-characters}), which gives rise to the $\TT$-equivariant trace map
            \begin{equation*}
                \tr \colon (\cD\core)^{\TT} \to \THH_{\spc}(\PSh(\cD)) \to \ounit_{\cD}.
            \end{equation*}
            If $\cD \in \CAlg(\Catst)$, the same construction but for $\Ind(\cD)$, replacing mapping spaces with mapping spectra in \labelcref{eqn:THH_spc}, has geometric realization $\THH_{\Sp}(\Ind(\cD)) = \THH(\cD)$. Therefore, the trace factors as 
            \begin{equation*}
                \tr \colon (\cD\core)^{\TT} \to \THH(\cD)\to \ounit_{\cD}.
            \end{equation*}
            In both cases it is clear that taking fixed points induces the dimension map, and if $\cD \in \CAlg(\Catst)$, the dimension factors through
            \begin{equation*}
                \dim \colon \cD\core \to \TCm(\cD) \to (\ounit_{\cD})^{\B\TT}.
            \end{equation*}
        \end{proof}

        \begin{corollary}\label{cor:dim-splits-exact-sequences}
             Let $\cD\in \CAlg(\Catst)$ in which every object is dualizable and 
             \begin{equation*}
                X \to Y \to Z
             \end{equation*}
             be a fiber sequence. Then 
             \begin{equation*}
                 \dim(Y) = \dim(X) + \dim(Z) \qin (\ounit_\cD)^{\B\TT}.
             \end{equation*}
        \end{corollary}
        
        \begin{proof}
            By \cref{lem:equivariant-trace} the dimension map factors as
            \begin{equation*}
                \dim \colon \cD\core \to \TCm(\cD)\to (\ounit_\cD)^{\B\TT}.
            \end{equation*}
            The canonical map $\cD\core \to \TCm(\cD)$ factors through the cyclotomic trace of connective $\K$-theory $\K(\cD) \to \TCm(\cD)$ (see \cite[\textsection~10.3]{Blumberg-Gepner-Tabuada-2013-K-theory}), therefore splits exact sequences.
        \end{proof}

        \begin{corollary}
            Let $\cD\in \CAlg(\PrL)$. Then the trace map
            \begin{equation*}
                \tr \colon (\cD\dblspace)^{\TT} \to \ounit_{\cD}
            \end{equation*}
            is $\TT$-invariant. Its $\TT$-orbits induces the dimension map, remembering the $\TT$-action,
            \begin{equation*}
                \dim \simeq \tr^{h\TT} \colon \cD\dblspace \to (\ounit_{\cD})^{\B\TT}.
            \end{equation*}
            Furthermore if $\cD \in \CAlg(\Prst)$ is stable, then the above map splits exact sequences
        \end{corollary}

        \begin{corollary}\label{cor:no-action-on-perf}
            Let $R$ be a commutative ring spectrum and $M$ be a finite $R$-module. Then the $\TT$-action on $\dim(M)$ is trivial. 
        \end{corollary}

        \begin{proof}
            Every finite module is constructed by extensions from the unit $R$. The claim now follows by \cref{cor:dim-splits-exact-sequences}.
        \end{proof}

        \begin{corollary}\label{cor:no-action-on-dim-En}
            The circle action on $\dim(M)$ is trivial for every $M \in (\ModEn)\dbl$. 
        \end{corollary}
        \begin{proof}
            By the proof of \cite[Proposition~10.11]{Mathew-2016-Galois}\footnote{The proposition in \cite{Mathew-2016-Galois} states that $(\ModEn)\dbl \simeq \Perf(\En)$, but the proof shows that $\pi_* M$ is finitely generated over $\pi_* \En$ for any $M \in (\ModEn)\dbl$, which implies that $M$ is actually finite.}, $(\ModEn)\dbl$ is isomorphic to the category of finite $\En$-modules in $\Sp$. The claim then follows by \cref{cor:no-action-on-perf}.        
        \end{proof}

        We recall now the definition of a character of a local system as studied in \cite{Ponto-Schulman-2014-induced-character}, \cite{Hoyois-Scherozke-Sibilla-2017-traces}, \cite{Hoyois-Safrobov-Sherozke-Sibila-2021-GRR}, and \cite{Carmeli-Cnossen-Ramzi-Yanovski-2022-characters}:
        \begin{definition}\label{def:character}
            Let $\cD \in \CAlg(\PrL)$ and $X \in \spc$. Then define
            \begin{equation*}
                \chi \colon \L X \times (\cD\dblspace)^X = \Map(\TT, X) \times (\cD\dblspace)^X \to (\cD\dblspace)^{\TT} \xto{\tr} \ounit_{\cD}.
            \end{equation*}
            For $V \in (\cD\dbl)^{X}$, denote the evaluation map by
            \begin{equation*}
                \chi_V \colon \L X \to \ounit_{\cD}
            \end{equation*}
            and call it the character of $V$.

            Equivalently, $V \in (\cD\dbl)^X$ defines a functor $V \colon \cD[X] \to \cD$, which is an internal left adjoint functor in $\Mod_{\cD}$ (\cite[Corollary~4.32]{Carmeli-Cnossen-Ramzi-Yanovski-2022-characters}). $\THH_{\cD} \colon \Mod_{\cD} \to \cD$ is functorial with respect to internal left adjoints, and the character of $V$ agrees with the image under $\THH_{\cD}$ of this functor (\cite[Proposition~5.14]{Carmeli-Cnossen-Ramzi-Yanovski-2022-characters}). That is
            \begin{equation*}
                \chi_{V} \colon \ounit_{\cD}[\L X] \simeq \THH_{\cD}(\cD[X]) \xto{\THH_{\cD}(V)} \THH_{\cD}(\cD) \simeq \ounit_{\cD}.
            \end{equation*}
        \end{definition}

        %%%%%%%%%%%%%%%%%%%%%%%%%%%%%%%%%%%%%%%%%%%%%%%%%%%%%%%%%%%%%%%%%%%%%%%%%%%%%%%%
        \subsubsection*{Trace of an automorphism of an invertible object}
        
        When restricting to invertible objects, rather than all dualizable objects, the trace and dimension maps become maps of connective spectra
        \begin{equation*}
            \begin{split}
                \tr \colon (\cD\units)^{\TT} \to \ounit_{\cD}\units \qin (\cnSp)^{\B\TT}, \\
                \dim \colon \cD\units \to (\ounit_{\cD}\units)^{\B\TT} \qin \cnSp.
            \end{split}
        \end{equation*}
        We can replace the Picard spectrum of $\cD$ by any connective spectrum $X$ (e.g.\ $X$ itself is a symmetric monoidal category and $X\units \simeq X$), and respectively $\ounit_{\cD}\units$ by $\Omega X$, therefore getting the natural transformations
        \begin{equation*}
            \begin{split}
                & \tr \colon (-)^{\TT} \To \Omega(-) \qin \Fun(\cnSp, (\cnSp)^{\B\TT}), \\
                & \dim \colon \id \To \Omega(-)^{\B\TT} \qin \Fun(\cnSp, \cnSp).
            \end{split}
        \end{equation*}

        These functors are representable in $(\cnSp)^{\B\TT}$ and in $\cnSp$:
        \begin{align*}
            & (-)^{\TT} = \hom(\SS[\TT], -), && \Omega(-) = \hom(\Sigma \SS, -), \\
            & \id = \hom(\SS, -), && \Omega(-)^{\B\TT} = \hom(\Sigma \SS[\B\TT], -).
        \end{align*}
        We now describe the maps representing the trace and dimension maps:
        \begin{equation*}
            \begin{split}
                & \tr^* \colon \Sigma \SS \to \SS[\TT] \qin (\cnSp)^{\B\TT}, \\
                & \dim^* \colon \Sigma \SS[\B\TT] \to \SS \qin \cnSp,
            \end{split}
        \end{equation*}
        and derive conclusions from that description. Note that, as the dimension map is the $\TT$-fixed points of the trace map, the representing map of the dimension is the $\TT$-orbits of the representing map of the trace map --- $\dim^* = (\tr^*)_{h\TT}$.

        The standard orientation on $\TT$ chooses a generator of $\Map_{\cnSp}(\Sigma \SS, \redSS[\TT]) \simeq \ZZ$, which is a canonical isomorphism $\Sigma \SS \to \redSS[\TT]$. The cofiber of the map
        \begin{equation*}
            \Sigma \SS \isoto \redSS[\TT] \to \SS[\TT]
        \end{equation*}
        is the canonical map $\SS[\TT] \to \SS$ induced by the terminal map $\TT \to \pt$. Any choice of a point in $\TT$ gives a splitting $\SS \to \SS[\TT]$, therefore we get an identification
        \begin{equation*}
            \SS \oplus \Sigma \SS \isoto \SS[\TT].
        \end{equation*}

        \begin{lemma}\label{lem:trace-is-eta+id}
            On the level of connective spectra, the map $\tr^*$ 
            \begin{equation*}
                \tr^* \colon \Sigma \SS \to \SS[\TT] \simeq \SS \oplus \Sigma \SS
            \end{equation*}
            is given by $(\eta, \id)$.
        \end{lemma}
        \begin{proof}
            Composing the trace map with the map $\id_{(-)} \colon \cD\units \to (\cD\units)^{\TT}$ sending $Z \in \cD\units$ to $(Z, \id_Z)$ we get the dimension map, not remembering the $\TT$-action:
            \begin{equation*}
                \dim \colon \cD\units \xto{\id_{(-)}} (\cD\units)^{\TT} \xto{\tr} \ounit_{\cD}\units.
            \end{equation*}
            By \cite[Proposition~3.20]{CSY-cyclotomic}, this map is represented by the map $\eta \colon \Sigma \SS \to \SS$. But the map $\id_{(-)}$ is represented by the natural map $\SS[\TT] \to \SS$, which under the identification $\SS[\TT] \simeq \SS \oplus \Sigma \SS$ corresponds to the projection $\SS \oplus \Sigma \SS \onto \SS$. That is, the composition
            \begin{equation*}
                \Sigma \SS \xto{\tr^*} \SS \oplus \Sigma \SS \onto \SS
            \end{equation*}
            is given by $\eta$.

            Now, as $\Aut(\ounit_{\cD}) = \ounit_{\cD}\units$ and $(\cD\units)^{\TT} = \{(Z, f) \mid f \in \Aut(Z)\}$, we have a natural map
            \begin{equation*}
                (\ounit_{\cD}, -) \colon \ounit_{\cD}\units \to (\cD\units)^{\TT}
            \end{equation*}
            sending $u \in \ounit_{\cD}\units \simeq \Aut(\ounit_{\cD})$ to $(\ounit_{\cD}, u)$. It is the fiber at $\ounit_{\cD}$ of the forgetful map
            \begin{equation*}
                \begin{split}
                    (\cD\units)^{\TT} & \mapsto \cD\units \\
                    (Z, f) & \mapsto Z,
                \end{split}
            \end{equation*}
            which is represented by $e^* \colon \SS \to \SS[\TT]$ for a choice of a base point $e \in \TT$. Therefore, the map $(\ounit_{\cD}, -)$ is represented by the cofiber of $e^*$, which corresponds to the projection $\SS \oplus \Sigma \SS \to \Sigma \SS$. The composition with the trace map sends $u\in \ounit_{\cD}\units$ to its trace as an automorphism of $\ounit_{\cD}$, which is again $u$. Therefore the composition
            \begin{equation*}
                \Sigma \SS \xto{\tr^*} \SS \oplus \Sigma \SS \onto \Sigma \SS
            \end{equation*}
            is the identity.
        \end{proof}

        \begin{proposition}\label{prop:trace-is-counit}
            The map $\tr^* \colon \Sigma \SS \to \SS[\TT]$ is identified with the composition
            \begin{equation*}
                \Sigma \SS \simeq \SS[\TT]^{h\TT} \xto{c} \SS[\TT]
            \end{equation*}
            where $c$ is the canonical map from the fixed points.
        \end{proposition}
        \begin{proof}
            By \cref{lem:equivariant-trace}, the map $\tr^* \colon \Sigma \SS \to \SS[\TT]$ is $\TT$-invariant. By the universal property of fixed points, it factors as
            \begin{equation*}
                \Sigma \SS \xto{(\tr^*)^{h\TT}} \SS[\TT]^{h\TT} \xto{c} \SS[\TT].
            \end{equation*}
            As $\SS[\TT]^{h\TT} \simeq \Sigma \SS$, the map $c$ on underlying spectra is of the form $(n, b\eta) \in \pi_0\SS \times \pi_1 \SS \cong \pi_1 \SS[\TT]$, where $n\in \ZZ$ and $b\in \{0,1\}$. The map $(\tr^*)^{h\TT} \colon \Sigma \SS \to \Sigma \SS$ is a multiplication by an integer $k$. Therefore, $\tr^*$ is of the form $(nk, bk\eta) \in \pi_1\SS[\TT]$. By \cref{lem:trace-is-eta+id}, $nk = 1$ and $bk = 1 \mod 2$. In particular $k = \pm 1$ and $(\tr^*)^{h\TT}$ is an isomorphism as needed.
        \end{proof}
   
        Recall the norm map of $\B\TT$-local system: If $X \in \cD^{\B\TT}$ there is a natural norm map
        \begin{equation*}
            \Nm \colon \Sigma X_{h\TT} \to X^{h\TT}.
        \end{equation*}
        We call the composition of the norm, with the natural map $X^{h\TT} \to X$, the \emph{transfer map} and denote it as
        \begin{equation*}
            \Tr \colon \Sigma X_{h\TT} \to X.
        \end{equation*}
        In particular for $\SS \in \Sp^{\B\TT}$ with the trivial $\TT$-action, we have a transfer map
        \begin{equation*}
            \Tr \colon \Sigma \SS[\B\TT] \to \SS.
        \end{equation*}
        
        \begin{proposition}\label{prop:dim-is-transfer}
            The map $\dim^* \colon \Sigma \SS[\B\TT] \to \SS$ is identified with the transfer map.
        \end{proposition}
        \begin{proof}
            We refer the reader to \cite{Cnossen-2023-twisted-ambi} for a modern discussion about dualizing objects, $\cD$-linear functors and norm maps. See also \cite{Klein-2001-dualizing-object} and \cite{Nikolaus-Scholze-2018-TC}.
            The transfer natural transformation is
            \begin{equation*}
                \Tr \colon (\DD_{\B\TT} \otimes -)_{h\TT} \xTo{\Nm} (-)^{h \TT} \xTo{c} \id \qin \FunL(\Sp^{\B\TT}, \Sp^{\B\TT}),
            \end{equation*}
            where $\DD_{\B\TT} = \SS[\TT]^{h\TT}$ is the dualizing local system of $\B\TT$. It is a natural transformation of colimit-preserving functors, it is therefore determined by its value on $\SS[\TT]$. 
            The transfer map on $\SS[\TT]$ is the map
            \begin{equation*}
                (\DD_{\B\TT} \otimes \SS[\TT])_{h\TT} \isoto \SS[\TT]^{h\TT} \xto{c} \SS[\TT].
            \end{equation*}
            By \cref{prop:trace-is-counit}, it agrees with $\tr^*$. 
            Now, as $\SS = \SS[\TT]_{h\TT}$, the transfer map on $\SS$ is the $\TT$-orbits of $\tr^*$, i.e. $\dim^*$.
        \end{proof}

    %%%%%%%%%%%%%%%%%%%%%%%%%%%%%%%%%%%%%%%%%%%%%%%%%%%%%%%%%%%%%%%%%%%%%%%%%%%%%%%%
    \subsection{The braiding character}
    \label{subsec:braiding-character}
    %%%%%%%%%%%%%%%%%%%%%%%%%%%%%%%%%%%%%%%%%%%%%%%%%%%%%%%%%%%%%%%%%%%%%%%%%%%%%%%%
        Recall the braiding functor $\T \colon \cD \to \cD[\MM]$. Given $V \in \cD\dbl$, $\T V \colon \MM \to \cD$ induces an internal left adjoint functor $\cD[\MM] \to \cD$. Applying $\THH_{\cD}$ we get the usual character 
        \begin{equation*}
            \chi_{\T V} \colon \ounit_{\cD}[\L\MM] \to \ounit_{\cD} \qin \cD,
        \end{equation*}
        that at each $\L\B\Sm$ corresponds to the character $\chi_{\Tm V}$ of the $\Sm$-representation $\Tm V = V\om$.
        As $\T V$ was a symmetric monoidal map, the character lifts to a map of commutative algebras.
        We now notice, that by the natural grading on $\MM$, we can lift this character to a graded character
        \begin{equation*}
            \mscr{X}_{\T V} \colon \ounit_{\cD}[\L \MM] \to \ounit_{\cD}[t^{\pm 1}] \qin \CAlg(\cD)
        \end{equation*}
        that at each $\L\B\Sm$ is $\chi_{\Tm V} \, t^{\degree}$.

        \begin{definition}\label{def:braiding-character-general}
            Let $\cD \in \CAlg(\PrL)$ and $V \in \cD\dbl$. Define the braiding character of $V$ as the map
            \begin{equation*}
                \mscr{X}_{\T V} \colon \ounit_{\cD}[\L \MM] \to \ounit_{\cD}[\ZZ] \qin \CAlg(\cD)
            \end{equation*}
            obtained by applying $\THH_{\cD}$ to the braiding
            \begin{equation*}
                \T V\shift{1} \colon \cD[\MM] \to \Gr_{\ZZ}\cD
            \end{equation*}
            of $V\shift{1} \in \Gr_{\ZZ} \cD$.
            Equivalently, it is the map of commutative monoids
            \begin{equation*}
                \mscr{X}_{\T V} \colon \L\MM \to \ounit_{\cD}[\ZZ] \qin \CMon.
            \end{equation*}
        \end{definition}
        
        \begin{remark}\label{rmrk:character-from-graded-character}
            The colimit functor 
            \begin{equation*}
                \colim \colon \Gr_{\ZZ}\cD \to \cD
            \end{equation*}
            is symmetric monoidal and sends $V\shift{1}$ to $V$. It induces on $\THH_{\cD}$ the usual evaluation at 1 map $\ounit_{\cD}[\ZZ] \to \ounit_{\cD}$. Therefore, the composition
            \begin{equation*}
                \ounit_{\cD}[\L\MM] \xto{\mscr{X}_{\T V}} \ounit_{\cD}[\ZZ] \to \ounit_{\cD}
            \end{equation*}
            agrees with the usual character of the braiding $\chi_{\T V}$.
        \end{remark}

        This graded character, albeit more aesthetically pleasing, does not contain more data than the usual character:
        \begin{lemma}\label{lem:map-factors-constant}
            Let $M,N$ be commutative monoids and $f \colon M \to \ounit_{\cD}[N]$ a map of commutative monoids. Assume that $N$ is discrete, so $\ounit_{\cD}[N] \simeq \bigsqcup_N \ounit_{\cD}$, and that $f$ restricted to any connected component of $m\in M$ is non-trivial exactly at one degree $g(m)$. Then $f$ is isomorphic to the composition
            \begin{equation*}
                M \xto{\Sigma_N f, g} \ounit_{\cD} \times N \to \ounit_{\cD}[N],
            \end{equation*}
            where $\sum_N f$ is the composition $M \xto{f} \ounit_{\cD}[N] \to \ounit_{\cD}$ and the right arrow is the assembly map of the forgetful functor $\CAlg(\cD) \to \CMon$.
        \end{lemma}
        \begin{proof}
            By the assumption, $f$ factors as
            \begin{equation*}
                f \colon M \xto{h, g} \ounit_{\cD} \times N \to \ounit_{\cD}[N],
            \end{equation*}
            for some $h \colon M \to \ounit_{\cD}$.
            Now, the composition
            \begin{equation*}
                \ounit_{\cD} \times N \to \ounit_{\cD}[N] \to \ounit_{\cD}
            \end{equation*}
            is just the projection, so $h = \sum_N f$.
        \end{proof}

        \begin{notation}
            Let $\deg \colon \L \MM \to \ZZ$ be the map of commutative monoids, sending $\L\B\Sm$ to $\{\degree\}\subseteq \ZZ$.
        \end{notation}

        \begin{corollary}\label{cor:braiding-character-factors-constant}
            Let $\cD \in \CAlg(\PrL)$ and $V \in \cD\dbl$. The braiding character $\mscr{X}_{\T V}$ factors in $\CMon$ as
            \begin{equation*}
                \L \MM \xto{\chi_{\T V}, \deg} \ounit_{\cD} \times \ZZ \to \ounit_{\cD}[\ZZ],
            \end{equation*}
            where the right arrow is the assembly map of the forgetful functor $\CAlg(\cD) \to \CMon$.
        \end{corollary}

        \begin{proof}
            This follows directly from \cref{lem:map-factors-constant} with $g = \deg$, and from \cref{rmrk:character-from-graded-character}.
        \end{proof}

        Although the braiding character does not contain more data than the character of the braiding, in \cref{sec:graded-braiding} we will see that when $\cD = \Gr^{\Zchar}_A \cC$ for some abelian group $A$, this character can be interpreted internally, providing interesting insights on braiding of twisted graded categories. The rest of this subsection is devoted to showing that the braiding character is quite computable and depends only on $\dim(V) \in (\ounit_{\cD})^{\B\TT}$.

        %%%%%%%%%%%%%%%%%%%%%%%%%%%%%%%%%%%%%%%%%%%%%%%%%%%%%%%%%%%%%%%%%%%%%%%%%%%%%%%%
        \subsubsection*{Independence of the braiding character}

       %  As we are working with dualizable objects, it will be convenient to go to the universal symmetric monoidal category generated by a dualizable object. The cobordism hypothesis, as conjectured by Baez--Dolan \cite{Baez-Dolan-1995-cobordism}, suggests in particular that the universal symmetric monoidal category with a dualizable object, is the category of stably framed cobordisms $\Bordn[1]$. In \cite{Lurie-2008-cobordism}, Lurie sketches a proof of the cobordism hypothesis for $n$-dualizable objects in symmetric monoidal $n$-categories. In \cite{Harpaz-2012-cobordism}, Harpaz proves completely the case for $n = 1$. Using the cobordism hypothesis, Lurie proves a more general cobordism hypothesis with tangential structures \cite[Theorem~2.4.18]{Lurie-2008-cobordism}, but we will need only the slightly less general theorem, when the space is connected:
       % \begin{theorem}[{The cobordism hypothesis in dimension 1 \cite{Harpaz-2012-cobordism}}]\label{thm:cobordism-hyp-in-dim-1}
       %     For any symmetric monoidal category $\cD$, the evaluation at the point defines an isomorphism of spaces
       %     \begin{equation*}
       %         \Fun^{\otimes}(\Bordn[1], \cD) \simeq \cD\dbl.
       %     \end{equation*}
       % \end{theorem}

        Recall that
        \begin{equation*}
             \L\MM = \bigsqcup_{\degree} \L\B\Sm \simeq \bigsqcup_{\degree} \bigsqcup_{\sigma \in \Sm / \conj} \B\Csigma.
        \end{equation*}
        We therefore begin by recalling the centralizers and stabilizers of elements in $\Sm$.
    
        \begin{notation}
            Let $\sigma\in \Sigma_m$. We denote by
            \begin{enumerate}
                \item $\numcyc$ --- the number of cycles in the decomposition of $\sigma$ to disjoint cycles.
                \item $\Nk$ --- The number of $k$ cycles in the decomposition to disjoint cycles of $\sigma$.
            \end{enumerate}
        \end{notation}
        
        \begin{lemma}\label{lem:decomposition-of-wreath}
            The centralizer of $\sigma\in \Sm$  is
            \begin{equation*}
                \Csigma \simeq \bigsqcap_{k=1}^{\degree} (\ZZ/k \wr \Sm[\Nk]) = \bigsqcap_{k=1}^{\degree} (\ZZ/k)^{\Nk} \rtimes \Sm[\Nk].
            \end{equation*}
            The groups $\ZZ/k$ are generated by the $k$-cycles of $\sigma$ and the groups $\Sn[\Nk]$ permutes the different $k$-cycles.
        \end{lemma}

        \begin{proof}
            This result is classical. See e.g.\ \cite{Bernhardt-Niemeyer-Rober-Wollenhaupt-2022-wreath} for a proof.
        \end{proof}

        \begin{corollary}
            The free loops space of $\B\Sm$ is
            \begin{equation*}
                \L\B\Sm  = \bigsqcup_{\sigma \in \Sm / \mrm{conj}} \bigsqcap_{k=1}^{\degree} \B(\ZZ/k \wr \Sigma_{\Nk}).
            \end{equation*}
        \end{corollary}

        \begin{definition}\label{def:stabilizer-action}
            Let $\cD$ be a symmetric monoidal category.
            Let $\sigma \in \Sm$ and $W \in \cD \dbl$. Then $(\dim W)^{\numcyc} \in \ounit_{\cD}$ admits a natural $\Csigma$-action, which we describe as follows:
            By \cref{lem:decomposition-of-wreath}, 
            \begin{equation*}
                \Csigma \simeq \bigsqcap_{k = 1}^{\degree} ((\ZZ / k)^{\Nk} \rtimes \Sigma_{\Nk}).
            \end{equation*}
            The component $(\ZZ / k)_{\tau} \subseteq (\ZZ / k)^{\Nk}$ corresponding to a cycle $\tau \subseteq \sigma$ of length $k$, acts on the component $(\dim (W))_{\tau}$ by restriction along the natural map $\ZZ/k \into \TT$. $\Sigma_{\Nk}$ acts on $\bigsqcap_{\substack{\tau \subseteq \sigma, \\ |\tau| = k}} (\dim W)_{\tau}$ by permutation.
    
            We denote this element with the $\Csigma$-action also by the name $(\dim W)^{\numcyc} \in \End(\ounit_{\cD})^{\B\Csigma}$.
        \end{definition}
    
        \begin{lemma}\label{lem:character-of-Tm}
            Let $\cD$ be a symmetric monoidal category and $W \in \cD\dbl$. Let $\T W \colon \MM \to \cD$ be the braiding of $W$. Then its character is the map
             \begin{equation*}
                \chi_{\T W} \colon \L\MM \to \ounit_{\cD},
            \end{equation*}
            that chooses $(\dim W)^{\numcyc}$ at each $\B\Csigma \subseteq \L\B\Sm$, with the action of \cref{def:stabilizer-action}.
        \end{lemma}

        \begin{proof}
            This can be proved in the universal case: The cobordism hypothesis, as conjectured by Baez--Dolan \cite{Baez-Dolan-1995-cobordism}, suggests in particular that the universal symmetric monoidal category with a dualizable object, is the category of framed cobordisms $\Bordn[1]$. In \cite{Lurie-2008-cobordism}, Lurie sketches a proof of the cobordism hypothesis for $n$-dualizable objects in symmetric monoidal $n$-categories. In \cite{Harpaz-2012-cobordism}, Harpaz proves completely the case for $n = 1$. That is, the evaluation at $+$ (the point with a positive framing) induces an isomorphism
            \begin{equation*}
                \Map^{\otimes}(\Bordn[1], \cD) \isoto \cD\dblspace.
            \end{equation*}
            
            In $\Bordn[1]$, the category of framed points and  framed $1$-dimensional bordisms between them as morphisms, the universal dualizable object is $+$. Its dual is $-$ (the point with a negative framing). $\Tm(+)$ is a collection of $\degree$ positively-framed points and $\Sm$ acts by permuting them. 
            The trace of $\sigma \in \Sm$ consists of one circle for every cycle in $\sigma$ (see \cref{fig:trace-of-(123)(45)} for an example). It is now clear that $\Sn[\Nk]$ acts by permuting the circles of the same length and that the $\ZZ/k$-action on each cycle is obtained by rotating the circle.
            \tikzset{every picture/.style={line width=0.75pt}} %set default line width to 0.75pt    

            \begin{figure}[ht]
                \centering
                \includegraphics[width=0.4\linewidth]{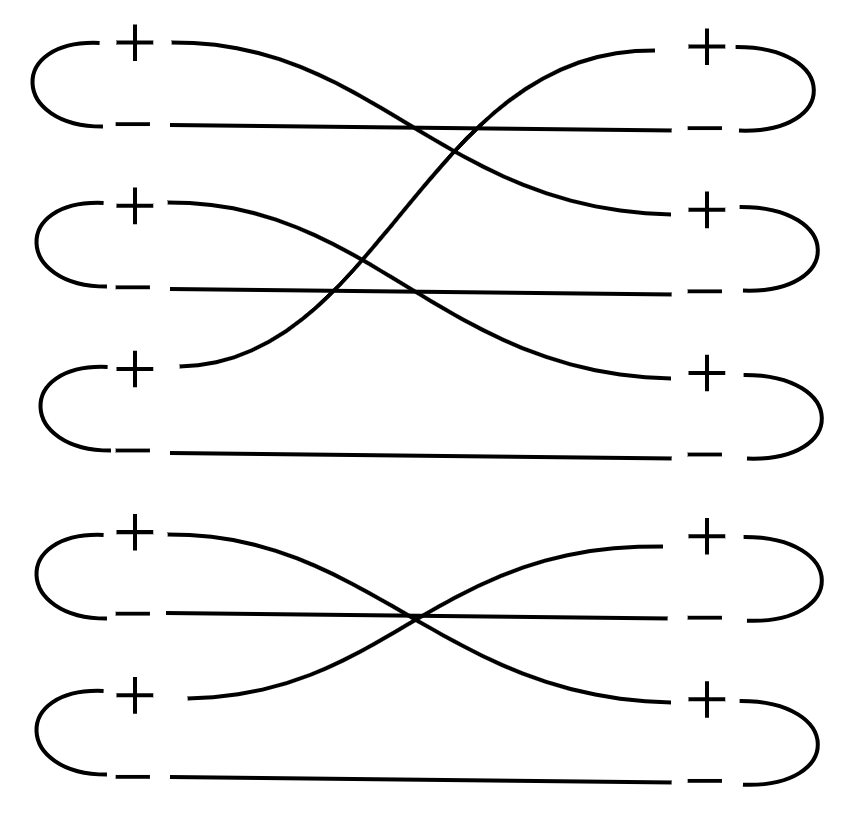}
                \caption{The trace of $(1,2,3)(4,5) \in \Sigma_5$ acting on $\Tm[5](+)$}
                \label{fig:trace-of-(123)(45)}
            \end{figure}
        \end{proof}

        Note that this action only depends on the restriction of the $\TT$-action to any finite subgroup, i.e. on the image of $\dim(V)$ under the restriction map $(\ounit_{\cD})^{\B\TT} \to (\ounit_{\cD})^{\vee_k \B\Ck}$.

        We now simply get the following lemma, which was independently proven in \cite[Lemma~4.7]{Ramzi-2025-endomorphisms-of-THH}.

        \begin{theorem}\label{thm:braiding-depends-only-on-dim}
            Let $\cD \in \CAlg(\PrL)$ and $V \in \cD\dbl$.
            Then the braiding character of $V$ depends only on $\dim V \in (\ounit_\cC)^{\vee_k \B\Ck}$.
        \end{theorem} 
        
        \begin{proof}
            From \cref{lem:character-of-Tm}, $\chi_{\T W}$ depends only on $\dim W$. The result now follows by \cref{cor:braiding-character-factors-constant}.
        \end{proof}

        By the identification of the $\TT$-action on the dimension with the $\TT$-transfer map \cref{prop:dim-is-transfer}, when $V$ is invertible, we can rephrase this as follows

        \begin{corollary}\label{cor:braiding-depends-only-on-dim-as-transfer}
            Let $\cD \in \CAlg(\PrL)$ and $V \in \cD\units$. Then the braiding character of $V$ depends only on the map
            \begin{equation*}
                \Sigma \SS [\vee_k \B\Ck] \to \Sigma \SS[\B\TT] \xto{\Tr} \SS \xto{\T V} \cD\units. 
            \end{equation*}
        \end{corollary}

%%%%%%%%%%%%%%%%%%%%%%%%%%%%%%%%%%%%%%%%%%%%%%%%%%%%%%%%%%%%%%%%%%%%%%%%%%%%%%%%
%%%%%%%%%%%%%%%%%%%%%%%%%%%%%%%%%%%%%%%%%%%%%%%%%%%%%%%%%%%%%%%%%%%%%%%%%%%%%%%%    
\section{Braiding and graded categories}
\label{sec:graded-braiding}
%%%%%%%%%%%%%%%%%%%%%%%%%%%%%%%%%%%%%%%%%%%%%%%%%%%%%%%%%%%%%%%%%%%%%%%%%%%%%%%%
%%%%%%%%%%%%%%%%%%%%%%%%%%%%%%%%%%%%%%%%%%%%%%%%%%%%%%%%%%%%%%%%%%%%%%%%%%%%%%%%

    In \cref{sec:twisted-graded-categories}, we constructed several symmetric monoidal structures that agree $\EE_1$-monoidally. In \cref{sec:braiding}, we discussed the braiding and the braiding characters as invariants of such structures. We now specialize and study these invariants in the case of twisted graded categories.
    
    We begin in \cref{subsec:graded-braiding} by showing that the braiding of $\ounit_\cC\shift{a}$ for $a \in A$ is almost a complete invariant of twisted graded categories: it recovers the symmetric monoidal structure, but not the $\EE_1$-isomorphism with the Day convolution structure. In the universal case $A = \ZZ$, it recovers both.
    We also identify when two twisted graded categories are the same, forgetting the $\EE_1$-trivialization, in term of the braiding.
    In the special case where $\Mod_{\cC}$ is $(\chrHeight+1)$-connected and that $\ounit_{\cC}$ is truncated, we show that the braiding of $\ounit_{\cC}\shift{a}$ depends only on a map of spaces $\B\Sinfty \to \Sigma \ounit_{\cC}\units$.
    Then, in \cref{subsec:graded-braiding-character}, we study the braiding character and show that, for $\ZZ$-graded categories and a dualizable object $V\in \cC\dbl$, the braiding character of $V\shift{1}$ agrees with the image under $\THH_\cC$ of the braiding of $V\shift{1}$. 
    We also examine the behavior of this character in graded categories for dualizable objects with finite support.
    Finally, in \cref{subsec:exterior-algebras}, we analyze the levelwise orbits of the braiding, yielding a graded-commutative algebra that we call the twisted exterior algebra. Using the previous results, we derive a generating function comparing the dimensions of twisted exterior algebras for different maps $\ZZ\Enull\to \cC\units$.

    %%%%%%%%%%%%%%%%%%%%%%%%%%%%%%%%%%%%%%%%%%%%%%%%%%%%%%%%%%%%%%%%%%%%%%%%%%%%%%%%
    \subsection{Braiding in graded categories}
    \label{subsec:graded-braiding}
    %%%%%%%%%%%%%%%%%%%%%%%%%%%%%%%%%%%%%%%%%%%%%%%%%%%%%%%%%%%%%%%%%%%%%%%%%%%%%%%%
        Let $A \in \Ab$ and $\Zchar \colon A \to \Sigma \cC\units$ with an $\EE_1$-nullhomotopy. If the braiding of $\ounit_{\cC}\shift{a}$ is given to us for all $a\in A$, it is simple to compute the braiding of any object in $\Gr^{\Zchar}_A \cC$ with finite support. 
        We now see that these braidings actually determine $\Zchar$, forgetting the $\EE_1$-nullhomotopy, and in particular the whole symmetric monoidal structure.

        For any $a\in A$, we define a new map
        \begin{equation*}
            \Zchar \circ a \colon \ZZ \xto{a} A \xto{\Zchar} \Sigma \cC\units,
        \end{equation*}
        which admits a canonical $\EE_1$-nullhomotopy.

        As $\ounit_{\cC}\shift{a}$ and $\ounit_{\cC}\shift{1}$ are invertible in $\Gr^{\Zchar}_A \cC$ and $\Gr^{\Zchar \circ a}_{\ZZ} \cC$ respectively, we can think of their braiding as maps $\hchar_{\ounit_{\cC}\shift{a}}, \hchar_{\ounit_{\cC}\shift{1}} \colon \tau_{\ge 1} \SS \to \Sigma \ounit_{\cC}\units$.
        \begin{lemma}\label{lem:braiding-of-A-determined-by-Z}
            The braiding of $\ounit_{\cC}\shift{1}$ in $\Gr^{\Zchar \circ a}_\ZZ \cC$ agrees with the braiding of $\ounit_{\cC}\shift{a}$ in $\Gr^{\Zchar}_A \cC$.
        \end{lemma}
        
        \begin{proof}
            The map $a \colon \ZZ \to A$ is a map of spectra over $\Sigma \cC\units$ and therefore defines a symmetric monoidal map $a_! \colon \Gr^{\Zchar\circ a}_{\ZZ} \cC \to \Gr^{\Zchar}_A \cC$, sending $\ounit_{\cC}\shift{1}$ to $\ounit_{\cC}\shift{a}$. We therefore have a commutative diagram of symmetric monoidal maps
            \begin{equation*}
                \begin{tikzcd}
                    && {(\Gr^{\Zchar \circ a}_{\ZZ} \cC))\units} \\
                    \SS && {(\Gr^{\Zchar}_A \cC)\units.}
                    \arrow["{a_!}", from=1-3, to=2-3]
                    \arrow["{\T\ounit_{\cC}\shift{1}}", from=2-1, to=1-3]
                    \arrow["{\T\ounit_{\cC}\shift{a}}"', from=2-1, to=2-3]
                \end{tikzcd}
            \end{equation*}
            The claim follows by taking connected covers.
        \end{proof}

        We may therefore direct our attention to the case $A = \ZZ$.

        \begin{lemma}\label{lem:maps-from-Z-are-from-truncated-sphere}
            Let $X$ be a connective spectrum. There is a bijection
            \begin{equation*}
                \pi_0 \Map(\ZZ, \Sigma^2 X) \simeq \pi_0 \Map(\tau_{\ge 1} \SS, \Sigma X).
            \end{equation*}
        \end{lemma}

        \begin{proof}
            Consider the exact sequence of spectra
            \begin{equation*}
                \Sigma^{-1} \ZZ \to \tau_{\ge 1} \SS \to \SS
            \end{equation*}
            which induces the exact sequence of spectra
            \begin{equation*}
                 \hom(\SS, \Sigma X) \to \hom(\tau_{\ge 1} \SS, \Sigma X) \to \hom(\Sigma^{-1} \ZZ, \Sigma X).
            \end{equation*}
            The result now follows as $\pi_0$ and $\pi_{-1}$ of $\Sigma X = \hom(\SS, \Sigma X)$ vanish.
        \end{proof}
        
        \begin{corollary}
            Let $\cC \in \CAlg(\PrL)$. Then composing with $\Sigma^2\ounit_{\cC}\units \to \Sigma \cC\units$ defines a bijection
            \begin{equation*}
                \pi_0 \Map_{\cnSp}(\ZZ, \Sigma^2 \ounit_{\cC}\units) \simeq \pi_0 \Map\Enull(\ZZ, \Sigma \cC\units).
            \end{equation*}
            That is, a map $\ZZ \to \Sigma \cC\units$ with an $\EE_1$-nullhomotopy is equivalent to a map $\ZZ \to \Sigma^2 \ounit_{\cC}\units$.
        \end{corollary}
        \begin{proof}
            Since $\ZZ\Enull = \tau_{\ge 1} \SS$, 
            \begin{equation*}
                \Map\Enull(\ZZ, \Sigma \cC\units) \simeq \Map(\tau_{\ge 1} \SS, \cC\units).
            \end{equation*}
            Since $\tau_{\ge 1} \SS$ is connected, $\Map(\tau_{\ge 1} \SS, \cC\units) \simeq \Map(\tau_{\ge 1} \SS, \tau_{\ge 1} \cC\units) \simeq \Map(\tau_{\ge 1} \SS, \Sigma \ounit_{\cC}\units)$.
            Finally, by \cref{lem:maps-from-Z-are-from-truncated-sphere}, 
            \begin{equation*}
                \pi_0 \Map\Enull(\ZZ, \Sigma \cC\units) \simeq \pi_0 \Map(\tau_{\ge 1} \SS, \Sigma \ounit_{\cC}\units) \simeq \pi_0\Map(\ZZ, \Sigma^2 \ounit_{\cC}\units).
            \end{equation*}
        \end{proof}

        \begin{corollary}\label{cor:E1-nullhomotopy-and-surjection}
            The map $\pi_0\Map(\ZZ, \Sigma^2 \ounit_{\cC}\units) \to \pi_0\Map(\ZZ, \Sigma \cC\units)$ is surjective.
        \end{corollary}
        \begin{proof}
            It is enough to show that any map $\ZZ \to \Sigma \cC\units$ is $\EE_1$-nullhomotopic, or equivalently that the pointed map
            \begin{equation*}
                S^1 \simeq \B\ZZ \to \B\Sigma \cC\units
            \end{equation*}
            is null. This is true as $\B\Sigma\cC\units$ is simply connected.
        \end{proof}
        
        \begin{remark}\label{rmrk:kernel-of-forgetting-the-E1-nullhomotopy}
            The exact sequence $\Sigma \ounit_{\cC}\units \to \cC\units \to \Pic(\cC)$ induces an exact sequence
            \begin{equation*}
                \hom(\ZZ, \Sigma^2 \ounit_{\cC}\units) \to \hom(\ZZ, \Sigma\cC\units) \to \hom(\ZZ, \Sigma \Pic(\cC)).
            \end{equation*}
            The long exact sequence of homotopy groups together with \cref{cor:E1-nullhomotopy-and-surjection} imply there is a short exact sequence
            \begin{equation*}
                0 \to \Pic(\cC)/\Pic^{\mrm{str}}(\cC) \to \pi_0\hom(\ZZ, \Sigma^2 \ounit_{\cC}\units) \to \pi_0 \hom(\ZZ, \Sigma \cC\units) \to 0,
            \end{equation*}
            where $\Pic^{\mrm{str}}(\cC) = \pi_0 \hom(\ZZ, \cC\units)$ is the group of strict Picard elements.
            The map
            \begin{equation*}
                \Pic(\cC) / \Pic^{\mrm{str}}(\cC) \to \pi_0 \hom(\ZZ, \Sigma^2 \ounit_{\cC}\units) \simeq \pi_0 \hom(\tau_{\ge 1}\SS, \Sigma \ounit_{\cC}\units)
            \end{equation*}
            sends a Picard element $Z$ to its braiding $\hchar_{Z}$.
        \end{remark}

        \begin{definition}
            Let $\Zchar \colon \ZZ \to \Sigma^2\ounit_{\cC}\units$. Then by \cref{lem:maps-from-Z-are-from-truncated-sphere}, it is equivalent to a map 
            \begin{equation*}
                \hchar_{\Zchar} \colon \tau_{\ge 1}\SS \to \Sigma\ounit_{\cC}\units
            \end{equation*}
            such that the composition $\ZZ \to \Sigma \tau_{\ge 1} \SS \xto{\hchar_{\Zchar}} \Sigma^2\ounit_{\cC}\units$ is $\Zchar$.
        \end{definition}

        \begin{corollary}\label{cor:graded-trivial-iff-picard}
            Let $\Zchar \colon \ZZ \to \Sigma^2 \ounit_{\cC}\units$ be a map of connective spectra. The symmetric monoidal category $\Gr^{\Zchar}_{\ZZ} \cC$ is isomorphic to $\Gr_{\ZZ} \cC$ if and only if there is $Z \in \cC\units$ and an isomorphism $\hchar_{Z} \simeq \hchar_{\Zchar}$.
        \end{corollary}
        \begin{proof}
            Forgetting the $\EE_1$ isomorphism to $\Gr_{\ZZ} \cC$, the symmetric monoidal structure is determined by the map $\ZZ \xto{\Zchar} \Sigma^2 \ounit_{\cC}\units \to \Sigma \cC\units$. It is the Day symmetric monoidal structure if and only if this map is null. By \cref{rmrk:kernel-of-forgetting-the-E1-nullhomotopy}, the composition is nullhomotopic if and only if $\hchar_{\Zchar} \simeq \hchar_{Z}$ for some $Z \in \cC\units$.
        \end{proof}

        \begin{proposition}\label{prop:braiding-of-1<1>-determines-the-character}
            Let $\Zchar \colon \ZZ \to \Sigma^2 \ounit_{\cC}\units$. Then $\hchar_{\Zchar}$ is the braiding of $\ounit_{\cC}\shift{1}$ in $\Gr^{\Zchar}_\ZZ \cC$.
        \end{proposition}

        \begin{proof}
            By \cref{lem:cofiber-of-pic-Day}, there is a cofiber sequence
            \begin{equation*}
                \cC\units \to (\Gr^{\Zchar}_{\ZZ}\cC)\units \to \ZZ,
            \end{equation*}
            and by the proof of \cref{thm:twFun-are-all-lifts}, the cofiber map $\ZZ \to \Sigma \cC\units$ is $\Zchar$.

            Consider also the cofiber sequence of the connected cover
            \begin{equation*}
                \Sigma \ounit_{\cC}\units \to (\Gr^{\Zchar}_{\ZZ}\cC)\units \to \Pic(\Gr^{\Zchar}_{\ZZ} \cC),
            \end{equation*}
            which admits a natural map to the previous cofiber sequence.

            The element $\ounit_{\cC}\shift{1}\in(\Gr^{\Zchar}_{\ZZ}\cC)\units$ induces a map $\T \ounit_{\cC}\shift{1} \colon \SS \to (\Gr^{\Zchar}_{\ZZ}\cC)\units$ which, by definition, its connected cover is the braiding of $\ounit_{\cC}\shift{1}$. On $\pi_0$ it is a map $\ZZ \to \Pic(\Gr^{\Zchar}_{\ZZ} \cC)$, which is a retract of the natural map $\Pic(\Gr^{\Zchar}_{\ZZ} \cC) \to \ZZ$. 
            
            We summarize this data as the following diagram where the columns are cofiber sequences:
            \begin{equation*}
                \begin{tikzcd}
                    {\tau_{\ge 1}\SS} & {\Sigma\ounit_{\cC}\units} & {\cC\units} \\
                    \SS & {(\Gr^{\Zchar}_{\ZZ}\cC)\units} & {(\Gr^{\Zchar}_{\ZZ}\cC)\units} \\
                    \ZZ & {\Pic(\Gr^{\Zchar}_{\ZZ}\cC)} & \ZZ \\
                    {\Sigma\tau_{\ge 1}\SS} & {\Sigma^2\ounit_{\cC}\units} & {\Sigma\cC\units.}
                    \arrow["{\hchar_{\ounit_{\cC}\shift{1}}}", from=1-1, to=1-2]
                    \arrow[from=1-1, to=2-1]
                    \arrow[from=1-2, to=1-3]
                    \arrow[from=1-2, to=2-2]
                    \arrow[from=1-3, to=2-3]
                    \arrow["{\T \ounit_{\cC}\shift{1}}", from=2-1, to=2-2]
                    \arrow[from=2-1, to=3-1]
                    \arrow[equals, from=2-2, to=2-3]
                    \arrow[from=2-2, to=3-2]
                    \arrow[from=2-3, to=3-3]
                    \arrow["{\ounit_{\cC}\shift{1}}", from=3-1, to=3-2]
                    \arrow[curve={height=12pt}, equals, from=3-1, to=3-3]
                    \arrow[from=3-1, to=4-1]
                    \arrow[from=3-2, to=3-3]
                    \arrow[from=3-2, to=4-2]
                    \arrow["\Zchar"{description}, dashed, from=3-3, to=4-2]
                    \arrow["\Zchar", from=3-3, to=4-3]
                    \arrow["{\Sigma\hchar_{\ounit_{\cC}\shift{1}}}", from=4-1, to=4-2]
                    \arrow[from=4-2, to=4-3]
                \end{tikzcd}
            \end{equation*}
            We deduce that the composition $\ZZ \to \Sigma \tau_{\ge 1} \SS \xto{\Sigma \hchar_{\ounit_{\cC}\shift{1}}} \Sigma^2 \ounit_{\cC}\units$ is isomorphic $\Zchar$. Therefore, $\hchar_{\ounit_{\cC}\shift{1}} \simeq \hchar_{\Zchar}$.
        \end{proof}

        \begin{corollary}\label{cor:braiding-determines-monoidal-structure}
            The braiding of $\ounit_{\cC}\shift{1}$ determines the symmetric monoidal structure of $\Gr^{\Zchar}_\ZZ \cC$ as a $\cC$-linear symmetric monoidal category.
        \end{corollary} 

        For a general abelian group $A$, \cref{cor:braiding-determines-monoidal-structure} does not hold verbatim. A similar statement holds, reproducing the map $\Zchar$ while forgetting the $\EE_1$-nullhomotopy. In this case, however, one must also keep track of the relations among the braidings of the generators of $A$.

        Let $A$ be an abelian group and $\Zchar \in \Map^{\h,1}\Enull(A, \Mod_{\cC}\units)$, that is $\Zchar \colon A \to \Mod_{\cC}\units$ together with an $\EE_1$-nullhomotopy of the composition
        \begin{equation*}
            A \xto{\Zchar} \Mod_{\cC}\units \to \tau_{\le 2} \Mod_{\cC}\units.
        \end{equation*}
        Therefore, by \cref{cor:Thom-of-homotopy-null-map}, the Thom construction $\Gr_A^{\Zchar} \cC$ is equipped with an $\EE_0$-isomorphism $\Gr_A \cC \isoto \Gr_A^{\Zchar} \cC$ together with an $\EE_1$-lift on homotopy categories.
            
        \begin{definition}
            Let $w = (a_1,\dots, a_r) \in A^r$ be a word. By the above discussion, there is a canonical isomorphism
            \begin{equation*}
                i_w \colon \hchar_{\ounit_{\cC}\shift{a_1}} \cdots \hchar_{\ounit_{\cC}\shift{a_r}} \isoto \hchar_{\ounit_{\cC}\shift{a_1 + \cdots + a_r}} \qin \h\Gr_A^{\Zchar} \cC.
            \end{equation*}
        \end{definition}

        \begin{corollary}\label{cor:braiding-for-general-A-determined}
            Let $A = \langle S \mid R \rangle$ be an abelian group given by generators and relations. Let $\Zchar \in \Map^{\h,1}(A, \Mod_{\cC}\units) \simeq \Map^{\h, 1}(A, \Sigma^2 \ounit_{\cC}\units)$. 
            Then the set of braidings $\hchar_{\ounit_{\cC}\shift{a}}$ for $a\in S$ along with the isomorphisms $i_w$ for $w \in R$, determine the lift $\Zchar \colon A \to \Sigma^2 \ounit_{\cC}$, without the nullhomotopy, and in particular the symmetric monoidal structure of $\Gr^{\Zchar}_A \cC \in \CAlg_{\cC}(\PrL)$.
        \end{corollary}

        \begin{proof}
            By \cref{lem:homotopy-E1-null-factors-through-Sigma^2}, $\Map^{\h,1}(A, \Mod_{\cC}\units) \simeq \Map^{\h,1}(A, \Sigma^2 \ounit_{\cC}\units)$.
            By \cref{prop:braiding-of-1<1>-determines-the-character}, \cref{lem:maps-from-Z-are-from-truncated-sphere}, for any $a \in S$, the braiding $\hchar_{\ounit_{\cC}\shift{a}}$ determines the map 
            \begin{equation*}
                \ZZ \xto{a} A \xto{\Zchar} \Sigma^2 \ounit_{\cC}\units.
            \end{equation*}
            Therefore, the set of braidings determines the map
            \begin{equation*}
                \ZZ^{\oplus S} \onto A \xto{\Zchar} \Sigma^2 \ounit_{\cC}\units.
            \end{equation*}
            Consider the cofiber sequence
            \begin{equation*}
                \ZZ^{\oplus R} \to \ZZ^{\oplus S} \to A.
            \end{equation*}
            For any $w = (a_1,\dots,a_r) \in R$, again by \cref{prop:braiding-of-1<1>-determines-the-character}, \cref{lem:maps-from-Z-are-from-truncated-sphere} the composition
            \begin{equation*}
                \ZZ \xto{w} \ZZ^{\oplus R} \to \ZZ^{\oplus S} \to A \to \Sigma^2 \ounit_{\cC}\units
            \end{equation*}
            is equivalent to the data of $\hchar_{\ounit_{\cC}\shift{a_1}} \cdots \hchar_{\ounit_{\cC}\shift{a_r}}$. The isomorphism $i_w$ gives a nullhomotopy of this map. Therefore, the data of braidings and relations determines a map $\ZZ^{\oplus S} \to \Sigma^2 \ounit_{\cC}\units$ with a nullhomotopy of its composition to $\ZZ^{\oplus R}$, or equivalently a map $A \to \Sigma^2 \ounit_{\cC}\units$. It is straightforward to see that this is the original map $\Zchar$.
        \end{proof}

        % \begin{definition}\label{def:Sm-action-on-0}
        %     The element $\degree \in \SS$, thought of as $\Tm 1 = (1 + 1 + \cdots + 1)$, admits a natural $\Sm$-action. Composing with the map $(-\degree) \colon \SS \to \SS$, the element $0$ admits a $\Sm$-action, i.e.\ a pointed map
        %     \begin{equation*}
        %         \jm \colon \B\Sm \to \SS.
        %     \end{equation*}
        %     As $\B\Sm$ is connected, it lifts to the connected cover
        %     \begin{equation*}
        %         \jm \colon \B\Sm \to \tau_{\ge 1} \SS.
        %     \end{equation*}
        % \end{definition}

        \begin{remark}[The braiding of a higher minus one as a map of spaces]
            Assume $\cC$ admits a higher minus one $\Sigma^{\chrHeight} \ZZ/2 \to \ounit_{\cC}\units$. We can understand quite concretely the braiding, and therefore the symmetric monoidal structure, of $\ZZ$-graded categories induced by it and by a cohomology class $\hchar \in \H^{\chrHeight+2}(\ZZ, \Fq{2})$ (\cref{def:higher-minus-one}). The braiding map $\MM \to \Sigma \ounit_{\cC}\units$, choosing $\ounit_{\cC}\shift{1}$ factors through $\Sigma^{\chrHeight+1}\ZZ/2$ and defines a cohomology class $f(\hchar) \in \H^{\chrHeight+1}(\MM, \Fq{2})$, classifying the underlying map of spaces of the braiding, and which does not depend on $\cC$.

            Working with the universal category admitting a minus one of height $\chrHeight$, i.e.\ $\Mod_{\Sigma^{\chrHeight}\ZZ/2}\spc$, one can verify that $f$ commutes with Steenrod squares. That is, considering the map 
            \begin{equation*}
                \Sq^i \hchar \colon \ZZ \xto{\hchar} \Sigma^{\chrHeight+2} \ZZ/2 \xto{\Sq^i} \Sigma^{\chrHeight+i+2} \ZZ/2,
            \end{equation*}
            the corresponding cohomology class is 
            \begin{equation*}
                f(\Sq^i \hchar) \simeq \Sq^i f(\hchar) \colon \MM \xto{f(\hchar)} \Sigma^{\chrHeight+1} \ZZ/2 \xto{\Sq^i} \Sigma^{\chrHeight+i+1} \ZZ/2.
            \end{equation*}

            In particular, $f \colon \H^{*+1}(\ZZ, \Fq{2}) \to \H^*(\MM, \Fq{2})$, for $* \ge 1$, is a map of modules over the Steenrod algebra $\StAlg$. We can reduce to the case
            \begin{equation*}
                f \colon \StAlg^{\ge 2}[1] \to \H^{\ge 1}(\MM, \Fq{2}).
            \end{equation*}
            The graded module $\StAlg^{\ge 2}$ is generated by ${\Sq^{2^j}}$ for $j \ge 1$, therefore all cohomology classes appearing are Steenrod squares applied to monomials in $f(\Sq^{2^j}) \in \H^{2^j-1}(\MM, \Fq{2})$. In \cite{Giusti-Salvatore-Sinha-2022-cohomology-of-symmetric-groups}, the Hopf ring $\H^*(\MM, \Fq{2})$, is completely described. In particular they show that it is generated as a Hopf ring by classes $\gamma_{\ell, k} \in \H^{k(2^\ell - 1)}(\B \Sigma_{k 2^{\ell}})$. 
            Using \cref{prop:Kos-is-Kos}, one can show that $f(\Sq^2) = \gamma_{1,1}$, and using \cite[Theorem 8.3]{Giusti-Salvatore-Sinha-2022-cohomology-of-symmetric-groups} and the Cartan formula, $\Sq^I \gamma_{1,1} \in \{0, \gamma_{1,1}^{|I|-1}\}$. We now claim by induction on $i$ that $f(\Sq^i) \in \{0, \gamma_{1,1}^{i-1}\}$. If $i$ is not a power of 2, then by the Adem relations we can express $f(\Sq^i)$ as a linear combination of $\Sq^{i_1}f(\Sq^{i_2})$ for small values of $i_1$, therefore finishing by induction. 
            If $i$ is a power of $2$, we prove by a new induction, that $f(\Sq^j)\in \{0, \gamma_{1,1}^{j-1}\}$ for $i < j < 2i$. For this, we use the Adem relations to express $\Sq^j$ as a linear combination of products of smaller squares. Now we express $\Sq^{2i-1}$ also as $\Sq^{i-1}\Sq^i$ plus lower terms. Therefore, by induction, $f(\Sq^{i-1}\Sq^i) \in \{0,\gamma_{1,1}^{2i-2}\}$. We then get that 
            \begin{equation*}
                f(\Sq^i)^2 = \Sq^{i-1}f(\Sq^i) = f(\Sq^{i-1}\Sq^i) \in \{0, \gamma_{1,1}^{2i-2}\}.
            \end{equation*}
            As $\H^*(\MM, \Fq{2})$ is free divided powers algebra, $f(\Sq^i) \in \{0, \gamma_{1,1}^{i-1}\}$.

            This shows that there are only two families of cohomology classes coming from this construction, one is the trivial, and the second is $\{\gamma_{1,1}^{\degree}\}$. 
        \end{remark}

        \begin{remark}\label{rmrk:maps-from-symmetric-groups-and-infty}
            Inverting the map $+1 \colon \MM \to \MM$ on the free commutative monoid $\MM$, regarded as an $\MM$-module in spaces, yields 
            \begin{equation*}
                \MM[-1] \simeq \ZZ \times \B\Sigma_{\infty}
            \end{equation*}
            (see \cite[Appendix~C]{Bunke-Nikolaus-Tamme-2018-Belinson-regulator} for the general relation between element inversions and localizations). 
            Although $\MM[-1]$ is not the group completion in spaces, after stabilization one obtains 
            \begin{equation*}
                \SS[\ZZ \times \B\Sigma_{\infty}] \simeq \SS[\SS],
            \end{equation*}
            by the Barratt--Priddy--Quillen theorem \cite{Barratt-Priddy-1972,Quillen-2006-higher-K}.
            Looking at the connected covers we get an isomorphism
            \begin{equation*}
                \redSS[\B\Sinfty] \simeq \redSS[\tau_{\ge 1} \SS] \qin \cnSp.
            \end{equation*}

            Define the map of spaces
            \begin{equation*}
                j_{\infty} \colon  \B\Sinfty \xto{u} \redSS[\B\Sinfty] \simeq \redSS[\tau_{\ge 1} \SS] \xto{c} \tau_{\ge 1} \SS,
            \end{equation*}
            where $u, c$ are the unit and counit of the $(\redSS[-] \dashv \text{forgetful})$ adjunction. Equivalently, $j_{\infty}$ is the connected cover of the canonical map $\B\Sigma_{\infty} \times \ZZ \to \SS$ between the module localization and the group completion.
            
            The maps $j_m \colon \B\Sm \to \tau_{\ge 1} \SS$ of \cref{rmrk:Sm-action-from-braiding} can be written equivalently as the composition $\B\Sm \to \B\Sinfty \xto{j_{\infty}} \tau_{\ge 1} \SS$.
        \end{remark}
        
        \begin{definition}
            Let $\Zchar \colon \ZZ \to \Sigma^2\ounit_{\cC}\units$. Define the $\Sinfty$-representation
            \begin{equation*}
                \hchar_{\Zchar}^{\infty} \colon \B\Sinfty \xto{j_{\infty}} \tau_{\ge 1} \SS \xto{\hchar_{\Zchar}} \B\ounit_{\cC}\units.
            \end{equation*}
        \end{definition}

        \begin{remark}\label{rmrk:BSinfty-is-truncated-sphere-after-stabilization}
           The map $\hchar_{\Zchar}^{\infty}$ determines the map $\hchar_{\Zchar}$ as a map of pointed spaces. Indeed, this data is equivalent to the map of connective spectra
            \begin{equation*}
                \redSS[\B\Sinfty] \simeq \redSS[\tau_{\ge 1} \SS] \xto{c} \tau_{\ge 1} \SS \xto{\hchar_{\Zchar}} \Sigma \ounit_{\cC}\units.
            \end{equation*}
        \end{remark}

        %%%%%%%%%%%%%%%%%%%%%%%%%%%%%%%%%%%%%%%%%%%%%%%%%%%%%%%%%%%%%%%%%%%%%%%%%%%%%%%%
        \subsubsection*{Braiding and truncatedness in highly connected categories}

        We can say more in the case where $\Mod_{\cC}$ is $(\chrHeight+1)$-connected, as defined in \cite{BCSY-Fourier}.
        \begin{definition}[{\cite[\textsection~6.4]{BCSY-Fourier}}]\label{def: cat connectedness}
            A symmetric monoidal category $\cC$ is said to be $d$-connected (at $p$) if for any $d$-truncated, $\pi$-finite ($p$-local) space $A$, the natural map
            \begin{equation*}
                A \to \Map_{\CAlg(\cC)}(\ounit^A, \ounit)
            \end{equation*}
            is an isomorphism.
        \end{definition}

        \begin{definition}
            A connective spectrum is $\pi$-finite if it is truncated and all its homotopy groups are finite. A spectrum is $\pi$-torsion if it is isomorphic to a (filtered) colimit of $\pi$-finite spectra. Let $\Sp\pitor\subseteq \cnSp$ be the subcategory of $\pi$-torsion connective spectra and $\plocSp\pitor \subseteq \cnSp$ be the subcategory of $\pi$-finite $p$-local connective spectra. Denote the right adjoints of the inclusions by
            \begin{equation*}
                (-)\pitor \colon \cnSp \to \Sp\pitor, \qquad (-)_{(p)}\pitor \colon \cnSp \to \plocSp\pitor.
            \end{equation*}
        \end{definition}

        \begin{theorem}[{\cite[Proposition~6.38, Corollary~6.40, Corollary~6.45, Theorem~6.58]{BCSY-Fourier}}]\label{thm:connectedness-implies-pitor}
            Let $\cC \in \CAlg(\PrL)$ be $\infty$-semiadditive and $(\Fp, \chrHeight)$-oriented at $p$. Assume that $\Mod_{\cC}$ is $(\chrHeight+1)$-connected at $p$, then
            \begin{equation*}
                (\ounit_{\cC} \units)\pitor_{(p)} \simeq \Inprime.
            \end{equation*}
        \end{theorem}

        \begin{example}[{\cite[Theorem~7.8]{BCSY-Fourier}}]
            $\ModEn$ is $(\chrHeight+1)$-connected at the prime $p$. It is $1$-connected at all primes $\ell \neq p$.
        \end{example}

        \begin{corollary}
            Let $\cC \in \CAlg(\PrL)$ be $\infty$-semiadditive. Assume that $\cC$ is $(\Fp, \chrHeight)$-oriented and that $\Mod_{\cC}$ is $(\chrHeight+1)$-connected at all primes $p$. Then
            \begin{equation*}
                (\ounit_{\cC}\units)\pitor \simeq \In.
            \end{equation*}
        \end{corollary}

        % \begin{corollary}
        %     Let $\cC \in \CAlg(\PrL)$ be $\infty$-semiadditive. Assume that $\ounit_{\cC}$ is $\chrHeight$-truncated and Galois-closed. Then $\cC$ is of height $\le \chrHeight$ at all primes.
        % \end{corollary}
        % \begin{proof}
        %     As $\ounit_{\cC}$ is Galois-closed at $p$, $(\ounit_{\cC}\units)\pitor_{(p)} \simeq \Inprime[\chrHeight_p]$, where $\chrHeight_p$ is the height of $\cC$ at $p$. As $\ounit_{\cC}\units$ is $\chrHeight$-truncated, $\chrHeight_p \le \chrHeight$. 
        % \end{proof}

        In an announced work of Johnson-Freyd and Reuter, they construct a pro-Galois extension $\Vectn \to \overline{\catname{Vect}}_{\field}^n$, such that $\Mod_{\overline{\catname{Vect}}_{\field}^n}$ is $(\chrHeight+1)$-connected. As an example, $\overline{\catname{Vect}}_{\field} = \sVect$ (\cite{Johson-Freyd-2017-sVect}). 

        Having introduced the necessary definitions and notation, we now return to the study of braidings. We restrict to the case where the underlying space of $\ounit_{\cC}$ is truncated, so that $\Sigma \ounit_{\cC}\units$ is also truncated. In this setting, the map
        \begin{equation*}
            \hchar_{\Zchar} \colon \tau_{\ge 1} \SS \to \Sigma \ounit_{\cC}\units
        \end{equation*}
        factors through some truncation $\tau_{[1,d+1]} \SS$ which is a $\pi$-finite spectrum.

        \begin{corollary}\label{cor:truncated-unit-truncated-map}
            Assume that $\ounit_{\cC}$ is $d$-truncated. Then $\hchar_{\Zchar}$ factors uniquely as
            \begin{equation*}
                \hchar_{\Zchar} \colon \tau_{\ge 1} \SS \to \tau_{\ge [1, d+1]} \SS \xto{\hchar_{\Zchar}^{\pi}} \Sigma(\ounit_{\cC}\units)\pitor \to \Sigma \ounit_{\cC}\units.
            \end{equation*}
        \end{corollary}

        \begin{lemma}
            If $\cC$ has an $\chrHeight$-truncated unit, then $\Mod_{\cC}$ has an $(\chrHeight+1)$-truncated unit.
        \end{lemma}
        \begin{proof}
            It follows by the identity $\ounit_{\cC} \simeq \Omega \cC\core \in \spc$.
        \end{proof}

        \begin{corollary}
            The unit of $\Vectn$ is $\chrHeight$-truncated.
        \end{corollary}

        \begin{proposition}
            Let $\cC \in \CAlg(\PrL)$ be $\infty$-semiadditive. Assume that for every $p$, $\cC$ is $(\Fp,\chrHeight_p)$-orientable and $\Mod_{\cC}$ is $(\chrHeight_p+1)$-connected. 
            Assume that $\ounit_{\cC}$ is $d$-truncated.
            Let $\Zchar_1, \Zchar_2 \colon \ZZ \to \Sigma^2 \ounit_{\cC}\units$. Then $\Zchar_1 \simeq \Zchar_2$ if and only if $\hchar_{\Zchar_1}^{\infty} \simeq \hchar_{\Zchar_2}^{\infty}$.

            That is, the $\Sinfty$-representation $\hchar_{\Zchar}^{\infty}$ determines the symmetric monoidal structure on $\Gr^{\Zchar}_{\ZZ} \cC$.
        \end{proposition}
        
        \begin{proof}
            It is enough to verify that $\hchar_{\Zchar}^{\infty}$ determines $\hchar_{\Zchar}$. By \cref{cor:truncated-unit-truncated-map} it is enough to verify that it determines $\hchar_{\Zchar}^{\pi}$. Moreover, it is enough to check it after $p$-localization at all primes $p$. Let $p$ be a prime and consider
            \begin{equation*}
                (\hchar_{\Zchar}^{\infty})_{(p)} \colon \redSS[\B\Sinfty]_{(p)} \to \Sigma(\ounit_{\cC}\units)_{(p)}.
            \end{equation*}
            As $(\Sigma\ounit_{\cC}\units)_{(p)}$ is $(d+1)$-truncated, this map is equivalent to a map
            \begin{equation*}
                (\hchar_{\Zchar}^{\infty})_{(p)} \colon \Sigma^{-1} \tau_{\le d+1} \redSS[\B\Sinfty]_{(p)} \to (\ounit_{\cC}\units)_{(p)}.
            \end{equation*}
            Note moreover, that $\Sigma^{-1} \tau_{\le d+1} \redSS[\B\Sinfty]_{(p)} \simeq \tau_{\le d} \Omega \redSS[\B\Sinfty]_{(p)}$ is a $\pi$-finite connective spectrum. Therefore, by \cref{thm:connectedness-implies-pitor}, it is equivalent to a map
            \begin{equation*}
                (\hchar_{\Zchar}^{\infty})_{(p)} \colon \tau_{\le d} \Omega \redSS[\B\Sinfty]_{(p)} \to \Inprime[\chrHeight_p].
            \end{equation*}
            By \cref{rmrk:maps-to-In} and \cref{rmrk:BSinfty-is-truncated-sphere-after-stabilization}, it is the same as an element
            \begin{equation*}
                (\hchar_{\Zchar_i}^{\infty})_{(p)} \in \widehat{\pi}_{\chrHeight_p}(\tau_{\le d} \Omega \redSS[\B\Sinfty]_{(p)}) \simeq \widehat{\pi}_{\chrHeight_p+1}(\redSS[\tau_{\ge 1} \SS]_{(p)}).
            \end{equation*}

            The map $(\hchar_{\Zchar}^{\pi})_{(p)}$ is a map of connective spectra
            \begin{equation*}
                \hchar_{\Zchar}^{\pi} \colon \tau_{[1,d+1]} \SS_{(p)} \to \Sigma\Inprime[\chrHeight_p],
            \end{equation*}
            which is equivalent to a map
            \begin{equation*}
                \hchar_{\Zchar}^{\pi} \colon \Omega \tau_{[1, d+1]} \SS_{(p)} \to \Inprime[\chrHeight_p].
            \end{equation*}
            By \cref{rmrk:maps-to-In}, this is the same as an element
            \begin{equation*}
                \hchar_{\Zchar}^{\pi} \in \widehat{\pi}_{\chrHeight_p}(\Omega \tau_{[1,d+1]} \SS_{(p)}) \simeq \widehat{\pi}_{\chrHeight_p +1}(\tau_{\ge 1} \SS_{(p)}).
            \end{equation*}

            Following \cref{rmrk:BSinfty-is-truncated-sphere-after-stabilization}, $(\hchar_{\Zchar}^{\infty})_{(p)}$ is the image of $(\hchar_{\Zchar}^{\pi})_{(p)}$ under the natural map from the unit map $\SS[\tau_{\ge 1} \SS] \to \tau_{\ge 1} \SS$. But this map has a section on the level of underlying spaces by the zig-zag identities, which in particular gives a retract on the level of $\hat{\pi}_{\chrHeight_p+1}$. That is, the map $\hchar_{\Zchar}^{\pi} \mapsto \hchar_{\Zchar}^{\infty}$ is injective.
        \end{proof}
        
    %%%%%%%%%%%%%%%%%%%%%%%%%%%%%%%%%%%%%%%%%%%%%%%%%%%%%%%%%%%%%%%%%%%%%%%%%%%%%%%%
    \subsection{Braiding character in graded categories}
    \label{subsec:graded-braiding-character}
    %%%%%%%%%%%%%%%%%%%%%%%%%%%%%%%%%%%%%%%%%%%%%%%%%%%%%%%%%%%%%%%%%%%%%%%%%%%%%%%%
        In \cref{subsec:graded-braiding} we saw that the braiding of different objects, together with the relations between them, determines the symmetric monoidal structure of twisted graded categories. In the special case of $\ZZ$-graded categories, this gives a complete invariant. Therefore, this invariant is in general computationally hard to use. The braiding character introduced in \cref{subsec:braiding-character} is a simpler invariant, and has an internal description within graded categories.

        %%%%%%%%%%%%%%%%%%%%%%%%%%%%%%%%%%%%%%%%%%%%%%%%%%%%%%%%%%%%%%%%%%%%%%%%%%%%%%%%
        \subsubsection*{Characters and Thom constructions}

        We use the theory of $\cC$-atomic objects as defined in \cite{Ben-Moshe-Schlank-2024-K-theory} when $\cC$ is a mode and in \cite{Ben-Moshe-2024-naturality-Yoneda, Ramzi-2023-seperability} for general $\cC$. In our categories of interest, atomic objects are just dualizable objects.
        \begin{definition}
            Let $\cD \in \Mod_{\cC}$. An object $X \in \cD$ is called $\cC$-atomic (or just atomic when $\cC$ is clear from context) if $X[-] \colon \cC \to \cD$ is an internal left adjoint in $\Mod_{\cC}$.

            We denote the full subcategory consisting of $\cC$-atomic objects in $\cD$ by $\cD\cCat$ (or just $\cD\at$ when $\cC$ is clear from context).
        \end{definition}

        \begin{example}[{\cite[Lemma~6.6]{Ramzi-2023-seperability}}]
            $\cC\at \simeq \cC\dbl$.
        \end{example}

        \begin{example}[{\cite[Lemma~4.50]{Ramzi-2024-rigid}}]\label{exm:rigit-atomic-dualizable}
            Let $\cD \in \CAlg(\cC)$ be rigid over $\cC$, in the sense of \cite[Definition~9.1.2]{Gaitsgory-Rozenblyum-2019-DAG}. Then $\cD\at \simeq \cD\dbl$
        \end{example}

        Our categories of interest are Thom categories, which are all rigid:
        \begin{proposition}\label{lem:Thom-is-rigid}
             Let $X\in \cnSp$ and $f \colon X \to \Mod_{\cC}\units$ be a map of connective spectra. 
             Then $\Th(f)$ is rigid over $\cC$. In particular, $\Th(f)\at \simeq \Th(f)\dbl$.
        \end{proposition}
        \begin{proof}
            By \cite[Lemma 4.54]{Ramzi-2024-rigid} it suffices to verify that:
            \begin{enumerate}
                \item\label{cond:atomic-unit} The unit $\ounit_{\Th(f)}$ is $\cC$-atomic,
                \item\label{cond:multiplication-is-left-adjoint} the multiplication map $\Th(f)\otimes_\cC \Th(f) \to \Th(f) \in \Mod_{\Th(f)}(\cC)$ is internally left adjoint.
            \end{enumerate} 

            Let $i \colon 0 \to X$ be the zero map. Then $\ounit_{\Th(f)}[-] \colon \cC \to \Th(f)$ is the functor $i_!$ which admits a $\cC$-linear right adjoint $i^*$. Therefore $\ounit_{\Th(f))}$ is atomic.
            
            By the Thom isomorphism (\cite[Proposition~3.16,Corollary 3.17]{Antolin-Barthel-2019-Thom}), 
            \begin{equation*}
                \Th(F)\otimes_\cC \Th(F)\simeq \Th(F)[X].
            \end{equation*}
            Under this isomorphism, the multiplication map is $X_!$, which admits a $\Mod_{\Th(F)}(\cC)$-linear right adjoint $X^*$, proving \labelcref{cond:multiplication-is-left-adjoint}.                
            The implication is \cref{exm:rigit-atomic-dualizable}.
        \end{proof}

        The topological Hochschild homology of a Thom ring spectrum was computed in \cite{Blumberg-Cohen-Schlichtkrull-2010-THH-of-Thom}. Their results were extended to $\THH_{\cC}$ of a $\cC$-linear Thom category, for general presentably symmetric monoidal category $\cC$ in \cite{Carmeli-Cnossen-Ramzi-Yanovski-2022-characters}. We recall their results:

        \begin{definition}\label{def:character-of-a-map}
            Let $X \in \spc$, $Y \in \cnSp$ and $f \colon X \to Y$. Define the character of $f$ as the map\footnote{
                In the case $Y = \cC\units$, $f$ is an $X$-local system of invertible elements and therefore its character (\cref{def:character}) is defined. By \cref{lem:trace-is-eta+id}, it agrees with this definition.

                When $X = \BG$ is a pointed connected space and $Y = \B^m A$ for an abelian group $A$, this map on $\pi_0$ is known as the \emph{transgression} of $f$ (see \cite{Willerton-2008-transgression}).
            }        
        
        \begin{equation*}
                \chi_f \colon \L X \xto{\L f} \L Y \simeq Y \times \Omega Y \xto{\eta + \id} \Omega Y.
            \end{equation*}
        \end{definition}
        In the case $Y = \Mod_{\cC}\units$, then $\Omega Y \simeq \cC\units$.

        \begin{remark}\label{rmrk:character-between-spectra}
            If $X \in \cnSp$ and $f \colon X \to Y$ is a map of connective spectra, then, by naturality of the decomposition of $\L$ and of $\eta + \id$, the character of $f$ is isomorphic to the map
            \begin{equation*}
                \L X \simeq X \times \Omega X \xto{\eta + \id} \Omega X \xto{\Omega f} \Omega Y.
            \end{equation*}
        \end{remark}
        
        \begin{theorem}[{\cite[Theorem~1]{Blumberg-Cohen-Schlichtkrull-2010-THH-of-Thom}, \cite[Corollary~7.15]{Carmeli-Cnossen-Ramzi-Yanovski-2022-characters}}]\label{thm:THH-of-Th}
            Let $X$ be a space and $f \colon X \to \Mod_{\cC}\units$. Then $\Th(f)\in \Mod_{\cC}$ is dualizable and
            \begin{equation*}
                \THH_{\cC}(\Th(f)) \simeq \Th(\chi_f).
            \end{equation*}
        \end{theorem}

         \begin{lemma}
            Let $A \in \Ab$ and $\Zchar \in \Map\Enull(A, \Sigma\cC\units)$. Then
            \begin{equation*}
                \THH_{\cC}(\Gr^{\Zchar}_A \cC) \simeq \ounit_{\cC}[A] \qin \CAlg(\Gr_A \cC),
            \end{equation*}
            and in particular does not depend on $\Zchar$.
        \end{lemma}
        
        \begin{proof}
            By \cref{thm:THH-of-Th} it is enough to see that the character map $\chi_{\Zchar} \colon \L A \to \cC\units$ is zero. But by \cref{rmrk:character-between-spectra} it factors through $\Omega A \simeq 0$.
        \end{proof}

        %%%%%%%%%%%%%%%%%%%%%%%%%%%%%%%%%%%%%%%%%%%%%%%%%%%%%%%%%%%%%%%%%%%%%%%%%%%%%%%%
        \subsubsection*{Braiding characters}

        We start by defining a different version of the braiding character of objects in degree exactly 1 in $\ZZ$-graded categories.
        \begin{definition}
            Let $\Zchar \colon \ZZ \to \Sigma^2 \ounit_{\cC}\units$. Let $V \in \cC\dbl$. By \cref{lem:Thom-is-rigid}, the braiding of $V\shift{1} \in \Gr^{\Zchar}_{\ZZ} \cC$ 
            \begin{equation*}
                \T V\shift{1} \colon \cC[\MM] \to \Gr^{\Zchar}_{\ZZ} \cC
            \end{equation*}
            is internal left adjoint. Define
            \begin{equation*}
                \mscr{X}^{\Zchar}_{\T V\shift{1}} \colon \ounit_{\cC}[\L \MM] \to \ounit_{\cC}[\ZZ] \qin \CAlg(\cC)
            \end{equation*}
            as its image under $\THH_{\cC}$.
        \end{definition}

        We can define this variant in any graded category with an object concentrated in a single degree, as we always have a symmetric monoidal map $\Gr^{\Zchar\circ a}_{\ZZ} \cC \to \Gr^{\Zchar}_A \cC$ sending $V\shift{1}$ to $V\shift{a}$.

        \begin{definition}\label{def:braiding-character-of-shift-by-an-element}
            Let $A \in \Ab$ and $\Zchar \in \Map\Enull(A, \Sigma \cC\units)$. Let $V \in \cC\dbl$ and $a\in A$. Then define 
            \begin{equation*}
                \mscr{X}^{\Zchar}_{\T V\shift{a}} \coloneqq \mscr{X}^{\Zchar \circ a}_{\T V\shift{1}} \colon \L\MM \to \ounit_{\cC}[\ZZ].
            \end{equation*}
        \end{definition}

        \begin{lemma}\label{lem:braiding-characters-agree}
            Let $A\in \Ab$ and $\Zchar \in \Map\Enull(A, \Sigma \cC\units)$. Let $V \in \cC\dbl$ and $a\in A$.
            Then $\mscr{X}^{\Zchar}_{\T V\shift{a}}$ is the braiding character of $V\shift{a}$ in the sense of \cref{def:braiding-character-general}.
        \end{lemma}
        \begin{proof}
            It is enough to prove it for the universal case $A = \ZZ$, $a = 1$. Note that the composition $\L\MM \xto{\mscr{X}^{\Zchar}_{\T V\shift{1}}} \ounit_{\cC}[\ZZ] \to \ounit_{\cC}$ is the usual character of $\T V\shift{1}$. Moreover, this map sends each $\L\B\Sm$ exactly to degree $\degree$, so the claim follows from \cref{lem:map-factors-constant} and \cref{cor:braiding-character-factors-constant}.
        \end{proof}        

        \begin{notation}
            When $\Zchar$ is not clear from context, we will denote by $\chi^{\Zchar}_{\T V\shift{1}}$ the usual character of the braiding map, i.e.\ the composition $\L\MM \xto{\mscr{X}^{\Zchar}_{\T V\shift{1}}} \ounit_{\cC}[\ZZ] \to \ounit_{\cC}$.
        \end{notation}

        In the semiadditive case, we can extend this result to any object with finite support which is pointwise dualizable. Such objects are of the form $\bigoplus V_i \shift{a_i}$ for $a_i\in A$ and $V_i \in \cC\dbl$. The elements $a_i$ define a map $\vec{a} \colon \ZZ^r \to A$, sending $e_i$ to $a_i$. This, allows one to define a braiding character 
        \begin{equation*}
            \L\MM \to \ounit_{\cC}[\ZZ^r].
        \end{equation*}
        % that is just the tensor of $\mscr{X}_{\T V_i\shift{a_i}}$, i.e. it sends $\B\Csigma \subseteq \L\B\Sm \subseteq \L\MM$ to
        % \begin{equation*}
        %     \sum_{\substack{\degree_1 + \cdots + \degree_r = \degree \\ \sigma_1 \in \Sm[\degree_1], \dots, \sigma_r \in \Sm[\degree_r] \\
        %     \sigma = \sigma_1\cdots \sigma_r \in \Sm[\degree_1] \times \cdots \times \Sm[\degree_r]\subseteq \Sm}} \Ind_{\B\Csigma[1] \times \cdots \times \B\Csigma[r]}^{\B\Csigma} \left(
        %         \chi_{\T V_1\shift{a_1}}(\sigma_1) \cdots \chi_{\T V_r\shift{a_r}}(\sigma_r) 
        %     \right)t_1^{\degree_1} \cdots t_r^{\degree_r}.
        % \end{equation*}
        % We can prove similar results for this character, replacing $\ZZ$ with $\ZZ^r$, but better yet, if we compose with the multiplication map $\ounit_{\cC}[\ZZ^r] \to \ounit_{\cC}[\ZZ]$, it again lives in one degree, and one can repeat \cref{cor:braiding-character-factors-constant}.

        \begin{definition}
            Assume that $\cC$ is semiadditive. Let $\Zchar \in \Map\Enull(\ZZ^r, \Sigma \cC\units)$. 
            Let $W \coloneqq \bigoplus_i V_i \shift{e_i} \in \Gr^{\Zchar}_{\ZZ^r} \cC$ where $V_i \in \cC\dbl$. 
            The braiding of $W$ induces the $\cC$-linear left adjoint functor $\T W \colon \cC[\MM] \to \Gr^{\Zchar}_{\ZZ^r} \cC$. Applying $\THH_{\cC}$ we get the map
            \begin{equation*}
                \widetilde{\mscr{X}}^{\Zchar}_{\T W} \colon \L\MM \to \ounit_{\cC}[\ZZ^r].
            \end{equation*}
            
            Composing with the summation map $\ZZ^r \to \ZZ$, we define
            \begin{equation*}
                \mscr{X}^{\Zchar}_{\T W} \colon \L\MM \xto{\widetilde{\mscr{X}}^{\Zchar}_{\T W}} \ounit_{\cC}[\ZZ^r] \to \ounit_{\cC}[\ZZ].
            \end{equation*}
        \end{definition}

        \begin{definition}
            Assume that $\cC$ is semiadditive. Let $A \in \Ab$, $\Zchar \in \Map\Enull(A, \Sigma \cC\units)$ and $W \in \Gr^{\Zchar}_A \cC$ be pointwise dualizable and with finite support. Then $W$ must be of the form $\bigoplus_{i=1}^r V_i\shift{a_i}$ for some $a_i \in A$ and $V_i \in \cC\dbl$. This defines a homomorphism $\vec{a} \colon \ZZ^r \to A$ sending $e_i$ to $a_i$. Let $W' \coloneqq \bigoplus_{i=1}^r V_i \shift{e_i} \in \Gr^{\Zchar \circ \vec{a}}_{\ZZ^r} \cC$. Then $W'$ is sent to $W$ under the symmetric monoidal functor $\vec{a}_! \colon \Gr^{\Zchar\circ \vec{a}}_{\ZZ^r} \cC \to \Gr^{\Zchar}_A \cC$. Similarly to \cref{def:braiding-character-of-shift-by-an-element}, we define
            \begin{equation*}
                \mscr{X}^{\Zchar}_{\T W} \coloneqq \mscr{X}^{\Zchar \circ \vec{a}}_{\T W'} \colon \L \MM \to \ounit_{\cC}[\ZZ].
            \end{equation*}
        \end{definition}

        Repeating the proof of \cref{lem:braiding-characters-agree}, we get
        
        \begin{lemma}\label{lem:braiding-characters-agree-multiple}
            Assume that $\cC$ is semiadditive. Let $A\in \Ab$ and $\Zchar\in \Map\Enull(\A, \Sigma \cC\units)$. 
            Let $W \in \Gr^{\Zchar}_A \cC$ with finite support and assume that $W$ is pointwise dualizable. 
            Then the braiding character of $W$ agrees with $\mscr{X}^{\Zchar}_{\T W}$.
        \end{lemma}

    \subsection{Exterior algebras}
    \label{subsec:exterior-algebras}
    %%%%%%%%%%%%%%%%%%%%%%%%%%%%%%%%%%%%%%%%%%%%%%%%%%%%%%%%%%%%%%%%%%%%%%%%%%%%%%%%

        Let $\cC\in \CAlg(\PrL)$ and $\Zchar \colon \ZZ \to \Sigma^2 \ounit_{\cC}\units$ be a map of connective spectra.
        Let $X \in \cC$. The free commutative algebra generated in $\Gr^{\Zchar}_{\ZZ} \cC$ by $X\shift{1}$ is a graded commutative algebra, which we think of as an analog of the exterior algebra. 
        For graded vector spaces, where there exist exactly two symmetric monoidal structures (\cref{rmrk:2-structurs-on-GrVect}), this construction recovers the classical (graded) symmetric and exterior algebras. In $\Mod_{\sVect}$, it naturally endows the exterior powers of super categories, defined in \cite{Ganter-Kapranov-2014-exterior-categories} by Ganter and Kapranov, with an algebra structure.

        \begin{definition}\label{def:exterior}
            Let $X \in \cC$. Define the $\Zchar$-twisted exterior algebra $\extAlg_{\Zchar} X \in \CAlg(\Gr_{\ZZ}^\Zchar \cC)$ of $X$ as the free commutative algebra generated by $X\shift{1}$:
            \begin{equation*}
                \extAlg_{\Zchar} X \coloneqq \mrm{Fr}_{\EE_{\infty}}(X\shift{1}) = \Sym[\bullet](X\shift{1}).
            \end{equation*}
        \end{definition}

        \begin{example}
            Assume $\Zchar \colon \ZZ \to \Sigma \cC\units$ is the trivial map. Then $\Gr^{\Zchar}_{\ZZ} \cC \simeq \Gr_{\ZZ} \cC$ and $\extAlg_{\Zchar} X \simeq \Sym[\bullet] X$ is the usual graded symmetric algebra.
        \end{example}

        \begin{example}\label{exm:exteior-Koszul}
            Assume $\cC = \Vect$. Let $(-1) \colon \ZZ/2 \to \field\units$ be the usual minus one map. Then $\Gr^{\Kos}_{\ZZ} \Vect$ is identified with the usual category of graded vector spaces with the Koszul sign (\cref{def:Koszul}) and $\extAlg_{\Kos} X$ is the usual graded exterior algebra of vector spaces.
        \end{example}

        \begin{example}
            For $\cC = \Mod_{\sVect}$, let $\widehat{\eta^2} \in \pinD[2]$ be the unique non-trivial element. Then, working in $\Gr^{\hat{\eta^2}}_{\ZZ} \cC$ as in \cref{def:dual-stable-stems}, $\alt_{\widehat{\eta^2}}$ is identified with the exterior power of super linear categories which is the main object of study of \cite{Ganter-Kapranov-2014-exterior-categories}. In particular, $\extAlg_{\widehat{\eta^2}}$ gives a natural graded symmetric monoidal structure on the collection of exterior power of super linear categories.
        \end{example}    
        
        \begin{corollary}
            For any $X \in \cC$, $\bigsqcup_{\degree} \alt_{\Zchar} X$ is naturally an associative algebra in $\cC$.
        \end{corollary}
        \begin{proof}
            As a monoidal category $\Gr^{\Zchar}_{\ZZ} \cC \simeq \Gr_{\ZZ} \cC$. The colimit functor, which in this case is the coproduct $\bigsqcup \colon \Gr_{\ZZ} \cC \to \cC$, is monoidal, therefore sends algebras to algebras. In particular $\extAlg_{\Zchar} X \in \Alg(\Gr_{\ZZ} \cC)$ is sent to $\bigsqcup_{\degree} \alt_{\Zchar} X \in \Alg(\cC)$.
        \end{proof}

        \begin{remark}\label{rmrk:exterior-as-alternating}
           The construction of the classical symmetric and alternating powers is through the fact that $\Sm$ admits exactly two characters --- the trivial and the sign. Given a character $\hchar \colon \Sm \to \field\units$ and $V \in \Vect$ we define
           \begin{equation*}
               \alt_{\hchar} V \coloneqq (V\om \otimes \field[\hchar])_{h\Sm}
           \end{equation*}
           where $\field[\hchar]$ is the one-dimensional representation corresponding to $\hchar$. Indeed for $\hchar = \triv$ we get the symmetric powers and for $\hchar = \sgn$ we get the exterior powers.

           For a general category $\cC$ there might be many more characters $\hchar \colon \Sm \to \ounit_{\cC}\units$. The braiding $\hchar_{\Zchar} \colon \MM \to \B\ounit_{\cC}\units$ of $\Zchar \colon \ZZ \to \Sigma^2\ounit_{\cC}\units$, defines a character $\hchar_{\Zchar}^{\degree}$ of $\B\Sm$ for each $\degree$. Moreover, for any $X\in \cC$, $\extAlg_{\Zchar} X$ at degree $\degree$ is identified with
           \begin{equation*}
               \alt_{\Zchar} X \simeq (X\om \otimes \ounit_{\cC}[\hchar_{\Zchar}^{\degree}])_{h\Sm}.
           \end{equation*}
           This is another justification for the name \quotes{exterior algebra}. 

            Note that not all characters arise in this manner; a further investigation of exterior powers associated with a broader class of characters is carried out in a separate work \cite{Keidar-Ragimov-2025-alternating}.
        \end{remark}

        \subsubsection*{Bialgebra structure}

        The classical exterior algebra is a Hopf algebra, taking advantage of the fact that $\extAlg (V \oplus W) \simeq \extAlg V \otimes \extAlg W$. We now extend this phenomenon to our settings. The existence of the antipode map is due to the fact that any vector space is grouplike. Therefore we manage to extend the construction to that of a bialgebra, which becomes Hopf for grouplike objects. We expect that for virtually $(\Fp, \chrHeight)$-orientable categories (see \cite{BCSY-Fourier}) there should be more structure on these rings, resembling a higher version of an antipode, yet we did not study it in this paper.

        \begin{definition}
            Let $\cC$ be a category. Define the category of co-groups in $\cC$ as
            \begin{equation*}
                \coGrp(\cC) \coloneqq \coCMon\gp(\cC) = (\CMon\gp(\cC\op))\op = \Grp(\cC\op)\op.
            \end{equation*}
        \end{definition}

        \begin{remark}
            When $\cC$ is semiadditive, $\coCMon(\cC) = \cC$ and being co-grouplike is a property of objects of $\cC$.
            Moreover in this case, for any object $X$, the shearing map of $X$ as a monoid is an isomorphism if and only if the shearing map of $X$ as a comonoid is an isomorphism. Thus being co-grouplike is the same as being grouplike.
        \end{remark}

        \begin{lemma}
            Assume that $\cC$ is semiadditive. Let $X\in \cC$. Then $\extAlg_{\Zchar} X$ has a natural structure of a bialgebra. If moreover $X$ is grouplike then $\extAlg_{\Zchar} X$ has a structure of a Hopf algebra.
        \end{lemma}
        
        \begin{proof}
            The functor 
            \begin{equation*}
                \extAlg_{\Zchar}\colon \cC \xto{(-)\shift{1}} \Gr^{\Zchar}_{\ZZ} \cC \xto{\mrm{Fr}_{\EE_{\infty}}} \CAlg(\Gr^{\Zchar}_{\ZZ} \cC)
            \end{equation*}
            is left adjoint to the evaluation at $1$ functor $\CAlg(\Gr^{\Zchar}_{\ZZ} \cC) \to \cC$. Therefore it commutes with colimits, and in particular sends finite direct sums to tensor products. In particular it gives rise to a functor
            \begin{equation*}
                \extAlg_{\Zchar}\colon \cC \simeq \coCMon(\cC) \to \coCMon(\CAlg(\Gr_{\ZZ} \cC)) \simeq \BiCAlg(\Gr_{\ZZ} \cC)).
            \end{equation*}

            Restricting to grouplike, or equivalently, co-grouplike objects we get
            \begin{equation*}
                \extAlg_{\Zchar}\colon \coGrp(\cC) \simeq \coCMon\gp(\cC) \to \coCMon\gp(\CAlg(\Gr_{\ZZ} \cC)) = \Hopf(\Gr_{\ZZ} \cC).
            \end{equation*}
        \end{proof}
        
        \begin{corollary}
            Assume $\cC$ is additive. Then the exterior algebra functor extends to
            \begin{equation*}
                \extAlg_{\Zchar}\colon \cC \to \Hopf(\Gr_{\ZZ} \cC).
            \end{equation*}
        \end{corollary}

        %%%%%%%%%%%%%%%%%%%%%%%%%%%%%%%%%%%%%%%%%%%%%%%%%%%%%%%%%%%%%%%%%%%%%%%%%%%%%%%%
        \subsubsection*{Dimensions comparison}

        One can extend the construction of the exterior algebra to a lax symmetric monoidal functor $\NN \to \Gr^{\Zchar}_{\ZZ} \cC$, thinking of it as a graded algebra in $\Gr^{\Zchar}_{\ZZ} \cC$.
        
        \begin{definition}\label{def:exterior-functor}
            Let $\cD$ be a presentably symmetric monoidal category. Define the functor
            \begin{equation*}
                \extAlg \colon \cD \xto{\T} \cD[\MM] \xto{\deg_!} \cD[\NN].
            \end{equation*}
            As $\deg_!$ is symmetric monoidal and $\T$ factors through $\CAlg(\cD[\MM])$, the functor $\extAlg$ lands in lax symmetric monoidal functors $\NN \to \cD$.
        \end{definition}
        By \cref{rmrk:exterior-as-alternating}, when $\cD = \Gr^{\Zchar}_{\ZZ}\cC$, for $X \in \cC$, $\alt_{\Zchar} X$ as in \cref{def:exterior}, agrees with $\alt X\shift{1}$ as in \cref{def:exterior-functor}. This also allows us to define an analog of the exterior algebra for objects not concentrated only in degree 1 (which is just the free commutative algebra generated by this object).

        \begin{corollary}\label{cor:exterior-takes-sums-to-products}
            Let $\cD$ be a presentably symmetric monoidal category. Then for any $X,Y \in \cD$
            \begin{equation*}
                \extAlg (X \sqcup Y) \simeq \extAlg X \otimes_{\Day} \extAlg Y.
            \end{equation*}
        \end{corollary}
        \begin{proof}
            Since $\deg_!$ is symmetric monoidal, this follows from \cref{cor:braiding-sends-coproducts-to-tensors}.
        \end{proof}

        Written concretely, we get the known formula
        \begin{equation*}
            \alt(X \sqcup Y) \simeq \bigsqcup_{i+j = \degree} \alt[i] X \otimes \alt[j] Y.
        \end{equation*}

        Now, similarly to the braiding character, we can consider the image of this map under $\THH_{\cD}$, getting a generating function for the dimensions:
        \begin{lemma}\label{lem:extAlg-is-dbl}
             Assume $\cD$ is 1-semiadditive. Then for any $V \in \cD\dbl$, $\extAlg_{\Zchar} V$ is point-wise dualizable.
        \end{lemma}
        \begin{proof}
            The object $\extAlg V$ at degree $\degree$ is given by $(V\om)_{h\Sm}$ which is dualizable as $V\om$ is dualizable and $\B\Sm$ is 1-finite. 
        \end{proof}
        \begin{definition}
            Assume that $\cD$ is 1-semiadditive. Let $V \in \cD\dbl$. Let
            \begin{equation*}
            \extAlg V\shift{1} \coloneqq \deg_! \T V\shift{1} \colon \NN \to \Gr_{\ZZ} \cD.
            \end{equation*}
            Notice again that its image under the symmetric monoidal functor $\colim \colon \Gr_{\ZZ} \cD \to \cD$ is $\extAlg V$. Then define
            \begin{equation*}
                \dim \extAlg V \colon \ounit_{\cD}[\NN] \to \ounit_{\cD}[\ZZ] \qin \CAlg(\cD).
            \end{equation*}
            as its image under $\THH_{\cD}$ of $\extAlg V\shift{1}$.
        \end{definition}

        \begin{remark}
            Writing $\ounit_{\cD}[\ZZ] = \ounit_{\cD}[t^{\pm 1}]$, the evaluation of this map at $\degree \in \NN$ is $\dim \alt_{\Zchar} V\, t^{\degree}$. We sometimes abuse notations, and write
            \begin{equation*}
                \dim \extAlg V = \sum_{\degree} \dim \alt V\, t^{\degree},
            \end{equation*}
            as the generating function of all the dimensions.
        \end{remark}

        \begin{corollary}\label{cor:dimensions-take-sums-to-products}
            Assume $\cD$ is 1-semiadditive. Let $V, W \in \cD\dbl$. Then, as generating functions
            \begin{equation*}
                \dim \extAlg (V \oplus W) = (\dim \extAlg V) (\dim \extAlg W).
            \end{equation*}
        \end{corollary}
        \begin{proof}
            Concretely, we need to prove that
            \begin{equation*}
                \sum_{\degree} \dim \alt (V\oplus W)\, t^{\degree} = (\sum_{\degree} \dim \alt V\, t^{\degree})(\sum_{\degree} \dim \alt W\, t^{\degree}).
            \end{equation*}
            Equivalently, that
            \begin{equation*}
                \dim \alt (V\oplus W) = \sum_{i+j = \degree} (\dim \alt[i] V)(\dim\alt[j] W).
            \end{equation*}
            This follows by \cref{cor:exterior-takes-sums-to-products}. 
        \end{proof}

        \begin{remark}
            If $\cD\in \CAlg(\PrL)$ is $1$-semiadditive, then $\deg \colon \MM \to \NN$ is $\cD$-adjointable\footnote{sometimes also called $\cD$-semi-affine}, i.e.\ $\deg^* \colon \cD^{\NN} \to \cD^{\MM}$ is internal left adjoint in $\Mod_{\cD}$. This follows from \cite[Proposition~2.32]{BCSY-Fourier} since the fiber of $\deg$ at each point is 1-finite.
        \end{remark}

        \begin{lemma}\label{lem:induced-character-formula-for-dimensions}
            Assume that $\cD$ is 1-semiadditive. Then $\dim \extAlg V$ is isomorphic to the composition
            \begin{equation*}
                \ounit_{\cD}[\NN] \xto{\THH_{\cD}(\deg^*)} \ounit_{\cD}[\L\MM] \xto{\mscr{X}_{\T V}} \ounit_{\cD}[\ZZ].
            \end{equation*}
        \end{lemma}
        \begin{proof}
            This is essentially the induced character formula \cite[Theorem 5.20]{Carmeli-Cnossen-Ramzi-Yanovski-2022-characters}. The functor $\extAlg V\shift{1} = \deg_! \T V\shift{1}$ decomposes as
            \begin{equation*}
                \cD[\NN] \xto{\deg^*} \cD[\MM] \xto{\T V\shift{1}} \Gr_{\ZZ} \cD.
            \end{equation*}
            The result now follows by taking $\THH_{\cD}$.
        \end{proof} 

        \begin{corollary}\label{cor:dim-depends-only-on-dim}
            Assume that $\cD$ is 1-semiadditive. Then $\dim \extAlg V$ depends only on $\dim(V) \in (\ounit_{\cD})^{\B\vee_k \Ck}$.
        \end{corollary}
        \begin{proof}
            This follows from \cref{lem:induced-character-formula-for-dimensions,thm:braiding-depends-only-on-dim}.
        \end{proof}

        \begin{corollary}\label{cor:dim-exterior-internal}
            Assume $\cC$ is 1-semiadditive. Let $A \in \Ab$ and $\Zchar \in \Map\Enull(A, \Sigma\cC\units)$. Let $V \in \cC\dbl$ and $a\in A$. Then $\dim \extAlg V\shift{a}$ identifies with
            \begin{equation*}
                \dim \extAlg_{\Zchar} V\shift{a} \coloneqq \THH_{\cC}(\extAlg_{\Zchar \circ a}V\shift{1}) \colon \ounit_{\cC}[\NN] \to \ounit_{\cC}[\ZZ].
            \end{equation*}
        \end{corollary}
        \begin{proof}
            This follows by \cref{lem:induced-character-formula-for-dimensions,lem:braiding-characters-agree}.
        \end{proof}
        Similarly:
        \begin{corollary}\label{cor:dim-exterior-internal-multiple}
            Assume $\cC$ is 1-semiadditive. Let $A \in \Ab$ and $\Zchar \in \Map\Enull(A, \Sigma \cC\units)$. Let $V_1,\dots,V_r \in \cC\dbl$ and $a_1,\dots, a_r \in A$. Then $\dim \extAlg (\bigoplus_i V_i\shift{a_i})$ identifies with
            \begin{equation*}
                \ounit_{\cC}[\NN] \xto{\THH_{\cC}[\otimes V_i\shift{e_i}]} \ounit_{\cC}[\ZZ^r] \to \ounit_{\cC}[\ZZ].
            \end{equation*}
        \end{corollary}
        \begin{proof}
            This follows by \cref{lem:induced-character-formula-for-dimensions,lem:braiding-characters-agree-multiple}.
        \end{proof}

        \begin{notation}
            When $\Zchar \colon \ZZ \to \Sigma^2 \ounit_{\cC}\units$ is not clear from context we denote the dimension in $\Gr^{\Zchar}_{\ZZ} \cC$ by $\dim_{\Zchar}$.
        \end{notation}
        
        \begin{corollary}\label{cor:generating-functions-dimensions}
            Let $\Zchar_1,\dots, \Zchar_r \colon \ZZ \to \Sigma^2 \ounit_{\cC}\units$. Then the generating function
            \begin{equation*}
                (\dim \extAlg_{\Zchar_1} V) \cdots (\dim \extAlg_{\Zchar_r} V)
            \end{equation*}
            depends only on $\sum_i \dim_{\Zchar_i} (V\shift{1}) \in (\ounit_{\cC})^{\vee_k\B\Ck}$.
        \end{corollary}

        \begin{proof}
            Define $\Zchar \colon \ZZ^r \to \Sigma^2 \ounit_{\cC}\units$ as the product of $\Zchar_i$. Then the commutative diagram
            \begin{equation*}
                \begin{tikzcd}
                    \ZZ & {\ZZ^r} \\
                    & {\Sigma\cC\units}
                    \arrow["{j_i}", hook, from=1-1, to=1-2]
                    \arrow["{\Zchar_i}"', from=1-1, to=2-2]
                    \arrow["\Zchar", from=1-2, to=2-2]
                \end{tikzcd}
            \end{equation*} 
            induces a symmetric monoidal functor $(j_i)_! \colon \Gr^{\Zchar_i}_{\ZZ} \cC \to \Gr^{\Zchar}_{\ZZ^r} \cC$, sending $V\shift{1}$ to $V\shift{e_i}$. By the formulas in \cref{cor:dim-exterior-internal,cor:dim-exterior-internal-multiple}
            \begin{equation*}
                \dim \extAlg V\shift{e_i} = \dim \extAlg_{\Zchar_i} V.
            \end{equation*}
            Now, by \cref{cor:dimensions-take-sums-to-products},
            \begin{equation*}
                \dim \extAlg(\bigoplus V_i\shift{e_i}) = \bigsqcap_i \dim \extAlg V_i \shift{e_i} = \bigsqcap_i \dim \extAlg_{\Zchar_i} V.
            \end{equation*}
            By \cref{cor:dim-depends-only-on-dim}, $\dim \extAlg(\bigoplus V_i\shift{e_i})$ depends only on 
            \begin{equation*}
                \dim(\bigoplus V_i\shift{e_i}) = \sum_i \dim (V_i\shift{e_i}) \in (\ounit_{\cC})^{\vee_k \B\Ck}.
            \end{equation*}
            Again using the symmetric monoidal functor $(j_i)_! \colon \Gr^{\Zchar_i}_{\ZZ} \cC \to \Gr^{\Zchar}_{\ZZ^r} \cC$, we get that $\dim(V\shift{e_i}) = \dim_{\Zchar_i}(V\shift{1})$ as needed.
        \end{proof}

%%%%%%%%%%%%%%%%%%%%%%%%%%%%%%%%%%%%%%%%%%%%%%%%%%%%%%%%%%%%%%%%%%%%%%%%%%%%%%%%
%%%%%%%%%%%%%%%%%%%%%%%%%%%%%%%%%%%%%%%%%%%%%%%%%%%%%%%%%%%%%%%%%%%%%%%%%%%%%%%%
 \section{Graded categories and orientation}
 \label{sec:Galois-closed-unit}
%%%%%%%%%%%%%%%%%%%%%%%%%%%%%%%%%%%%%%%%%%%%%%%%%%%%%%%%%%%%%%%%%%%%%%%%%%%%%%%%
%%%%%%%%%%%%%%%%%%%%%%%%%%%%%%%%%%%%%%%%%%%%%%%%%%%%%%%%%%%%%%%%%%%%%%%%%%%%%%%%
    
    Throughout this section we fix a presentably symmetric monoidal, $\infty$-semiadditive, $(\SS, \chrHeight)$-oriented category $\cC$ (see \cref{Galois-closed}). As before, when working $p$-typically we assume that it is $(\SS_{(p)},\chrHeight)$-oriented and will write $\In$ for $\Inprime$ and $\pinD$ for $\widehat{\pi}_{\chrHeight+1}(\SS_{(p)})$ trusting the reader to infer the relevant $p$ from the context. 
    
    We start in \cref{subsec:graded-categories-and-orientability} by showing that there exists a universal $\cC[\omega^{(0)}_{\SS}]$ Galois extension of $\cC$, such that $\Mod_{\cC[\omega^{(0)}_{\SS}]}$ is $(\SS,\chrHeight+1)$-oriented. It is constructed as a homotopy twisted graded category. In particular, $\Mod_{\cC}$ is $(\SS,\chrHeight+1)$-oriented if and only if it is the trivial Galois extension, or equivalently, its braiding is trivial.

    This motivates us to investigate the braiding of $\cC[\omega^{(0)}_{\SS}]$, which can be analyzed independently for each $\pichar \in \pinD$. In line with the theme of this paper, we do so by examining the associated braiding characters.
    
    In \cref{subsec:low-heights} we show that, at height $\le 4$, the reduced $\Ck$-action on the monoidal dimension of  $\ounit_{\cC}\shift{1} \in \Gr^{\pichar}_{\ZZ} \cC$ is trivial for any $k$. We use this and the results of \cref{sec:graded-braiding} to get generating functions for the dimensions of exterior algebras under additive assumptions, and to completely compute the braiding characters in the categorical case.
    
    In \cref{subsec:chromatic-braiding-character} we study the case $\cC = \ModEn$ for any height, and demonstrate that the  reduced $\Cn[p^k]$-action on the dimension of any object in $\Gr^{\pichar}_\ZZ (\ModEn)\dbl$ is trivial. As a result, we compute the braiding characters for all $\pichar$ at any height, and all primes.
    This provides some support for the orientability conjecture of $\Mod_{\ModEn}$ \cite[Conjecture~7.10]{BCSY-Fourier}, which, in our language, predicts that $\Gr^{\pichar}_\ZZ \ModEn \simeq \Gr_\ZZ \ModEn$ for any $\pichar \in \pinD$.

    %%%%%%%%%%%%%%%%%%%%%%%%%%%%%%%%%%%%%%%%%%%%%%%%%%%%%%%%%%%%%%%%%%%%%%%%%%%%%%%%
    \subsection{The orientable extension}
    \label{subsec:graded-categories-and-orientability}
    %%%%%%%%%%%%%%%%%%%%%%%%%%%%%%%%%%%%%%%%%%%%%%%%%%%%%%%%%%%%%%%%%%%%%%%%%%%%%%%%

        In this subsection we construct a cyclotomically-closed Galois extension of $\cC$ as a homotopy twisted graded category. 
        Recall that since $\cC$ is $(\SS,\chrHeight)$-oriented we have a map $\In \to \cC\units$ which we think of as the embedding of spherical roots of unity.
       
        \begin{definition}\label{def:zeta}
            Rolling the cofiber sequence 
            \begin{equation*}
                \Sigma \In \to \Inplus\to \pinD,
            \end{equation*}
            we get a map 
            \begin{equation*}
                \zeta = \zeta_{\chrHeight} \colon \pinD \to \Sigma^2 \In \to \Sigma^2\ounit_\cC\units \to \Sigma\cC\units.
            \end{equation*}
    
            We define $\cC[\omega^{(0)}_{\SS}] \coloneq \Gr^{\zeta}_{\pinD} \!\! (\cC)$\footnote{Or $\cC[\omega^{(0)}_{\SS_{(p)}}] \coloneqq \Gr^{\zeta}_{\pinD} \!\! (\cC)$ when working $p$-typically.}.
        \end{definition}
    
        Let $\pin \coloneqq \pi_{\chrHeight+1}(\SS)$ be the $(\chrHeight+1)$-st stable stem. 
        We will show that $\cC[\omega^{(0)}_{\SS}]$ is a $\B^{\chrHeight+1}\pin$-Galois extension, that $\Mod_{\cC[\omega^{(0)}_{\SS}]}$ it is $(\SS,\chrHeight+1)$-oriented, and that it is the universal such extension in an appropriate sense. See \cref{subsec:intro-galois-closed} for further justification of this notation.
        
        We first show it admits a natural $\B^{\chrHeight+1}\pin$-action.
        The map $\zeta \colon \pinD \to \Mod_{\cC}\units$ has a mate, which we also denote 
        \begin{equation*}
            \zeta \colon \Mod_{\cC}[\pinD] \to \Mod_{\cC}.
        \end{equation*}
    
        \begin{lemma}\label{lem:image-under-Fourier-and-zeta}
            The commutative algebra $\cC^{\B^{\chrHeight+1}\pin} \in \CAlg(\Mod_{\cC}^{\B^{\chrHeight+2}\pin})$ is sent to $\cC[\omega^{(0)}_{\SS}]$ under the composition of symmetric monoidal functors
            \begin{equation*}
                \Mod_{\cC}^{\B^{\chrHeight+2}\pin} \xgets[{\,\smash{\raisebox{0.9ex}{\ensuremath{\scriptstyle\sim}}}\,}]{\mscr{F}} \Mod_{\cC}[\pinD] \xto{\zeta} \Mod_{\cC},
            \end{equation*}
            where $\mscr{F}$ is the semiadditive Fourier transform (using that $\Mod_{\Mod_{\cC}}$ is $(\ZZ, \chrHeight+2)$-oriented).
        \end{lemma}
    
        \begin{proof}
            The map of connective spectra $\pinD \colon \pinD \to 0$, is sent under the height $(\chrHeight+2)$-Pontryagin duality (i.e.\ $\tau_{\ge 0}\hom(-,\In[\chrHeight+2])$) to the map $0 \to \B^{\chrHeight+2}\pin$. By the functoriality of the Fourier transform, we have a commutative diagram of symmetric monoidal functors
            \begin{equation*}
                \begin{tikzcd}
                    {\Mod_{\cC}[\pinD]} & {\Mod_{\cC}^{\B^{\chrHeight+2}\pin}} \\ \\
                    {\Mod_{\cC}} & {\Mod_{\cC},}
                    \arrow["{\mscr{F}}", "\sim"', from=1-1, to=1-2]
                    \arrow["{(\pinD)_!}", from=1-1, to=3-1]
                    \arrow["{0^*}", from=1-2, to=3-2]
                    \arrow["{\mscr{F}}", equals, from=3-1, to=3-2]
                \end{tikzcd}
            \end{equation*}
            and $(\pinD)_!$ is the colimit functor and $0^*$ is the functor forgetting the $\B^{\chrHeight+1}\pin$-action.
            Taking right adjoints, we see that $\mscr{F}^{-1}(0_* \cC) = \mscr{F}^{-1}(\cC^{\B^{\chrHeight+1}\pin})$ is isomorphic to $(\pinD)^* \cC$ which is the constant symmetric monoidal functor $\pinD \to \Mod_{\cC}$ with value $\cC$.
    
            The claim now follows by composing with $\zeta$.
        \end{proof}
    
        \begin{corollary}
            $\cC[\omega^{(0)}_{\SS}]$ admits a natural $\B^{\chrHeight+1}\pin$-action.
        \end{corollary}
        \begin{proof}
            The symmetric monoidal category $\cC^{\B^{\chrHeight+1}\pin}$ admits a $(\B^{\chrHeight+1}\pin \times \B^{\chrHeight+1}\pin)$-action. Therefore the symmetric monoidal functor of \cref{lem:image-under-Fourier-and-zeta}, which extends to
            \begin{equation*}
                \CAlg(\Mod_{\cC}^{\B^{\chrHeight+2}\pin})^{\B^{\chrHeight+2}\pin} \to \CAlg_{\cC}(\PrL)^{\B^{\chrHeight+2}\pin},
            \end{equation*}
            sends $\cC^{\B^{\chrHeight+1}\pin}$ to $\cC[\omega^{(0)}_{\SS}]$ equipped with a $\B^{\chrHeight+1}\pin$-action.
        \end{proof}
    
        \begin{proposition}\label{prop:cyc-is-Galois}
            $\cC[\omega^{(0)}_{\SS}]$ is a $\B^{\chrHeight+1}\pin$-Galois extension of $\cC$.
        \end{proposition}
    
        \begin{proof}
            By construction and \cite[3.12]{BMCSY-cycloshift}, it suffices to show that $\cC^{\B^{\chrHeight+1}\pin}\in \Mod_\cC^{\B^{\chrHeight+2}\pin}$ is a $\B^{\chrHeight+1}\pin$-Galois extension which is immediate as $\B^{\chrHeight+1}\pin$ is dualizable in $\Mod_\cC$. 
        \end{proof}
    
        \begin{remark}
            Note that by \cite[Theorem 5.15]{BCSY-Fourier} we have that $\Mod_\cC$ is $(\SS,\chrHeight+1)$-oriented if and only if there exists a lifting
            \begin{equation*}
                \begin{tikzcd}
                   {\Sigma \In } & {\Sigma \ounit_\cC\units} & {\cC\units} \\
                   {\Inplus}
                   \arrow[from=1-1, to=1-2]
                   \arrow[from=1-1, to=2-1]
                   \arrow[from=1-2, to=1-3]
                   \arrow[dashed, from=2-1, to=1-3]
                \end{tikzcd}
            \end{equation*}
            where the map $\In  \to \ounit_{\cC}\units$ is the orientation and $\Sigma \In \to \Inplus$ is the connected cover map.   
        \end{remark}

        \begin{lemma}\label{lem:cyc_is_cyc_closed}
            The category $\Mod_{\cC[\omega^{(0)}_{\SS}]}$ is $(\SS,\chrHeight+1)$-oriented.
        \end{lemma}
    
        \begin{proof}
            Recall from \cref{prop:picard-of-Day} and from the proof of \cref{thm:twFun-are-all-lifts} that there exists a cofiber sequence
            \begin{equation*}
                \cC\units \to (\cC[\omega^{(0)}_{\SS}])\units \to \pinD,
            \end{equation*}
            and the induced map $\pinD \to \Sigma \cC\units$ is $\zeta$.
            Consider the following commutative square
            \begin{equation*}
                \begin{tikzcd}
                    \pinD & {\Sigma^2 \In} \\
                    \pinD & {\Sigma \cC\units.}
                    \arrow["{\zeta}", from=1-1, to=1-2]
                    \arrow[equals, from=1-1, to=2-1]
                    \arrow[from=1-2, to=2-2]
                    \arrow["{\zeta}", from=2-1, to=2-2]
                \end{tikzcd}
            \end{equation*}
            Taking fibers, we get a map
            \begin{equation*}
                i_{\chrHeight+1} \colon \Inplus \to (\cC[\omega^{(0)}_{\SS}])\units.
            \end{equation*}
            and by construction we have that.
            \begin{equation*}
                \Omega i_{\chrHeight+1} \colon \In \simeq \Omega \Inplus \to \Omega (\cC[\omega^{(0)}_{\SS}])\units \simeq \ounit_{\cC}\units
            \end{equation*}
            is identified with the canonical map $(\ounit_{\cC}\units)\pitor \to \ounit_{\cC}\units$, as required.
        \end{proof}

        The following proposition allows us to view $\cC[\omega^{(0)}_\infty]$ as the universal $(\SS,\chrHeight+1)$-oriented extension of $\cC$:
    
        \begin{proposition}\label{prop:C-cyclotoimcally-closed-iff-gradedpi-is-trivial}
               The following are equivalent:
               \begin{enumerate}
                   \item\label{item:cyc-closed} The category $\Mod_\cC$ is $(\SS,\chrHeight+1)$-orientable.
                   \item\label{item:trivial-Galois} $\cC[\omega^{(0)}_{\SS}]$ is a trivial Galois extension, i.e.\ $\cC[\omega^{(0)}_{\SS}] \simeq \cC^{\B^{\chrHeight+1} \pin}$.
                   \item\label{item:trivial-braiding} $\cC[\omega^{(0)}_{\SS}] \simeq \Gr_{\pinD} (\cC)$, i.e.\ the braiding of $\cC[\omega^{(0)}_{\SS}]$ is trivial.
               \end{enumerate}
        \end{proposition}

        \begin{proof}
            The equivalence of (\labelcref{item:trivial-Galois}) and (\labelcref{item:trivial-braiding}) follows by the semiadditive Fourier transform.
            
            We show now that (\labelcref{item:cyc-closed}) is equivalent to (\labelcref{item:trivial-braiding}).  
            By the universal property of the Thom construction, $\cC[\omega^{(0)}_{\SS}] = \Gr^{\zeta}_{\pinD}\!\!(\cC)$ is isomorphic to $\Gr_{\pinD} (\cC)$ if and only if the map $\zeta \colon \pinD \to \Sigma \cC\units$ is null as an $\EE_{\infty}$-map. 
            Using the cofiber sequence 
            \begin{equation*}
                \pinD \to \Sigma^2 \In \to \Sigma \Inplus,
            \end{equation*}
            $\zeta$ is null if and only if there exists a lift
            \begin{equation*}
                \begin{tikzcd}
                    \pinD \\
                    {\Sigma^2 \In} & {\Sigma\cC\units.} \\
                    {\Sigma \Inplus}
                    \arrow[from=1-1, to=2-1]
                    \arrow[from=1-1, to=2-2]
                    \arrow[from=2-1, to=2-2]
                    \arrow[from=2-1, to=3-1]
                    \arrow[dashed, from=3-1, to=2-2]
                \end{tikzcd}
            \end{equation*}
        \end{proof}
    
       \begin{remark}
           One can use \cref{prop:C-cyclotoimcally-closed-iff-gradedpi-is-trivial} in order to construct the cyclotomic closure of $\cC$ as in \cref{def:cyc-closure}. More generally by adding homotopy groups of $\Inplus$ inductively one can construct the the cyclotomic closure of $\cC$ for any $\infty$-semi-additive category which is virtually $(\Fp,\chrHeight)$-orientable.
       \end{remark}

    %%%%%%%%%%%%%%%%%%%%%%%%%%%%%%%%%%%%%%%%%%%%%%%%%%%%%%%%%%%%%%%%%%%%%%%%%%%%%%%%
    \subsection{Braiding characters of oriented categories in low heights}
    \label{subsec:low-heights}
    %%%%%%%%%%%%%%%%%%%%%%%%%%%%%%%%%%%%%%%%%%%%%%%%%%%%%%%%%%%%%%%%%%%%%%%%%%%%%%%%

        Thinking of $\SS$ as a symmetric monoidal category, the element $1 \in \SS$ is invertible. Therefore it has a dimension
        \begin{equation*}
            \dim(1) \in \End_{\SS}(1)^{\B\TT} = (\Omega \SS)^{\B\TT}.
        \end{equation*}
        By \cite[Proposition~3.20]{CSY-cyclotomic}, $\dim(1)$ is identified with $\eta$ in $\Omega \SS$. 
        \begin{definition}
            We denote by $\eta \in (\Omega\SS)^{\B\TT}$ the dimension of $1\in \SS$ with the corresponding $\TT$-action.
        \end{definition}

        Let $\cC \in \CAlg(\PrL)$ and $\Zchar \colon \ZZ \to \Sigma \cC\units$. By \cref{lem:cofiber-of-pic-Day} and the proof of \cref{thm:twFun-are-all-lifts}, the fiber of this map is $(\Gr^{\Zchar}_{\ZZ}\cC)\units$.
        As $(\Gr^{\Zchar}_{\ZZ}\cC)\units$ is a spectrum, the dimension of $\ounit_{\cC}\shift{1}$ is given by the action of $\eta \in (\Omega \SS)^{\B\TT}$
        \begin{equation*}
            \dim(\ounit_{\cC}\shift{1}) = \eta \cdot \ounit_{\cC}\shift{1} \qin \End_{(\Gr^{\Zchar}_{\ZZ}\cC)\units}(\ounit_{\cC}\shift{1})^{\B\TT} = (\Omega(\Gr^{\Zchar}_{\ZZ}\cC)\units)^{\B\TT} \simeq (\ounit_{\cC}\units)^{\B\TT}.
        \end{equation*}

        \begin{notation}
            Let $\pichar \in \pinD$. Then $\pichar \in \Inplus$ is invertible. Denote its dimension by
            \begin{equation*}
                \eta_{\pichar} \coloneqq \dim(\pichar) \qin \End_{\Inplus}(1)^{\B\TT} = (\Omega \Inplus)^{\B\TT} \simeq (\In)^{\B\TT}.
            \end{equation*}
            Equivalently it is $\eta \cdot \pichar$, for $\eta \in (\Omega \SS)^{\B\TT}$ and $\pichar \in \Inplus$.
        \end{notation}

        \begin{lemma}
            Let $\pichar \in \pinD$ and $\cC \in \CAlg(\PrL)$ be $\infty$-semiadditive, $(\SS,\chrHeight)$-orientable. Then, under the inclusion map $\In \to \ounit_{\cC}\units$, the dimension of $\ounit_{\cC}\shift{1} \in \Gr^{\pichar}_{\ZZ} \cC$ is
            \begin{equation*}
                \dim(\ounit_{\cC}\shift{1}) = \eta_{\pichar} \qin (\ounit_{\cC}\units)^{\B\TT}.
            \end{equation*}
        \end{lemma}
        \begin{proof}
            Let $\Inpi \in \cnSp$ be the pullback
            \begin{equation*}
                \begin{tikzcd}
                    \Inpi & \Inplus \\
                    \ZZ & \pinD.
                    \arrow[from=1-1, to=1-2]
                    \arrow[from=1-1, to=2-1]
                    \arrow["\lrcorner"{anchor=center, pos=0.125}, draw=none, from=1-1, to=2-2]
                    \arrow[from=1-2, to=2-2]
                    \arrow["\pichar", from=2-1, to=2-2]
                \end{tikzcd}
            \end{equation*}
            Equivalently, it is the fiber of the map
            \begin{equation*}
                \ZZ \xto{\pichar} \pinD \to \Sigma^2 \In.
            \end{equation*}
            The map of connective spectra $\Inpi \to \Inplus$ sends $1$ to $\pichar$. Taking fibers of the commutative diagram
            \begin{equation*}
                \begin{tikzcd}
                    \ZZ & {\Sigma^2 \In} \\
                    \ZZ & {\Sigma \cC\units,}
                    \arrow["{\Zchar_{\pichar}}", from=1-1, to=1-2]
                    \arrow[equals, from=1-1, to=2-1]
                    \arrow[hook, from=1-2, to=2-2]
                    \arrow["{\Zchar_{\pichar}}", from=2-1, to=2-2]
                \end{tikzcd}
            \end{equation*}
            we get a symmetric monoidal map $\Inpi \to \Gr^{\pichar}_{\ZZ}\cC$ sending $1\in \Inpi$ to $\ounit_{\cC}\shift{1}$. Therefore, both the dimension of $\pichar \in \Inplus$ and the dimension of $\ounit_{\cC}\shift{1}\in \Gr^{\pichar}_{\ZZ} \cC$, agree with the dimension of $1 \in \Inpi$.
        \end{proof}

        \begin{corollary}\label{cor:dim-of-shift-in-cyc-closed}
            Let $\pichar \in \pinD$ and $\cC \in \CAlg(\PrL)$ be $\infty$-semiadditive, $(\SS,\chrHeight)$-orientable. Then, working in $\Gr^{\pichar}_{\ZZ} \cC$, for any $V \in \cC\dbl$
            \begin{equation*}
                \dim(V\shift{1}) \simeq \eta_{\pichar} \cdot \dim(V) \qin \End(\ounit_{\cC})^{\B\TT}.
            \end{equation*}
        \end{corollary}

        Moreover, using \cref{prop:dim-is-transfer}, we can understand the $\TT$-action on $\eta_{\pichar} \in \In$ --- it is just (the loops of) the composition
        \begin{equation*}
            \Sigma \SS[\B\TT] \xto{\Tr} \SS \xto{\pichar} \Inplus.
        \end{equation*}

        In low heights, everything is much simpler: Choose $\pichar \in \pinD$ and consider $\Gr^{\pichar}_{\ZZ} \cC$ as in \cref{def:dual-stable-stems}.
        \begin{lemma}\label{lem:trivial-T-action-0-1}
            Assume $\chrHeight \le 1$. Let $V \in \cC\dbl$ and assume that the $\TT$-action on $\dim(V)$ is trivial. Then the $\TT$-action on $\dim(V\shift{1})$ is trivial.     
        \end{lemma}
        \begin{proof}
            It suffices to show that the $\TT$-action on $\eta_{\pichar} \in \In$ is trivial. This follows as there are no non-trivial maps $\B\TT \to \In$ for $\chrHeight \le 1$. 
        \end{proof}

        \begin{lemma}\label{lem:trivial-T-action-2}
            Assume $\chrHeight = 2$. Let $V \in \cC\dbl$ and assume that the $\TT$-action on $\dim(V)$ is trivial. Then the restricted $\Ck$-action on $\dim(V\shift{1})$ is trivial for any $k$.
        \end{lemma}
        \begin{proof}
            It suffices to show that the $\Ck$-action on $\eta_{\pichar} \in \In[2]$ is trivial. 
            As $\B\TT$ is simply connected, the map $\B\TT \to \In[2]$ factors through $\tau_{\ge 2} \In[2] = \B^2 \QQ/\ZZ$. Therefore the $\Ck$-action maps factors as maps
            \begin{equation*}
                \B\Ck \to \B^2 \QQ/\ZZ,
            \end{equation*}
            or equivalently,
            \begin{equation*}
                \B\Ck \to \B^3 \ZZ.
            \end{equation*}
            As there is no 3-rd cohomology to $\B\Ck$, these maps are null.
        \end{proof}

        Combining both lemmas and \cref{thm:braiding-depends-only-on-dim} we conclude.
        \begin{corollary}
            Assume $\chrHeight \le 4$. Let $V \in \cC\dbl$ and assume that the $\TT$-action on $\dim(V)$ is trivial. Then $\mscr{X}^{\Zchar}_{\T V}$ depends only on $\eta_{\pichar} \cdot \dim(V) \in \ounit_{\cC}$.
        \end{corollary}

        \begin{proof}
            For $\chrHeight \le 2$ it follows from \cref{lem:trivial-T-action-0-1,lem:trivial-T-action-2,thm:braiding-depends-only-on-dim}. For $\chrHeight = 3,4$ it follows since $\pinD[4] = \pinD[5] = 0$.
        \end{proof}

        %%%%%%%%%%%%%%%%%%%%%%%%%%%%%%%%%%%%%%%%%%%%%%%%%%%%%%%%%%%%%%%%%%%%%%%%%%%%%%%%
        \subsubsection*{Dimensions of exterior algebras}
        We can use the triviality of the $\Ck$-actions to compare, in nice cases, dimensions of exterior algebras coming from different elements in $\pinD$. 

        In height 0 everything works, and we get a generalization of a well known generating function (in the case $\cC = \Vect$) for the dimensions of symmetric and exterior powers:
        \begin{proposition}\label{prop:generating-function-ht-0}
            Let $\cC$ be an $\infty$-semiadditive, $(\SS,0)$-oriented\footnote{It is enough to assume we are working 2-typically.} category.
            Let $V \in \cC\dbl$ such that the $\TT$-action on $\dim(V)$ is trivial. Then,
            \begin{equation*}
                (\sum_{\degree} \dim(\Sym V)\, t^{\degree}) (\sum_{\degree} \dim(\alt_{\Kos} V)\, (-t)^{\degree}) = 1.
            \end{equation*}
        \end{proposition}

        \begin{proof}
            Let $\hat{\eta}\in\pinD[1]$ be the unique non-trivial element. By \cref{cor:dim-of-shift-in-cyc-closed,lem:trivial-T-action-0-1}, $\dim_0 (V\shift{1}) = \dim(V) \in \Gr_{\ZZ} \cC$ and $\dim_{\hat{\eta}}(V\shift{1}) = \eta_{\hat{\eta}} \cdot \dim(V) \in \Gr^{\Kos}_{\ZZ} \cC$.\footnote{We sincerely apologize for the notation $\eta_{\hat{\eta}}$. It is a good opportunity to apologize for the excessive use of footnotes in this paper.}
            As an element in $\In[0] = \QQ/\ZZ$, $\eta_{\hat{\eta}}$ is just the action of $\eta \in \Omega \SS$ on $\hat{\eta} \in \In[1]$, which is $(-1) \in \QQ/\ZZ$. 
            Summarizing: $\dim_0(V\shift{1}) + \dim_{\hat{\eta}}(V\shift{1}) = 0$, where $0$ is equipped with the trivial $\TT$-action.
            By \cref{rmrk:height-0-dual-stems}, $\Gr^{\hat{\eta}}_{\ZZ} \cC = \Gr^{\Kos}_{\ZZ} \cC$. Finally, by \cref{cor:generating-functions-dimensions},
            \begin{equation*}
                (\dim \Sym[\bullet] V) \cdot (\dim \extAlg_{\Kos} V) = \dim \extAlg_{0,\Kos} 0 = 1,
            \end{equation*}
            where $\extAlg_{0,\Kos}$ is the exterior algebra functor in $\Gr^{0,\Kos}_{\ZZ^2} \cC$, proving the statement.
        \end{proof}

        \begin{remark}
            The same proof shows that, working in $\Gr^{0,\hat{\eta}}_{\ZZ^2} \cC$,
            \begin{equation*}
                \mscr{X}_{\T (V\shift{e_1} \oplus V\shift{e_2})} = \mscr{X}_{0} = 1.
            \end{equation*}
            Using that $\T (V \shift{e_1} \oplus V\shift{e_2}) = \T V\shift{e_1} \otimes_{\Day} \T V\shift{e_2}$, and the symmetric monoidal functors $\Gr^0_{\ZZ} \cC \to \Gr^{0,\hat{\eta}}_{\ZZ^2} \cC$, $\Gr^{\Kos}_{\ZZ} \cC \to \Gr^{0,\hat{\eta}}_{\ZZ^2} \cC$, we get a relation between $\chi^{0}_{\Tm V}$ and $\chi^{\Kos}_{\Tm V}$.
            This relation is, unfortunately, much less comprehensible.
        \end{remark}

        In higher heights, the orientation does not provide us with minus one. Therefore we can not hope generally that the sum of the dimensions will be 0. E.g. for $\cC = \Mod_{\sVect}$, the dimension $\dim_{\hat{\eta^2}}(\sVect\shift{1})$ is the  $(0|1) $-dimensional super vector space, and $\dim_0(\sVect \shift{1}) \oplus \dim_{\hat{\eta^2}}(\sVect \shift{1})$ is the $(1|1)$-dimensional super vector space, which is non-zero. 
        
        We can try to overcome this by forcing additivity, but even that does not guarantee that the dimension would be sent to -1. Thus, we will force both:

        \begin{notation}
            Denote the generators of the cyclic groups $\pinD[2]$ and $\pinD[3]$ by $\hat{\eta^2}$ and $\hat{\nu}$ respectively.
        \end{notation}

        The proof of the following is exactly as \cref{prop:generating-function-ht-0}, using that $\dim(\hat{\eta^2}) = \hat{\eta} \in \In[1]$ and $\dim(\hat{\nu}) = \hat{\eta^2} \in \In[2]$:
        \begin{lemma}\label{lem:generating-function-ht-1-additive}
            Let $\cC$ be an $\infty$-semiadditive, $(\SS,1)$-oriented category. Assume that $\cC$ is additive and that the map $\In[1] \to \ounit_{\cC}\units$ sends $\hat{\eta}$ to $(-1)\in \ounit_{\cC}\units$.
            Let $V \in \cC\dbl$ such that the $\TT$-action on $\dim(V)$ is trivial. Then,
            \begin{equation*}
                (\sum_{\degree} \dim(\Sym V)\, t^{\degree}) (\sum_{\degree} \dim(\alt_{\hat{\eta^2}} V)\, (-t)^{\degree}) = 1.
            \end{equation*}
        \end{lemma}
        \begin{lemma}\label{lem:generating-function-ht-2-additive}
            Let $\cC$ be an $\infty$-semiadditive, $(\SS,2)$-oriented category. Assume that $\cC$ is additive and that the map $\In[2] \to \ounit_{\cC}\units$ sends $\hat{\eta^2}$ to $(-1)\in \ounit_{\cC}\units$.
            Let $V \in \cC\dbl$ such that the $\TT$-action on $\dim(V)$ is trivial. Then,
            \begin{equation*}
                (\sum_{\degree} \dim(\Sym V)\, t^{\degree}) (\sum_{\degree} \dim(\alt_{\hat{\nu}} V)\, (-t)^{\degree}) = 1.
            \end{equation*}
        \end{lemma}

        \begin{example}
            In the prime $p=2$, at heights $\chrHeight = 1,2$, $\ModEn$ satisfies the above conditions (see \cref{prop:truncated-units}).
        \end{example}

        \begin{remark}\label{rmrk:iterated-dimension}
            A way to bypass this problem in categorical examples such as $\cVectn$, which are never additive, is to consider the iterated dimension. For example, for $\cC = \Mod_{\sVect}$, in both graded categories, corresponding to $0$ and $\hat{\eta^2}$, the object $\sVect \shift{1}$ is 2-dualizable and 
            \begin{equation*}
                \dim^2_0(\sVect\shift{1}) = 1, \qquad \dim^2_\Kos(\sVect\shift{1}) = -1.
            \end{equation*}
            Therefore, one can get a generating function for the iterated dimensions. This is similar in flavor to \cite[Example~4.1.3]{Ganter-Kapranov-2014-exterior-categories}.
        \end{remark}

        %%%%%%%%%%%%%%%%%%%%%%%%%%%%%%%%%%%%%%%%%%%%%%%%%%%%%%%%%%%%%%%%%%%%%%%%%%%%%%%%
        \subsubsection*{Examples of braiding characters}

        Using \cref{lem:character-of-Tm} we can compute by hand the braiding characters in low heights.

        \begin{example}[Day convolution]
            Let $\cC$ be an $\infty$-semiadditive, presentably symmetric monoidal category and $(\SS,\chrHeight)$-oriented for height $\chrHeight \le 2$. Let $V \in \cC\dbl$ and assume that $\TT$ acts trivially on $\dim(V)$. Then, in $\Gr_{\ZZ} \cC$, for $\sigma \in \L\B\Sm\subseteq \L\MM$
            \begin{equation*}
                \chi_{\T V}(\sigma) = (\dim(V))^{\numcyc}\, t^{\degree}.
            \end{equation*}
        \end{example}

        In height 0, we know that the dimension of $\hat{\eta}$ is $(-1)$, so $\dim_{\hat{\eta}}(V\shift{1}) = -\dim(V)$. In particular,
        \begin{example}[Height 0]
            Let $\cC$ be an $\infty$-semiadditive, presentably symmetric monoidal category which is $(\SS,0)$-oriented. Let $V \in \cC\dbl$ and assume that $\TT$ acts trivially on $\dim(V)$. Then, in $\Gr^{\Kos}_{\ZZ} \cC$, for $\sigma \in \L\B\Sm\subseteq \L\MM$
            \begin{equation*}
                \chi^{\Kos}_{\T V}(\sigma) = (-\dim(V))^{\numcyc}\, t^{\degree} = \begin{cases}
                    \dim(V)^{\numcyc}\, t^{\degree}, & \text{$\numcyc$ is even}, \\
                    -\dim(V)^{\numcyc}\, t^{\degree}, & \text{$\numcyc$ is odd}.
                \end{cases}.
            \end{equation*}
        \end{example}

        \begin{remark}
            The same can be said for the braiding character in $\Gr^{\hat{\eta^2}}_{\ZZ} \cC$ and $\Gr^{\hat{\nu}}_{\ZZ} \cC$, under the assumptions of \cref{lem:generating-function-ht-1-additive} and \cref{lem:generating-function-ht-2-additive}, respectively. 
            
            In particular, this allows us to compute the braiding character in $\ModEn$ for $\chrHeight \le 4$.
        \end{remark}

        \begin{example}[Height 1]
            In $\Mod_{\sVect}$, which is $(\SS,1)$-oriented, $\dim_{\hat{\eta^2}}(\ounit_{\cC}\shift{1}) = \Pi \field$ is the $(0|1)$-dimensional super vector space. For any $\mcal{V} \in \Mod_{\sVect}\dbl$
            \begin{equation*}
                \chi_{\T \mcal{V}}(\sigma) = (\Pi \dim(\mcal{V}))^{\otimes \numcyc}\, t^{\degree} = \begin{cases}
                    \dim(\mcal{V})^{\otimes \numcyc}\, t^{\degree}, & \text{$\numcyc$ is even}, \\
                    \Pi \dim(\mcal{V})^{\otimes \numcyc}\ t^{\degree}, & \text{$\numcyc$ is odd}
                \end{cases}
                \qin \sVect.
            \end{equation*}
        \end{example}

        \begin{remark}
           One can write similar results in height $2$ in terms of $\dim_{\hat{\nu}}(\cVectn[2]\shift{1})$.
        \end{remark}

        \begin{remark}
            Similarly to \cref{rmrk:iterated-dimension}, working in categorical examples, one can consider the iterated character. These behave more like the characters in the additive case. A study of these is carried out in \cite{Keidar-Ragimov-2025-alternating}.
        \end{remark}

    %%%%%%%%%%%%%%%%%%%%%%%%%%%%%%%%%%%%%%%%%%%%%%%%%%%%%%%%%%%%%%%%%%%%%%%%%%%%%%%%
    \subsection{The chromatic braiding character}
    \label{subsec:chromatic-braiding-character}
    %%%%%%%%%%%%%%%%%%%%%%%%%%%%%%%%%%%%%%%%%%%%%%%%%%%%%%%%%%%%%%%%%%%%%%%%%%%%%%%%

        In the categorical setting, for example in $\Vectn[\chrHeight+1] = \Mod_{\Vectn}$, the monoidal character map is a categorical operation that decreases the height at the cost of adding a free loops:
        \begin{equation*}
            \chi \colon ((\Vectn[\chrHeight+1])\dbl)^A = (\Mod_{\Vectn}\dbl)^A \to (\Vectn)^{\L A}
        \end{equation*}
        for any space $A$.

        In the chromatic setting, we have an analogous story, using the transchromatic character theory developed by Hopkins--Kuhn--Ravenel, Stapleton and Lurie (\cite{Hopkins-Kuhn-Ravanel-2000-HKR,Stapleton-2013-HKR,Lurie-2019-Elliptic3}). There exists a $\Kn$-local algebra, called the ($\Kn$-local) splitting algebra at height $\chrHeight$  \footnote{That is, it is the $\Kn[\chrHeight]$-localization of the splitting algebra constructed in \cite{Stapleton-2013-HKR} and \cite{Lurie-2019-Elliptic3}. Throughout this paper, we always consider the splitting algebra in its monochromatically localized form.},
        \begin{equation*}
            \LKn \Enp \to C^{\chrHeight+1}_{\chrHeight},
        \end{equation*}
        and a transchromatic character map
        \begin{equation*}
            \chi^{\HKR} \colon \Enp^A \to (C^{\chrHeight+1}_{\chrHeight})^{\L A}
        \end{equation*}
        for any $p$-finite space $A$. Since $C^{\chrHeight+1}_{\chrHeight}$ is a nonzero $\Kn$-local algebra, we can relate it to a Morava $E$-theory via the chromatic Nullstellensatz \cite{Burklund-Schlank-Yuan-2022-Nullstellensatz}.
        
        \begin{notation}
            A map $R \to \En(L)$ in $\CAlg(\SpTn)$, for $L$ algebraically closed, is called a \emph{geometric point of $R$}.
        \end{notation}

        Thus, we obtain a refined transchromatic character map
        \begin{equation*}
            \chi^{\HKR} \colon \Enp^A \to (C^{\chrHeight+1}_{\chrHeight})^{\L A} \to \En(K)^{\L A},
        \end{equation*}
        where $K$ is an algebraically closed field, and the second map arises from a chosen geometric point.

        This connects with the monoidal character, by using again the chromatic Nullstellensatz, which allows us to relate $(\ModEn)\dbl$ with $\Enp$.
        
        \begin{notation}
            Let $\KTnp \coloneqq \LTnp \K$ denote the $\Tnp$-localized $\K$-theory functor. We also write $\KTnp(\En) \coloneqq \KTnp((\ModEn)\dbl)$.
        \end{notation}
        
        There is a canonical map of commutative monoids
        \begin{equation*}
            (\ModEn)\dblspace \to \KTnp(\En).
        \end{equation*}
        By chromatic redshift for $\En$ \cite{Yuan-2024-redshift-En}, we know that $\KTnp(\En)$ is nontrivial, so we fix a geometric point
        \begin{equation*}
            \KTnp(\En) \to \Enp(L),
        \end{equation*}
        for some algebraically closed field $L$.
        \begin{definition}
            Using this geometric point, we define the \emph{decategorification map}
            \begin{equation*}
                \de \colon (\ModEn)\dblspace \to \KTnp(\En) \to \Enp(L).
            \end{equation*}
        \end{definition}

        Let $\LKn \Enp(L) \to C^{\chrHeight+1}_{\chrHeight}(L)$ be the associated $\Kn$-local splitting algebra at height $\chrHeight$.
        
        \begin{construction}\label{cons:character-diagram}
            Choose a geometric point $C^{\chrHeight+1}_{\chrHeight}(L) \to \En(K')$. By \cite[Lemma 7.14, Theorem 7.2(2)]{Burklund-Schlank-Yuan-2022-Nullstellensatz}, we obtain geometric points $\En \to \En(K)$ and $\En(K') \to \En(K)$. For $A \in \pspcpi$, this gives rise to a (not necessarily commutative) diagram:
            \begin{equation*}
                \begin{tikzcd}
                    {((\ModEn)\dblspace)^A} && {\Enp(L)^A} \\
                    && {C^{\chrHeight+1}_{\chrHeight}(L)^{\L A}} \\
                    {\En^{\L A}} && {\En(K)^{\L A}}
                    \arrow["\de", from=1-1, to=1-3]
                    \arrow["\chi", from=1-1, to=3-1]
                    \arrow["\chi^{\HKR}", from=1-3, to=2-3]
                    \arrow[from=2-3, to=3-3]
                    \arrow[from=3-1, to=3-3]
                \end{tikzcd}
            \end{equation*}
            We refer to this as the \emph{character diagram}.
        \end{construction}

        While we expect this diagram to commute in general (up to an isomorphism of $\En(K)$), thereby providing an interpretation of transchromatic character theory as a form of monoidal character theory, we do not know how to prove this in full generality. However, we verify that the two paths agree in special cases (\cref{cor:commutativity-of-character-diagram-suspension-In}, \cref{lem:commutativity-of-character-diagram-En}). These cases serve as the main tool in our analysis of the chromatic braiding character.

        By \cref{lem:character-of-Tm,thm:braiding-depends-only-on-dim}, the study of the braiding character reduces to understanding the character of the restricted action of $\ZZ/p^k \subseteq \Sm[p^k]$ for $p^k$-permutation representations for all $k$. Let $\pichar \in \pinD$. The permutation representation
        \begin{equation*}
            \Tm[p^k] \En\shift{1} \qin (\Gr^{\pichar}_\ZZ \ModEn)^{\B\ZZ/p^k}
        \end{equation*}
        is concentrated in a single degree $p^k$, and thus corresponds to a $\ZZ/p^k$-action on $\En$ in $\ModEn$. We show that the decategorification of this representation is given by
        \begin{equation*}
            \Tm[p^k](\omega(\pichar)) = \omega(\pichar)^{p^k} \qin \Enp(L)^{\B\ZZ/p^k},
        \end{equation*}
        where
        \begin{equation*}
            \omega \colon \pinD = \pi_0(\Inplus) \to \pi_0 \Enp(L)\units
        \end{equation*}
        is the $\pi_0$ of the orientation map.

        The image of $\omega$ lies in $\{\pm 1\}$ (\cref{def:truncated-units}), therefore $\omega(\pichar)$ admits a categorification (i.e.\ a preimage under $\de$):
        \begin{equation*}
            \de(\En) = 1, \qquad \de(\Sigma \En) = -1.
        \end{equation*}
        Using the commutativity of the character diagram in these cases, together with \cref{cor:no-action-on-dim-En}, we conclude that the character of $\Tm \En\shift{1}$ agrees with the character of one of the following:
        \begin{equation*}
            \Tm \En \qor \Tm (\Sigma \En).
        \end{equation*}

        We begin in \cref{subsubsec:character-diagram} by showing that the character diagram commutes in these specific cases. Then in \cref{subsubsec:truncated-units} we study the image of the orientation map $\omega$, classifying when it takes the values $1$ and $-1$ (\cref{prop:truncated-units}). Finally, in \cref{subsubsec:T-action-chromatic}, we deduce a formula for the braiding character in the chromatic setting (\cref{thm:chromatic-braiding-character}).

        %%%%%%%%%%%%%%%%%%%%%%%%%%%%%%%%%%%%%%%%%%%%%%%%%%%%%%%%%%%%%%%%%%%%%%%%%%%%%%%%
        \subsubsection{Compatibility between monoidal and transchromatic characters}
        \label{subsubsec:character-diagram}

        We investigate certain relations between transchromatic character theory and monoidal character theory, expressed through the commutativity of the character diagram \cref{cons:character-diagram}. Some of the results in this subsection likely hold in greater generality than we state. However, since the special cases presented here suffice for our purposes, we do not pursue the general case further.

        This subsubsection is relatively technical. We begin with a few notations that will help us keep track:
        
        \begin{notation}
            When working in an $\infty$-commutative monoid $M$ (e.g.\ a $\Tn$-local spectrum) in the sense of \cite[Definition~5.10]{Harpaz-2020-ambidexterity}, we denote the operation of semiadditive integration in $M$ by~$\int^M$.
        \end{notation}
        
        \begin{notation}
            We work with Morava $E$-theories over various algebraically closed fields. When stating a claim that holds for all such theories, we typically write $\En$ to mean $\En(L)$ for some algebraically closed field $L$.
        \end{notation}
        
        \begin{notation}
            Let $L$ be an algebraically closed field and $0 \le t \le \chrHeight$. We denote by
            \begin{equation*}
                \LKn[t]\En(L) \to C^{\chrHeight}_t(L)
            \end{equation*}
            the $\Kn[t]$-local splitting algebra for $\En(L)$ at height $t$ 
            (\cite[\textsection~3.10]{Stapleton-2013-HKR}, 
            \cite{Barthel-Stapleton-2016-centralizers-good-groups}, 
            \cite{Barthel-Stapleton-2018-rings-transchromatic}, \cite[Definition~2.7.12]{Lurie-2019-Elliptic3}).
        
            For any $p$-local $\pi$-finite space $A$, we denote the transchromatic character map by
            \begin{equation*}
                \chi^{\chrHeight - t, \HKR} \colon \En(L)^A \to C^{\chrHeight}_t(L)^{\L^{\chrHeight - t} A},
            \end{equation*}
            and write specifically $\chi^{\HKR} = \chi^{1, \HKR}$.
            We write the evaluation of the character at a specific map $\rho \colon A \to \En(L)$ as
            \begin{equation*}
                \chi^{\chrHeight - t, \HKR}_{\rho} \colon \L^{\chrHeight - t} A \to C^{\chrHeight}_t(L).
            \end{equation*}
            
            When there is no need to distinguish between different $E$-theories, we omit the field $L$ from the notation.
        \end{notation}

        Recall the following claim from \cite{BMCSY-cardinality}:
        \begin{proposition}[{\cite[Proposition~11]{BMCSY-cardinality}}]\label{prop:KTnp-is-infty-semiadditive}
            The functor $\KTnp \colon \Cat_{\Lnf, \mfin[p]} \to \SpTnp$ is $p$-typically $\infty$-semiadditive.
        \end{proposition}
        
        \begin{corollary}\label{cor:integral-in-K-is-colim}
            Let $A\in \pspcpi$ and let $V \in ((\ModEn)\dbl)^A$ be an $A$-local system of dualizable objects. Then 
            \begin{equation*}
                \int_A^{\KTnp(\En)} [V] = [\colim_A V] \qin \KTnp(\En).
            \end{equation*}
        \end{corollary}
        
        \begin{proof}
            As $\KTnp$ is $p$-typically $\infty$-semiadditive, it respects $p$-typical semiadditive integrals. Therefore it sends
            \begin{equation*}
                \colim_A V = \int^{\ModEn}_A V \qin \ModEn
            \end{equation*}
            to 
            \begin{equation*}
                [\colim_A V ] = \int_A^{\KTnp(\En)} [V] \qin \KTnp(\En)
            \end{equation*}
            as claimed.
        \end{proof}
        
        \begin{lemma}\label{lem:integral-Enp-to-En}
            Let $A \in \pspcpi$, $V \in ((\ModEn)\dbl)^A$. Then the following three semiadditive integrals
            \begin{equation*}
                \int_{\L_p A}^{\En} \chi_V \in \pi_0 \En, \qquad
                \int_{A}^{\Enp(L)} \de(V) \in \pi_0 \Enp(L), \qquad
                \int_{\L_p A}^{C^{\chrHeight+1}_{\chrHeight}(L)} \chi^{\HKR}_{\de(V)} \in \pi_0 C^{\chrHeight+1}_{\chrHeight}(L)
            \end{equation*}
            land in $\ZZ$, and they are all identified.
        \end{lemma}
        
        \begin{proof}
            By \cref{cor:no-action-on-dim-En}, for each $a \in A$, the value $\de(V)(a) \in \pi_0 \Enp(L)$ is an integer and agrees with $\dim(V_a) \in \ZZ \subseteq \pi_0 \En$.
            
            As $A$ is $p$-typically $\pi$-finite, $\Lp A \simeq \L A$. Thus, by the induced character formula \cite[Theorem 5.20]{Carmeli-Cnossen-Ramzi-Yanovski-2022-characters},
            \begin{equation*}
                \int_{\Lp A}^{\En} \chi_V \simeq \int_{\L A}^{\En}\chi_V \simeq \dim(\colim_A V).
            \end{equation*}
            By \cite[Proposition~13]{BMCSY-cardinality}, the dimension map
            \begin{equation*}
                \dim \colon \K(\En) \to \ZZ \subseteq \pi_0 \En
            \end{equation*}
            is an isomorphism on $\pi_0$. Under this identification, we identify $\dim(\colim_A V)$ with $[\colim_A V] \in \pi_0 \K(\En)$. The $\Tnp$-localization map gives an identification of $\ZZ\subseteq \K(\En)$ with $\ZZ \subseteq \KTnp(\En)$, so we identify $\dim(\colim_A V)$ with $[\colim_A V] \in \pi_0 \KTnp(\En)$. Now, using \cref{cor:integral-in-K-is-colim},
            \begin{equation*}
                [\colim_A V] = \int_A^{\KTnp(\En)} [V] \qin \pi_0 \KTnp(\En).
            \end{equation*}
            Using the geometric point $\KTnp(\En)\to \Enp(L)$ and \cite[Proposition~2.1.14(1)]{CSY-teleambi}, we get
            \begin{equation*}
                \int_A^{\KTnp(\En)} [V] = \int_A^{\Enp(L)} \de(V),
            \end{equation*}
            which together implies the first equality. The second equality follows by \cite[Remark~7.4.8, Remark~7.4.9]{Lurie-2019-Elliptic3}, see also \cite[Theorem~A]{Ben-Moshe-2024-shifted-semiadditive}.
        \end{proof}
 
        \begin{definition}\label{def:iterated-character-of-a-map}
            Let $A$ be a $\pi$-finite $p$-space and let $\rho \colon A \to \Sigma\In$ be a map of spaces.  
            For $k \ge 1$, we define the \emph{$k$-fold character} of $\rho$ to be the composition
            \begin{equation*}
                \chi^k_\rho \colon \L^kA \xto{\L^k \rho} \L^k \Sigma\In \xrightarrow{(1+\eta)^k} \Omega^k \Sigma \In \simeq \In[\chrHeight+1-k].
            \end{equation*}
            This is the $k$-fold iterated character of the map $\rho$, as in \cref{def:character-of-a-map}.
        \end{definition}

        \begin{notation}\label{not:higher-character}
            By an abuse of notation, we will usually think of $\rho \colon A \to \Sigma \In$ as the composition (of maps of spaces)
            \begin{equation*}
                A \xto{\rho} \Sigma \In \to \Enp\units \to \Enp.
            \end{equation*}
            Similarly, we consider the the $k$-fold character of $\rho$ as a map to $\En[\chrHeight+1-k]$:
            \begin{equation*}
                \chi^k_{\rho} \colon \L^k A \to \In[\chrHeight+1-k] \to \En[\chrHeight+1-k]\units \to \En[\chrHeight+1-k].
            \end{equation*}
        \end{notation}
            
        \begin{definition}
            Let $\alpha \in \pinD = \pi_0 (\Inplus)$. The composition
            \begin{equation*}
                \B\ZZ/p^k \to \B\Sm[p^k] \into \MM \xto{\T \alpha} \Inplus \xto{\alpha^{-p^k} \cdot (-)} \Inplus
            \end{equation*}
            is a pointed map. We denote its factorization through the connected cover by
            \begin{equation*}
                \rho_{\alpha,k} \colon \B\ZZ/p^k \to \Sigma\In.
            \end{equation*}
        \end{definition}

         \begin{definition}
            Let $A$ be a $\pi$-finite $p$-space and $\rho \colon A \to \Sigma\In$ be a map of spaces. We denote by $\En{}[\rho]\in (\ModEn)^{A}$ the local system
            \begin{equation*}
                A \xto{\rho} \Sigma\In \to \Sigma \En\units \to \ModEn
            \end{equation*}
            
            In particular, for $\alpha \in \pinD$, we denote by $\En{}[\rho_{\alpha,k}]$ the local system 
            \begin{equation*}
                \B\ZZ/p^k \xto{\rho_{\alpha,k}} \Sigma \In \to \B\En\units \to \ModEn.
            \end{equation*}
        \end{definition}

        \begin{remark}
            Let $A$ be $\pi$-finite $p$-space and $\rho \colon A \to \Sigma\In$ be a map of spaces. Then $\chi_{\rho}$ as in \cref{not:higher-character} agrees with the character $\chi_{\En{}[\rho]}$.
        \end{remark}

        \begin{lemma}
            Let $\pichar \in\pinD$. We identify the automorphism group of $\En\shift{p^k} \in \Gr^{\pichar}_{\ZZ} \cC$ with $\Aut(\En)$, via the $(-p^k)$ shift map. Then $(\En\shift{1})^{\otimes p^k} \in \B\Aut(\En\shift{p^k})^{\B\ZZ/p^k}$ identifies with $\En{}[\rho_{\pichar, k}] \in \B\Aut(\En)^{\B\ZZ/p^k}$. 
        \end{lemma}

        \begin{proof}
            This is a reformulation of \cref{prop:braiding-of-1<1>-determines-the-character}. 
        \end{proof}

        \begin{lemma}\label{lem:HKR-factors-through-a-map-of-spectra}
            Let $X$ be an $\Ind$-$\pi$-finite $p$-local connective spectrum and $f \colon X \to \En\units$ be a map of spectra. Let $0\le t\le \chrHeight-1$. Then the transchromatic character 
            \begin{equation*}
                \chi^{\chrHeight-t,\HKR}_f \colon \L^{\chrHeight-t}X \to C^{\chrHeight}_t \qin \spc
            \end{equation*}
            factors through a map of connective spectra $\L^{\chrHeight-t} X \to (C^{\chrHeight}_t)\units$.
        \end{lemma}

         \begin{proof}
            For a commutative ring spectrum $R$, we denote by $\spc_{/R}$ the category of spaces over $\Omega^{\infty} R$. It is equipped with the symmetric monoidal structure given by the cartesian product of spaces, using the commutative monoid structure on $\Omega^{\infty} R$ induced by the multiplication on $R$.
            
            Let $k \in \NN$. Then the map $\En^{(-)}\to \LKn[t]\En^{(\B\ZZ/p^k)^{\chrHeight-t} \times\Map((\B \ZZ/p^k)^{\chrHeight-t},(-))}$, induces a symmetric monoidal map
            \begin{equation*}
                \Map((\B\ZZ/p^k)^{\chrHeight-t},(-)) \colon \spc_{/\En} \to \spc_{/\LKn[t]\En^{(\B\ZZ/k)^{\chrHeight-t}}}.
            \end{equation*}
            Composing with the natural map $\LKn[t]\En^{(\B\ZZ/p^k)^{\chrHeight-t}}\to C^{\chrHeight}_t$ and using the compatibility of these maps, we get a symmetric monoidal map
            \begin{equation*}
                \colim_k\Map((\B\ZZ/p^k)^{\chrHeight-t},(-)) \colon \spc_{/\En} \to \spc_{/C^{\chrHeight}_t}.
            \end{equation*}
            Restricting to $(\pspcpi)_{/\En}$ and using \cite[Proposition 3.4.7]{Lurie-2019-Elliptic3} we get a symmetric monoidal map 
            \begin{equation*}
                \L^{\chrHeight-t} \colon (\pspcpi)_{/\En} \to (\pspcpi)_{/C^{\chrHeight}_t}.
            \end{equation*}
            Using that the free loops functor $\L$ commutes with filtered colimits since $\TT$ is compact, we may take $\Ind$ to both sides. As the functor is symmetric monoidal it induces a map on commutative algebras
            \begin{equation*}
                \L^{\chrHeight-t} \colon \CAlg(\Ind(\pspcpi)_{/\En}) \to \CAlg(\Ind(\pspcpi)_{/C^{\chrHeight}_t}),
            \end{equation*}
            sending $f \colon X \to \En$ to $\chi^{\chrHeight-t, \HKR}_f \colon \L^{\chrHeight-t} X \to C^{\chrHeight}_t$. In particular, if $X$ is a spectrum and $f$ is a map of commutative monoids, then $\chi^{\chrHeight-t,\HKR}_f$ factors through $(C^{\chrHeight}_t)\units$.
          \end{proof}   

        \begin{remark} 
            By the proof of \cref{lem:HKR-factors-through-a-map-of-spectra}, for any $\pi$-finite $p$-space $X$ and 
            \begin{equation*}
                \rho \colon X\to \In \to \En,
            \end{equation*}
            the transchromatic character map $\chi^{\chrHeight-t, \HKR}_{\rho}\colon \L^{\chrHeight-t}X \to C^{\chrHeight}_t$ factors through a map of spectra $\L^{\chrHeight-t}\In \to (C^{\chrHeight}_t)\units$.  
        \end{remark}

        \begin{lemma}\label{lem:orbits-of-free-loops-is-loops}
            Let $X$ be a spectrum. Then $(\Lnc^k\, X)_{h\TT^k} \simeq \Sigma^{-k} X$, where $\Lnc$ is the non-connective free loops spectrum. Under the isomorphism
            \begin{equation*}
                \Lnc^k \, X \simeq (\Sigma^{-1} \oplus \id_{\cnSp})^k\, X \simeq \bigoplus_{i=0}^k \bigoplus_{k \choose i} (\Sigma^{-i} X),
            \end{equation*} 
            the canonical map to the orbit is $(1+\eta)^k \simeq \bigoplus_{i=0}^k \bigoplus_{k \choose i} \eta^i$.  
        \end{lemma}

        \begin{proof}
            It is enough to prove it for the case $k=1$. As $\Lnc\, X = \hom(\SS[\TT], X)$ with the $\TT$-action only coming from $\SS[\TT]$, and as $\SS[\TT]$ is dualizable, it is enough to prove it for $X = \SS$, in which case it is of the dual decomposition in \cref{prop:trace-is-counit}.
        \end{proof}

        For the next two lemmas we will use some notation from \cite{Lurie-2019-Elliptic3}. Mainly:
        \begin{enumerate}
            \item Given a commutative ring spectrum $A$ and a preoriented $p$-divisible group $\mbf{G}$ over $A$, we write $A_\mbf{G}$ for the \emph{$\mbf{G}$-tempered function spectrum} as defined in \cite[Construction~4.0.3]{Lurie-2019-Elliptic3}.
            \item For a complex periodic and $\Kn$-local commutative ring $A$, we denote its Quillen $p$-divisible group by $\mbf{G}^Q_A$ (see \cite[\textsection~2.4]{Lurie-2019-Elliptic3}). 
        \end{enumerate}
        
        \begin{lemma}\label{lem:Qformal-group-pulled-back-to-Qformal-group}
            Let $f \colon A \to B$ be a map of complex periodic and $\Kn$-local commutative rings. Then $\mbf{G}^Q_A$ is pulled back under $f$ to $\mbf{G}^Q_B$
        \end{lemma}
        
        \begin{proof}
            By \cite[Theorem~3.5.5]{Lurie-2019-Elliptic3} and by the definition \cite[Notation~4.0.1, Consrtuction~4.0.3]{Lurie-2019-Elliptic3}, it suffice to show that the canonical preorientation gives an isomorphism of the $p$-divisible functors $\cT\op \to \CAlg_B$ 
            \begin{equation*}
                B_{f^* \mbf{G}^Q_A}\simeq B_{\mbf{G}^Q_B}.
            \end{equation*}
            By \cite[Theorem~4.7.1]{Lurie-2019-Elliptic3} we have an isomorphism
            \begin{equation*}
                A_{\mbf{G}^Q_A}^X \otimes_{A} B \simeq B_{f^* \mbf{G}^Q_A}^X 
            \end{equation*}
            By \cite[Theorem~4.2.5]{Lurie-2019-Elliptic3}, for any $p$-local $\pi$-finite space $X$, in particular for any $X \in \cT$, the Atiyah-Segal comparison map (\cite[Construction~4.2.2]{Lurie-2019-Elliptic3}) provides an isomorphism
            \begin{equation*}
                A_{\mbf{G}^Q_A}^X\simeq A^{X}.
            \end{equation*} 
            By \cite[Theorem~4.4.16(1)]{Lurie-2019-Elliptic3}, for any $X \in \cT$ there is an isomorphism
            \begin{equation*}
                A^X \otimes_A B \simeq B^X.
            \end{equation*}
            Applying \cite[Theorem 4.2.5]{Lurie-2019-Elliptic3} again we see that for any $X \in \cT$
            \begin{equation*}
                B_{f^* \mbf{G}^Q_A}^X \simeq B^X \simeq B_{\mbf{G}^Q_B}^X. 
            \end{equation*}
        \end{proof}

        \begin{lemma}\label{lem:HKR-vs-twice-HKR}
            Let $A \in \pspcpi$ and $0 \le s \le t \le \chrHeight$. Then for every geometric point $C^{\chrHeight}_t(K) \to \En[t](K_t)$ there exists a map of $\Kn[s]$-local commutative algebras $C^{\chrHeight}_s(K) \to C^t_s(K_t)$ rendering the following diagram commutative
            \begin{equation*}
                \begin{tikzcd}
                    {\En(K)^{A}} & {C^{\chrHeight}_t(K)^{\L^{\chrHeight-t}A}} & {\En[t](K_t)^{\L^{\chrHeight-t}A}} \\ \\
                    {C^{\chrHeight}_s(K)^{\L^{\chrHeight-s}A}} && {C^t_s(K_t)^{\L^{\chrHeight-s}A}.}
                    \arrow["{\chi^{\chrHeight-t,\HKR}}", from=1-1, to=1-2]
                    \arrow["{\chi^{\chrHeight-s,\HKR}}", from=1-1, to=3-1]
                    \arrow[from=1-2, to=1-3]
                    \arrow["{\chi^{t-s,\HKR}}", from=1-3, to=3-3]
                    \arrow[from=3-1, to=3-3]
                \end{tikzcd}
            \end{equation*}
        \end{lemma}
        
        \begin{proof}
            Let $\mbf{G}_{\LKn[t]\En(K)}$ be the pullback of $\mbf{G}^{\mcal{Q}}_{\En(K)}$ to $\LKn[t]\En(K)$. By \cite[Corollary~2.5.7(2)]{Lurie-2018-Elliptic2}, the canonical orientation of $\mbf{G}^{\mcal{Q}}_{\En(K)}$ induces an injective preorientation $\mbf{G}^{\mcal{Q}}_{\LKn[t]\En(K)} \to \mbf{G}_{\LKn[t]\En(K)}$. By the proof of \cite[Proposition~2.5.6]{Lurie-2018-Elliptic2}, it is identified with the connected component at the unit so sits in a cofiber sequence
            \begin{equation*}
                0 \to \mbf{G}^{\mcal{Q}}_{\LKn[t](\En)} \to \mbf{G}_{\LKn[t](\En)} \to \mbf{G}^{\mrm{\'et}}_{\LKn[t](\En)} \to 0.
            \end{equation*}
            By \cite[Corollary~2.18]{Stapleton-2013-HKR},\cite[Definition~2.7.12]{Lurie-2019-Elliptic3} the pullback of $\mbf{G}_{\LKn[t]}(\En)$ to $C^{\chrHeight}_t(K)$ splits and has a trivial \'etale part 
            \begin{equation*}
                \mbf{G}_{C^{\chrHeight}_t(K)} \simeq \underline{(\QQ_p/\Zp)}^{\chrHeight-t} \oplus \mbf{G}^{\mcal{Q}}_{C^{\chrHeight}_t(K)}.
            \end{equation*}
            This in turn, by \cref{lem:Qformal-group-pulled-back-to-Qformal-group}, is pulled to the $p$-divisible group
            \begin{equation*}
                \underline{(\QQ_p/\Zp)}^{\chrHeight-t} \oplus \mbf{G}^{\mcal{Q}}_{\En[t](K_t)}
            \end{equation*}
            on $\En[t](K_t)$. 
            Finally, repeating the above arguments, this is pulled to the $p$-divisible group 
            \begin{equation*}
                \underline{(\QQ_p/\Zp)}^{\chrHeight-s} \oplus \mbf{G}^{\mcal{Q}}_{C^t_s(K_t)}
            \end{equation*}
            on $C^t_s(K_t)$. As $C^{\chrHeight}_s(K)$ is initial with respect to this property (\cite[Corollary~2.18]{Stapleton-2013-HKR}, \cite[Definition~2.7.12, Proposition 2.7.15]{Lurie-2019-Elliptic3}), we get a factorization
            \begin{equation*}
                \begin{tikzcd}
        	        {\En(K)_{\mbf{G}^{\mcal{Q}}}} & {C_t^{\chrHeight}(K)_{\mbf{G}^{\mcal{Q}}\oplus\QQ_p/\ZZ_p^{\chrHeight-t}}} & {\En[t](K_t)_{\mbf{G}^{\mcal{Q}}\oplus\QQ_p/\ZZ_p^{\chrHeight-t}}} \\
        	        {C_s^n(K)_{\mbf{G}^{\mcal{Q}}\oplus\QQ_p/\ZZ_p^{\chrHeight-s}}} && {C_s^t(K_t)_{\mbf{G}^{\mcal{Q}}\oplus\QQ_p/\ZZ_p^{\chrHeight-s}}.}
        	        \arrow[from=1-1, to=1-2]
                    \arrow[from=1-1, to=2-1]
                    \arrow[from=1-2, to=1-3]
                    \arrow[from=1-3, to=2-3]
                    \arrow[from=2-1, to=2-3]
                \end{tikzcd}
            \end{equation*}
            Applying \cite[Theorem 4.2.5, Theorem 4.3.2]{Lurie-2019-Elliptic3} gives the desired diagram.
        \end{proof}

        \begin{proposition}\label{prop:HKR-factors-through-trace}
            Let $A\in \pspcpi$ and $\rho \colon A \to \Inplus$ be a map of spaces. Using the notations of \cref{cons:character-diagram}, the map
            \begin{equation*}
                \L A \xto{\chi^\HKR_\rho} C^{\chrHeight+1}_{\chrHeight}(L)\to \En(K)
            \end{equation*}
            we get by the right vertical composition of \cref{cons:character-diagram} factors as 
            \begin{equation*}
                \L A \xto{\L \rho} \L\Inplus \xto{1+\eta} \In \to \En(K)\units \to \En(K).
            \end{equation*}
        \end{proposition}

        \begin{proof}
            Since $\Inplus$ is $\Ind$-$\pi$-finite, by \cref{lem:HKR-factors-through-a-map-of-spectra} we have that $\chi_\rho^\HKR$ factors through a map of spectra 
            \begin{equation*}
                f \colon \L \Inplus \to  C^{\chrHeight+1}_{\chrHeight}(L)\units.
            \end{equation*}
            As $\Mod_{\En(K)}$ is $(\chrHeight+1)$-connected (\cref{def: cat connectedness}), using that $\L\Inplus$ is $\Ind$-$\pi$-finite, the composition 
            \begin{equation*}
                \L\Inplus \xto{f} C^{\chrHeight+1}_{\chrHeight}(L)\units \to \En(K)\units,
            \end{equation*}
            factors, as a map of spectra, through $\In = (\En(K)\units)\pitor$. 
            % \begin{equation*}
            %     \L A\to \L \In \simeq \In \oplus \In[\chrHeight-1] \xto{1+\eta}  \In[\chrHeight-1] \to \En[\chrHeight-1](L)\units \to \En[\chrHeight-1](L)
            % \end{equation*}
            % where all the maps apart from the last and the first are maps of spectra.
            It is left to show that the map 
            \begin{equation*}
                f \colon \Inplus \times \In \simeq \L\Inplus \to \In
            \end{equation*}
            is the trace map, i.e.\ $1+\eta$. 
            Using the universal property of $\In$ (\cref{rmrk:maps-to-In}), this is equivalent to the map being classified by:
            \begin{itemize}
                \item  $1 \in \ZZ_2 = \widehat{\pi}_{\chrHeight}(\In)$ and $\eta \in \{1,\eta\} = \widehat{\pi}_{\chrHeight}(\Inplus)$ when $p=2$, or
                \item $1 \in \Zp = \widehat{\pi}_{\chrHeight}(\In)$ (and $\widehat{\pi}_{\chrHeight}(\Inplus) = 0$) when $p \neq 2$. 
            \end{itemize}
            
            We first show that for any $p$, the restriction to $\In$ is classified by an invertible element in $\Zp$.
            Let $\tau \colon \Sigma^{\chrHeight+1}\ZZ/p \to \Sigma^{\chrHeight+1}\QQ_p/\Zp \to \Inplus$ be the natural map. Then, the composition of maps of connective spectra
            \begin{equation*}
                \Sigma^{\chrHeight+1} \ZZ/p \xto{\tau} \Inplus \to \En\units
            \end{equation*}
            chooses the higher $p$-th root of unity $\omega^{(\chrHeight)}_p$ (see \cite{CSY-cyclotomic}) and therefore is not nullhomotopic. 
            By \cite[Corollary 8.12]{Burklund-Schlank-Yuan-2022-Nullstellensatz} and \cref{lem:integral-Enp-to-En}, we see that 
            \begin{equation*}
                0=\dim(\colim_{\B^{\chrHeight+1}\ZZ/p} \En[\chrHeight][\tau])=\int_{\B^{\chrHeight+1}\ZZ/p}^{\En[\chrHeight+1](L)} \tau=\int_{\L\B^{\chrHeight+1}\ZZ/p}^{C^{\chrHeight+1}_{\chrHeight}(L)}\chi_\tau^\HKR.
            \end{equation*}
            Therefore, the image of the element that classifies our map, under the modulo $p$ map
            \begin{equation*}
                \Zp = \Hom_{\Ab}(\QQ_p/\Zp, \QQ_p/\Zp) \to \Hom_{\Ab}(\ZZ/p,\QQ_p/\Zp) \simeq \ZZ/p,
            \end{equation*}
            is non zero. That is, it is an invertible element in $\Zp$. Thus, up to a choice of a different orientation $\In\to \En\units$, the restriction to $\In$ is the identity.
            
            We are now left with the case $p=2$. We will show that the restriction to $\Inplus$ is classified by $\eta \in  \pi_{\chrHeight}(\Inplus)$. Assume otherwise, i.e.\ that it is the trivial map. Let 
            $I \coloneqq \hom(\tau_{\le 1}\SS/2, \Inplus)$ and consider the map 
            \begin{equation*}
                g \colon I \to \Inplus,
            \end{equation*}
            dual to the natural modulo 2 and truncation map $\SS \to \tau_{\le 1}\SS/2$. 
            By assumption, the composition
            \begin{equation*}
                I \xto{g} \Inplus \to \L\Inplus \xto{f} \In
            \end{equation*}
            is trivial, or equivalently,
            \begin{equation*}
                I \to \L I \xto{\L g} \L\Inplus \xto{f} \In
            \end{equation*}
            is trivial.
            By \cref{lem:HKR-vs-twice-HKR}, we have a commutative diagram 
            % \begin{equation*}
            %     \begin{tikzcd}
            %         {C^{\chrHeight+1}_{\chrHeight}(L)^{\L I}} & {\En[\chrHeight](K)^{\L I}} \\
            %         {C^{\chrHeight+1}_{0}(L)^{\L^{\chrHeight+1} I}} & {C^{\chrHeight}_0(K)^{\L^{\chrHeight+1}I}\ ,}
            %         \arrow[from=1-1, to=1-2]
            %         \arrow[from=1-1, to=2-1]
            %         \arrow[from=1-2, to=2-2]
            %         \arrow[from=2-1, to=2-2]
            %     \end{tikzcd}
            % \end{equation*}
            \begin{equation*}
                \begin{tikzcd}
                    {\Enp^{I}} & {(C^{\chrHeight+1}_{\chrHeight})^{\L I}} & {\En(K)^{\L I}} \\ \\
                    {(C^{\chrHeight+1}_0)^{\L^{\chrHeight+1}I}} && {C^{\chrHeight}_0(K)^{\L^{\chrHeight+1}I}.}
                    \arrow["{\chi^{\HKR}}", from=1-1, to=1-2]
                    \arrow["{\chi^{\chrHeight+1,\HKR}}", from=1-1, to=3-1]
                    \arrow[from=1-2, to=1-3]
                    \arrow["{\chi^{\chrHeight,\HKR}}", from=1-3, to=3-3]
                    \arrow[from=3-1, to=3-3]
                \end{tikzcd}
            \end{equation*}
            As $C^{\chrHeight}_0(K)$ is rational and $\L^{\chrHeight+1} I$ is $\pi$-finite, the map of spaces 
            \begin{equation*}
                \chi^{\chrHeight+1, \HKR}_g \colon \L^{\chrHeight+1} I \to C^{\chrHeight}_0(K)
            \end{equation*}
            factors through $\pi_0 \L^{\chrHeight+1} I \simeq \ZZ/2^{\oplus(\chrHeight+2)}$. By \cref{lem:HKR-factors-through-a-map-of-spectra}, this map factors through a map of spectra $\ZZ/2^{\oplus(\chrHeight+2)} \to C^{\chrHeight}_0(K)\units$. 
            The space of maps $\ZZ/2^{\oplus (\chrHeight+2)} \to C^{\chrHeight}_0(K)\units$ is discrete, therefore the map
            \begin{equation*}
                \L^{\chrHeight+1} I \to C^{\chrHeight}_0(K)\units
            \end{equation*}
            is $\TT^{\chrHeight+1}$-invariant, and induces a map of connective spectra
            \begin{equation*}
                (\L^{\chrHeight+1} I)_{h\TT^{\chrHeight+1}} \to C^{\chrHeight}_0(K)\units.
            \end{equation*}
            
            Consider the $(\chrHeight+1)$-connective cover map
            \begin{equation*}
                \Sigma^{\chrHeight+1} \ZZ/2 \to I
            \end{equation*}
            and the commutative diagram
            \begin{equation*}
                \begin{tikzcd}
                    {\L^{\chrHeight+1}\Sigma^{\chrHeight+1}\ZZ/2} & {\L^{\chrHeight+1}I} & {\Lnc^{\chrHeight+1}I} \\
                    {(\L^{\chrHeight+1}\Sigma^{\chrHeight+1}\ZZ/2)_{h\TT^{\chrHeight+1}}} & {(\L^{\chrHeight+1}I)_{h\TT^{\chrHeight+1}}} & {(\Lnc^{\chrHeight+1}I)_{h\TT^{\chrHeight+1}}.}
                    \arrow[from=1-1, to=1-2]
                    \arrow[from=1-1, to=2-1]
                    \arrow[from=1-2, to=1-3]
                    \arrow[from=1-2, to=2-2]
                    \arrow[from=1-3, to=2-3]
                    \arrow[from=2-1, to=2-2]
                    \arrow[from=2-2, to=2-3]
                \end{tikzcd}            
            \end{equation*}
            Using \cref{lem:orbits-of-free-loops-is-loops}, and that $\L^{\chrHeight+1}\Sigma^{\chrHeight+1} \ZZ/2 = \Lnc^{\chrHeight+1} \Sigma^{\chrHeight+1}\ZZ/2$, we write it equivalently as
            \begin{equation*}
                \begin{tikzcd}
                    {\L^{\chrHeight+1}\Sigma^{\chrHeight+1}\ZZ/2} & {\L^{\chrHeight+1}I} & {\Lnc^{\chrHeight+1}I} \\
                    {\ZZ/2} & {(\L^{\chrHeight+1}I)_{h\TT^{\chrHeight+1}}} & {\Sigma^{-(n+1)}I.}
                    \arrow[from=1-1, to=1-2]
                    \arrow["{(\eta+\id)^{\chrHeight+1}}", from=1-1, to=2-1]
                    \arrow[from=1-2, to=1-3]
                    \arrow[from=1-2, to=2-2]
                    \arrow["{(\eta+\id)^{\chrHeight+1}}", from=1-3, to=2-3]
                    \arrow[from=2-1, to=2-2]
                    \arrow[from=2-2, to=2-3]
                \end{tikzcd}
            \end{equation*}
            Note that the vertical left map and the bottom horizontal composition are isomorphisms on $\pi_0$. In particular, there is a $\ZZ/2$-summand in $\pi_0((\L^{\chrHeight+1} I)_{h\TT^{\chrHeight+1}})$ such that
            \begin{enumerate}
                \item\label{item:onto} the map $\pi_0 (\L^{\chrHeight+1} I) \to \pi_0((\L^{\chrHeight+1} I)_{h\TT^{\chrHeight+1}}) \onto \ZZ/2$ is given by $\pi_0 (\eta + \id)^{\chrHeight+1}$, and
                \item\label{item:into} the composition 
                \begin{equation*}
                    \ZZ/2 \into \pi_0((\L^{\chrHeight+1} I)_{h\TT^{\chrHeight+1}}) \to \pi_0 C^{\chrHeight}_0(K) \units
                \end{equation*}
                agrees with
                \begin{equation*}
                    \ZZ/2 \xto{\pi_0 \chi^{\chrHeight+1, \HKR}_{\tau}} \pi_0 C^{\chrHeight+1}_0 \to \pi_0 C^{\chrHeight}_0(K).
                \end{equation*}
                where $\tau \colon \Sigma^{\chrHeight+1} \ZZ/2 \to \Inplus$ is the natural map.
            \end{enumerate}
            
            By our assumption, the composition
            \begin{equation*}
                \pi_0 \L^{\chrHeight} I \to \pi_0 \L^{\chrHeight+1} I  \to \pi_0 C^{\chrHeight}_0(K)\units
            \end{equation*}
            is zero, therefore so is the composition with the projection to the $\ZZ/2$-summand
            \begin{equation*}
                \pi_0 \L^{\chrHeight} I \to \pi_0 \L^{\chrHeight+1} I  \onto \ZZ/2 \into \pi_0 \L^{\chrHeight+1} I \to \pi_0 C^{\chrHeight}_0(K)\units.
            \end{equation*}
            By (\labelcref{item:onto}), the composition
            \begin{equation*}
               \ZZ/2 = \pi_0 \L^{\chrHeight} I \to \pi_0 \L^{\chrHeight+1} I \onto \ZZ/2
            \end{equation*}
            is an isomorphism, therefore the restriction
            \begin{equation*}
                \ZZ/2 \into \pi_0 \L^{\chrHeight+1} I \to \pi_0 C^{\chrHeight}_0(K)\units
            \end{equation*}
            is zero.            
            But by (\labelcref{item:into}), this composition is not zero, in contradiction.
        \end{proof} 

       % We now relate the transchromatic character theory, and categorical character theory, for maps factoring through $\In$, via the chromatic Nullstellensatz and decent.

        \begin{lemma}\label{lem:spherical-roots-of-unity-Nullstellensatz-shift}
            Following the notations of \cref{cons:character-diagram}, the orientation map, composed with the geometric point
            \begin{equation*}
                \Sigma \In \to \Sigma \En\units \to (\ModEn)\units \to \KTnp(\En)\units \to \Enp(L)\units,
            \end{equation*}
            is a $(\tau_{\le \chrHeight} \SS, \chrHeight+1)$-orientation. In particular, we can assume it is the connected cover of the $(\SS,\chrHeight+1)$-orientation map of $\Enp(L)$.
        \end{lemma}
        
        \begin{proof}
            Let $\Sigma^{\chrHeight+1} \QQ_p / \Zp \to \Sigma \In$ denote the $(\chrHeight+1)$-connective cover. Since the induced map
            \begin{equation*}
                \pi_0 \Map(\Sigma \In, \Inplus) \to \pi_0 \Map(\Sigma^{\chrHeight+1} \QQ_p / \Zp, \Inplus)
            \end{equation*}
            is an isomorphism, it suffices to verify the claim after precomposing with it.  
            By \cite[Theorem~5.15]{BCSY-Fourier}, the map
            \begin{equation*}
                \Sigma^{\chrHeight+1} \QQ_p / \Zp \to \Sigma \En\units \to (\ModEn)\units
            \end{equation*}
            is an orientation of $\ModEn$, exhibiting it as having primitive roots of unity of height $\chrHeight+1$.
            
            By \cite[Theorem~B]{BMCSY-cycloshift}, primitive roots of unity of height $\chrHeight+1$ in $\ModEn$ map to the corresponding height $\chrHeight+1$ primitive roots of unity in $\KTn(\En)$, which in turn are sent to the primitive roots of unity of $\Enp(L)$.    
        \end{proof}

         \begin{corollary}\label{cor:commutativity-of-character-diagram-suspension-In}
            Let $A$ be a $\pi$-finite $p$-local space and $\rho \colon A\to \Sigma \In$ be a map of spaces. 
            Then the image of $\En{}[\rho] \in ((\ModEn)\dbl)^A$, under both compositions in the character diagram \cref{cons:character-diagram}, is the same, up to multiplication by an invertible $p$-adic integer.\footnote{I.e.\ up to a change of $(\SS,\chrHeight)$-orientation of $\En(K)$.}
            
            Written differently, $\chi_{\rho}$ in the sense of \cref{def:iterated-character-of-a-map} agrees with $\chi^{\HKR}_{\rho}$ (after mapping both to $\En(K)$).
         \end{corollary}

         \begin{proof}
             This follows from \cref{prop:HKR-factors-through-trace} and \cref{lem:spherical-roots-of-unity-Nullstellensatz-shift}.
         \end{proof}  

         \begin{lemma}\label{lem:commutativity-of-character-diagram-En}
            Let $P \subseteq \Sm$ be a $p$-subgroup. Assume that $P$ is an $\En$-good group in the sense of \cite{Barthel-Stapleton-2016-centralizers-good-groups} (which is implied by \cite[Definition~7.1]{Hopkins-Kuhn-Ravanel-2000-HKR}). Let $V \in (\ModEn)\dbl$. Then the image of $\Tm V \in ((\ModEn)\dbl)^{\B P}$ under the two composition in the character diagram \cref{cons:character-diagram} is the same.
         \end{lemma}

        \begin{proof}
            Since $P$ is an $\En$-good group, by \cite[Corollary 7.1]{Barthel-Stapleton-2016-centralizers-good-groups}, so are its centralizers, and therefore the transchromatic character map 
            \begin{equation*}
                \chi^{\chrHeight,\HKR} \colon \En(K)^{\L \BP}\to C^{\chrHeight}_0(K)^{\L^{\chrHeight+1}\BP}
            \end{equation*}
            is injective on $\pi_0$. Therefore it suffices to check commutativity after composing to $C^{\chrHeight}_0(K)^{\L^{\chrHeight+1} \BP}$. 

            We now show that both compositions arise as instances of the \emph{character of the total power operation}, as constructed in \cite[Theorem 9.1]{Barthel-Stapleton-2017-power-operations}. More precisely, they both coincide with the character applied to $d \coloneqq \dim(V)$ with $G = e$, and restricted along $\L^{\chrHeight+1} \BP \to \Lp^{\chrHeight+1} \B \Sm$.

            First, by \cite[Corollary 10.2]{Barthel-Stapleton-2017-power-operations}, the character of the total power operation evaluated on $d$ is independent of the choice of $\phi$. By \cref{lem:HKR-vs-twice-HKR} $\Tm V$ is sent under the composition 
            \begin{equation*}
                ((\ModEn)\dbl)^{\BP} \xto{\de} \En[n+1](\L)^{\BP} \xto{\chi^{\HKR}} \En(K)^{\L\BP} \xto{\chi^{\chrHeight,\HKR}} C^{\chrHeight}_0(K)^{\L^{\chrHeight+1}\BP}
            \end{equation*}
            to the character of the total power operation applied to $d$. 

            To see that the other composition agrees with this character, we use the formula for the character of the total power operation given just before \cite[Proposition 5.5]{Barthel-Stapleton-2017-power-operations}. Let $[\degree]$ be a set of cardinality $\degree$. An element of
            \begin{equation*}
                \pi_0 \Lp^{\chrHeight+1} \B \Sm = \Hom(\Zp^{\chrHeight+1}, \Sm) / \mrm{conj}
            \end{equation*}
            can be viewed as a $\Zp^{\chrHeight+1}$-action on $[\degree]$. 
            The formula then evaluates the character of the total power operation applied to $d =\dim(V)$ at this $\Zp^{\chrHeight+1}$-set, as $d^r$, where $r$ is the number of $\Zp^{\chrHeight+1}$-orbits in $[\degree]$.

            The claim now follows from \cref{lem:character-of-Tm}, which asserts that the value of $\chi_{\Tm V}$ on a $\Zp$-set of size $m$ is $d^\ell$, where $\ell$ is the number of $\Zp$-orbits in its decomposition. In this case, the stabilizer of the corresponding conjugacy class acts by permuting isomorphic orbits.
        \end{proof}

        %%%%%%%%%%%%%%%%%%%%%%%%%%%%%%%%%%%%%%%%%%%%%%%%%%%%%%%%%%%%%%%%%%%%%%%%%%%%%%%%
        \subsubsection{Truncated units}
        \label{subsubsec:truncated-units}

        By \cref{subsubsec:character-diagram}, since the character diagram commutes for both permutation representations and characters factoring through~$\In$, it follows that the character of $\Tm\En\shift{1} \in \Gr^{\pichar}_{\ZZ} \ModEn$ agrees with the transchromatic character of the image of $\pichar$ under $\pi_0$ of the orientation map (\cref{lem:character-of-rho-is-character-of-Tm}).

        We are thus led to study the image of the orientation map $\In \to \En\units$ on~$\pi_0$. We call elements in this image truncated units, and immediately show that almost always there are no such units.

        \begin{definition}\label{def:truncated-units}
            Denote the $\pi_0$ of the orientation map by
            \begin{equation*}
                \omega \colon \pinD[\chrHeight] = \pi_0 \In \to \pi_0 \En\units.
            \end{equation*}
            We call an element in $\pi_0 \En\units$ a truncated unit if it is in the image of $\omega$.
        \end{definition}

        Note that a truncated unit must lie in the $p$-power torsion group
        \begin{equation*}
            \pi_0 \En\units[p^{\infty}] = \Witt(\Fpbar)[\![u_1,\dots,u_{\chrHeight-1}]\!]\units[p^{\infty}] = \begin{cases}
                \{1\}, & p \neq 2 \\
                \{\pm 1\}, & p = 2.
            \end{cases}
        \end{equation*}
        Therefore, when $p \neq 2$ there are no non-trivial truncated units. When $p=2$, only $-1$ can be a non-trivial truncated unit. 
        Our goal is therefore to understand when $-1\in \pi_0 \En\units$ is truncated in the case $p=2$. Until the end of this part we assume $p=2$. We will use the following notation:
        \begin{notation}
             Let $\pichar \in \pinD[\chrHeight]$. Define
             \begin{equation*}
                 \Sigma^{\pichar} \Enm \coloneqq \begin{cases}
                     \Sigma \Enm, & \omega(\pichar) = -1, \\
                     \Enm, & \omega(\pichar) = 1
                 \end{cases}
                 \qin \ModEnm.
             \end{equation*}
        \end{notation}

        \begin{definition}
            Let $\pichar \in \pinD[\chrHeight]$. Define
            \begin{equation*}
                \rho_{\pichar,k} \colon \B\ZZ/2^k \to \B\Sm[2^k] \into \MM \xto{\pichar} \In.
            \end{equation*}
            As before, we will abuse notation and sometimes identify it with the map of spaces
            \begin{equation*}
                \rho_{\pichar,k} \colon \B\ZZ/2^k \to \In \to \En\units \to \En.
            \end{equation*}
        \end{definition}

        \begin{lemma}\label{lem:character-of-rho-is-character-of-Tm}
             Let $\pichar \in\pinD[\chrHeight]$. Then 
            \begin{equation*}
                \chi_{\Tm[p^k](\Sigma^\pichar \Enm)} \simeq \chi_{\Enm{}[\rho_{\pichar,k}]},
            \end{equation*}
            where $\Tm[p^k](\Sigma^{\pichar} \Enm) \in (\ModEnm)^{\B\ZZ/p^k}$ is equipped with the restricted action of $\ZZ/p^k \subseteq \Sm[p^k]$.
        \end{lemma}

        \begin{proof}
            By \cref{lem:spherical-roots-of-unity-Nullstellensatz-shift}, the image of ${\Enm{}[\rho_{\pichar,k}]}$ and $\Tm[p^k](\Sigma^\pichar \Enm)$ is the same under the decategorification map
            \begin{equation*}
                \de \colon (\ModEnm)^{\B\ZZ/p^k} \to \En(L)^{\B\ZZ/p^k}.
            \end{equation*}
            Consider the character diagram \cref{cons:character-diagram} for height $\chrHeight-1$. Since the bottom horizontal map $\Enm^{\L \B \ZZ/p^k} \to \Enm(K)^{\L\B\ZZ/p^k} $ is injective on $\pi_0$, the claim follows from \cref{cor:commutativity-of-character-diagram-suspension-In} and \cref{lem:commutativity-of-character-diagram-En}.
        \end{proof}

        \begin{corollary}\label{cor:character-of-alpha}
            Let $0 \le t \le \chrHeight$. Choose a geometric point $C^{\chrHeight}_{t} \to \En[t](K_t)$ for some algebraically-closed field $K_t$. 
            Let $\pichar \in \pinD[\chrHeight]$. Define $\pichar_{t} \coloneqq \eta^{\chrHeight-t}\cdot \pichar \in \pi_{\chrHeight-t} (\In) = \pinD[t]$, where $\eta \in \pi_1 \SS$ acts by the $\SS$-module structure. Let $\omega_t \in \pi_0 \En[t](K_t)\units$ be the image of $\pichar_t$ under the orientation map. Then 
            \begin{equation*}
                \chi^{\chrHeight-t,\HKR}_{\rho_{\pichar,k}} \colon \L^{\chrHeight-t} \B \ZZ/2^k \simeq \bigsqcup_{(\ZZ/2^k)^{\chrHeight-t}} \B\ZZ/2^k \to C^{\chrHeight}_{t} \to \En[t](K_t)
            \end{equation*}
            can be computed inductively on $k$, as follows:
            \begin{enumerate}
                \item If $k = 0$, it is $\omega_t \in \En[t](K_t)\units$.
                \item On a connected component $(x_1,\dots,x_{\chrHeight-t})$ which is not 2-divisible, it is $\omega_t \in \En[t](K_t)\units$ with the trivial $\ZZ/2^k$-action.
                \item On a connected component of the form 
                \begin{equation*}
                    2(x_1,\dots, x_{\chrHeight-t}) \in (\ZZ/2^k)^{\chrHeight-t}, \qquad \text{for} \qquad (x_1,\dots, x_{\chrHeight-t}) \in (\ZZ/2^{k-1})^{\chrHeight-t},
                \end{equation*}
                it is $(\chi^{\chrHeight-t, \HKR}_{\rho_{\pichar,k-1}}(x_1,\dots,x_{\chrHeight-t}))^2$ computed as the induction (i.e.\ the 0-semiadditive integral) from $\ZZ/2^{k-1}$ to $\ZZ/2^k$, in $\En[t](K_t)\units$:
                \begin{equation*}
                    \begin{split}
                        \chi^{\chrHeight-t, \HKR}_{\rho_{\pichar,k-1}}(2x_1,\dots,2x_{\chrHeight-t}) 
                        & = \int^{\En[t](K_t)\units}_{\B\ZZ/2^{k-1} \to \B\ZZ/2^k} \chi^{\chrHeight-t, \HKR}_{\rho_{\pichar,k-1}}(x_1,\dots,x_{\chrHeight-t}) \\
                        & \simeq (\chi^{\chrHeight-t, \HKR}_{\rho_{\pichar,k}}(x_1,\dots,x_{\chrHeight-t}))^2 \qin (\En[t](K_t)\units)^{\B\ZZ/2^k}.
                    \end{split}
                \end{equation*}                
            \end{enumerate}
        \end{corollary}

        \begin{proof}
            This follows from \cref{cor:commutativity-of-character-diagram-suspension-In}, \cref{lem:character-of-rho-is-character-of-Tm}, \cref{lem:character-of-Tm} and \cref{cor:no-action-on-dim-En}.
        \end{proof}

        \begin{remark}
            Note that as we write the addition rule in $\In[t]$ and $\En[t](K_t)\units$ as multiplication, $\omega_t^2$ is the image of $\pichar_t^2$ which is the action of $2 \in \SS$ on $\pichar_t$. In particular,
            \begin{equation*}
                \pichar_t^2 = 2\eta^{\chrHeight-t} \cdot \pichar = 0 \cdot \pichar = \pichar^0 = 1
            \end{equation*}
            whenever $t < \chrHeight$, and therefore $\chi^{\chrHeight-t,\HKR}_{\rho_{\pichar,k}}$ sends any 2-divisible connected component to the connected component of $1 \in \En[t](\K_t)\units$. It can still admit a non-trivial $\ZZ/2^k$-action.
        \end{remark}

        \begin{remark}
            In particular, for $k=1$, 
            \begin{equation*}
                \chi^{\chrHeight-t,\HKR}_{\rho_{\pichar,1}} \colon \bigsqcup_{(\ZZ/2)^{\chrHeight-t}} \B\ZZ/2 \to \En[t](K_t)
            \end{equation*}
            is the map that 
            \begin{enumerate}
                \item on a connected component $(x_1,\dots, x_{\chrHeight-t}) \neq (0,\dots,0) \in (\ZZ/2)^{\chrHeight-t}$ is $\omega_t$ with the trivial $\ZZ/2$-action.
                \item on the connected component $(0,\dots,0)$ is $1 \simeq \omega_t^2$ in $\En[t](K_t)\units$, with the cyclic $\ZZ/2$-action.
            \end{enumerate}
        \end{remark}

        \begin{lemma}\label{lem:representation-is-trivial-iff-character}
            Let $\pichar \in \pinD[\chrHeight]$. Then $\rho_{\pichar,k} \colon \B\ZZ/2^k \to \En(K)\units$ is nullhomotopic if and only if there exists $0 \le t \le \chrHeight$ such that 
            \begin{equation*}
                \chi^{\chrHeight-t,\HKR}_{\rho_{\pichar,k}} \colon \L^{\chrHeight-t}\B\ZZ/2^k \to (C^{\chrHeight}_t)\units \to \En[t](K_t)\units
            \end{equation*}
            is null, where $C^{\chrHeight}_t \to \En[t](K_t)$ is a geometric point.
        \end{lemma}

        \begin{proof}
            The \quotes{only if} part is obvious. 
            Assume that $\chi^{\chrHeight-t, \HKR}_{\rho_{\pichar,k}}$ is null. Therefore, the $t$-fold character
            \begin{equation*}
                \chi^{t, \HKR}_{\chi^{\chrHeight-t, \HKR}_{\rho_{\pichar,k}}} \colon \L^{\chrHeight}\B\ZZ/2^k \to C^t_0(K_t)
            \end{equation*}
            is the constant map choosing $1 \in C^t_0(K_t)$.
            By \cref{lem:HKR-vs-twice-HKR} the following diagram is commutative
            \begin{equation*}
                \begin{tikzcd}
                    {\En(K)^{\B\ZZ/2^k}} & {C^{\chrHeight}_t(K)^{\L^{\chrHeight-t}\B\ZZ/2^k}} & {\En[t](K_t)^{\L^{\chrHeight-t}\B\ZZ/2^k}} \\
                    {C^{\chrHeight}_0(K)^{\L^{\chrHeight}\B\ZZ/2^k}} && {C^t_0(K_t)^{\L^{\chrHeight}\B\ZZ/2^k}.}
                    \arrow["{\chi^{\chrHeight-t,\HKR}}", from=1-1, to=1-2]
                    \arrow["{\chi^{\chrHeight,\HKR}}", from=1-1, to=2-1]
                    \arrow[from=1-2, to=1-3]
                    \arrow["{\chi^{t,\HKR}}", from=1-3, to=2-3]
                    \arrow[from=2-1, to=2-3]
                \end{tikzcd}
            \end{equation*}

            Combining with previous part, the composition
            \begin{equation*}
                \L^{\chrHeight}\B\ZZ/2^k \xto{\chi^{\chrHeight,\HKR}_{\rho_{\pichar,k}}} C^{\chrHeight}_0(K) \to C^t_0(K_t)
            \end{equation*}
            is the constant map choosing $1$.
        
            The character $\chi^{\chrHeight, \HKR}_{\rho_{\pichar,k}} \in C^{\chrHeight}_0(K)^{\L^{\chrHeight}\B\ZZ/2^k}$ is the character of the $\ZZ/2^k$-power operation for $G = e$, applied to $\omega(\pichar)$, as in \cite[Theorem 9.1]{Barthel-Stapleton-2017-power-operations}. Therefore, by the formula for the total power operation (appearing right before \cite[Proposition 5.5]{Barthel-Stapleton-2017-power-operations}), it lands in $\ZZ^{\L^{\chrHeight}\B\ZZ/2^k}$. In particular, as it is constant after pre-composition with $C^{\chrHeight}_0(K) \to C^t_0(K_t)$, it is the constant map choosing $1$.

            As $\ZZ/2^k$ is an $\En$-good group (\cite[Proposition~7.2]{Hopkins-Kuhn-Ravanel-2000-HKR}), the character map
            \begin{equation*}
                \chi^{\chrHeight,\HKR} \colon \En(K)^{\B\ZZ/2^k} \to C^{\chrHeight}_0(K)^{\L^{\chrHeight}\B\ZZ/2^k}
            \end{equation*}
            is injective on $\pi_0$, therefore $\rho_{\pichar,k}$ is the constant map choosing $1 \in \En(K)$ as needed.
        \end{proof}

        \begin{corollary}\label{cor:representation-is-null-for-n>=4}
            Let $\chrHeight \ge 4$ and $\pichar \in \pinD[\chrHeight]$. Then for any $k$, $\rho_{\pichar,k}$ is nullhomotopic.
        \end{corollary}

        \begin{proof}
            We use the criterion of \cref{lem:representation-is-trivial-iff-character} with $t = \chrHeight-4$. This now follows by \cref{cor:character-of-alpha}, since $\pichar_{t} = \eta^{4} \cdot \pichar = 1$.
        \end{proof}

        Recall the following fact from \cite{CSY-teleambi}:
        \begin{lemma}[{\cite[Lemma 5.3.3]{CSY-teleambi}}]\label{lem:cardinality-of-BZ/2}
            In $\ModEn$ at $p=2$, the cardinality of $\B\ZZ/2$ is
            \begin{equation*}
                |\B \ZZ/2| = 2^{\chrHeight-1}.
            \end{equation*}
        \end{lemma}
        
        We also need the following simple computation:

        \begin{lemma}\label{lem:integral-of-(-1)^2}
            Let $\chrHeight \ge 0$. Then, in $\ModEn$
            \begin{equation*}
                \int_{\B\ZZ/2} (-1)^2 = -2^{\chrHeight-1}+1.
            \end{equation*}
        \end{lemma}
        \begin{proof}
            The map 
            \begin{equation*}
                \begin{split}
                    \pi_0 \En & \to \pi_0 \En \\
                    x & \mapsto 2\int_{\B\Cn[2]}x^2-x^2
                \end{split}
            \end{equation*}
            is simply seen to be additive, and at $1$ it is equal to $2|\B\ZZ/2|-1 = 2^\chrHeight-1$. Therefore 
            \begin{equation*}
                \int_{\B\ZZ/2}(-1)^2 = \frac{1}{2}(2\int_{\B\ZZ/2}(-1)^2 - (-1)^2) + \frac{(-1)^2}{2} = -2^{\chrHeight-1}+1.
            \end{equation*}
        \end{proof}

        \begin{corollary}\label{cor:no-truncated-units-for-n>=4}
            Let $\chrHeight \ge 4$. Then $-1\in \pi_0 \En\units$ is not a truncated unit.  
        \end{corollary}

        \begin{proof}
            Assume otherwise, then there exists $\chrHeight \ge 4$ and $\pichar \in \pinD[\chrHeight]$ that maps to $-1\in \pi_0 \En\units$ under $\omega$. Therefore the map 
            \begin{equation*}
                \rho_{\pichar,1} \colon \B\ZZ/2 \into \MM \xto{\pichar} \In \to \En\units \to \En
            \end{equation*}
            is the map that chooses $(-1)^2$. By \cref{cor:representation-is-null-for-n>=4} this map is nullhomotopic, and therefore, by \cref{lem:cardinality-of-BZ/2} 
            \begin{equation*}
                \int_{\B\ZZ/2}(-1)^2 = |\B\ZZ/2| = 2^{\chrHeight-1}.
            \end{equation*}
            On the other hand, by \cref{lem:integral-of-(-1)^2},
            \begin{equation*}
                \int_{\B\ZZ/2}(-1)^2 = -2^{\chrHeight-1} + 1,
            \end{equation*}
            in contradiction.
        \end{proof}

        Finally, we show:
        
        \begin{lemma}\label{lem:-1-truncated-unit-for-n<=3}
            Let $\chrHeight \le 3$. Then $-1\in \pi_0 \En \units$ is a truncated unit. 
        \end{lemma}

        \begin{proof}
            For $\chrHeight=0$ the claim is trivial. 
            
            Note that for $1 \le \chrHeight \le 3$, the group $\pinD[\chrHeight]$ is cyclic and $2$-power torsion. Therefore, the claim is equivalent to saying that the generator of $\pinD[\chrHeight]$ is sent to $(-1)\in \pi_0 \En\units$.
            
            We prove the claim by induction on $\chrHeight$.
            Assume the claim is true for $\chrHeight$ and that $\chrHeight+1 \le 3$. 
            Let $\pichar \in \pinD$ be the generator (i.e.\ $\hat{\eta}$, $\hat{\eta^2}$, $\hat{\nu}$ for $\chrHeight=0,1,2$ respectively). Note that, $\eta \cdot \pichar$ is $(-1)\in \pinD[0] = \QQ/\ZZ$, when $\chrHeight = 0$, and is the generator of $\pinD[\chrHeight]$ when $\chrHeight > 0$.
            In both cases, by it being $-1$ or by the induction hypothesis, $\eta \cdot \pichar$ is sent to $(-1) \in \pi_0 \En\units$.
            
            Consider the decategorification map
            \begin{equation*}
                \de \colon ((\ModEn)\dblspace)^{\B\ZZ/2} \to \KTnp(\En)^{\B\ZZ/2}\to \Enp(L)^{\B\ZZ/2}.
            \end{equation*}
            Then, by \cref{lem:spherical-roots-of-unity-Nullstellensatz-shift}, it sends the representation $\En{}[\rho_{\pichar,1}]$ to 
            \begin{equation*}
                \rho_{\pichar,1} \simeq \de(\En{}[\rho_{\pichar,1}]) \colon \B\ZZ/2 \into \MM \xto{\pichar} \Inplus \to \Enp(L)\units \to \Enp(L).
            \end{equation*}
            By \cref{lem:integral-Enp-to-En} we get 
            \begin{equation*}
                \int_{\L\B\ZZ/2}^{\En}\chi_{\En{}[\rho_{\pichar,1}]}=\dim(\colim_{\B\ZZ/2} \En{}[\rho_{\pichar,1}])=\int_{\B\ZZ/2}^{\Enp} \rho_{\pichar,1}.
            \end{equation*}
            By our assumption, \cref{cor:character-of-alpha} and \cref{lem:integral-of-(-1)^2}:
            \begin{equation*}
                \int_{\L\B\ZZ/2}^{\En}\chi_{\rho_{\pichar,1}}= \int_{\B\ZZ/2}^{\En}(-1)^2-|\B\ZZ/2|_{\En}= -2^{\chrHeight-1}+1-2^{\chrHeight-1}=-2^{n}+1,
            \end{equation*}
            which by \cref{lem:integral-of-(-1)^2} and \cref{lem:cardinality-of-BZ/2}, is euqal to $\int_{\B\ZZ/2}^{\Enp} (-1)^2$ and not to $\int_{\B\ZZ/2}^{\Enp} 1^2=|\B\ZZ/2|_{\Enp}$.
        \end{proof}

        We summarize the discussion about truncated units, specifically \cref{cor:no-truncated-units-for-n>=4} and \cref{lem:-1-truncated-unit-for-n<=3}:
        \begin{proposition}\label{prop:truncated-units}
            Let $p$ be a prime and $\chrHeight \ge 0$. Then
            \begin{enumerate}
                \item If $p > 2$ or $\chrHeight \ge 4$ then there are no non-trivial truncated units in $\En\units$.
                \item If $p=2$ and $\chrHeight \le 3$ then $\pm 1$ are exactly the truncated units in $\En\units$.
            \end{enumerate}
        \end{proposition}

        %%%%%%%%%%%%%%%%%%%%%%%%%%%%%%%%%%%%%%%%%%%%%%%%%%%%%%%%%%%%%%%%%%%%%%%%%%%%%%%%
        \subsubsection{The $\TT$-action and the braiding character}
        \label{subsubsec:T-action-chromatic}

        We now combine the previous results to show that the reduced $\B\ZZ/p^k$-action on the monoidal dimensions is trivial in our cases of interest. We then leverage this to compute the braiding character at all primes and heights.
        
        \begin{proposition}\label{prop:chromatic-T-action-trivial}
            Let $\pichar \in \pinD$. Then the action of $\ZZ/p^k \subseteq\TT$ on the dimension of $\En\shift{1} \in \Gr^{\pichar}_\ZZ\ModEn$ is trivial.  
        \end{proposition}

        \begin{proof}
            Assume first that $p$ is odd or $\chrHeight \ge 4$. By \cref{prop:truncated-units}(1), the map 
            \begin{equation*}
                \rho_{\pichar,k} \colon \B\ZZ/p^k \to \SS \xto{\pichar} \Inplus \to \Enp\units
            \end{equation*}
            is nullhomotopic, and therefore $\En{}[\rho_{\pichar,k}]$ is mapped, under the decategorification map,
            \begin{equation*}
                \de \colon ((\ModEn)\dblspace)^{\B\ZZ/2^k} \to \Enp(L)^{\B\ZZ/2^k}
            \end{equation*}
           to a constant map choosing $1$. 
           By \cref{cor:commutativity-of-character-diagram-suspension-In}, using that the map $\En^{\L\B\ZZ/2^k}\to \En(K)^{\L\B\ZZ/2^k}$ is injective on $\pi_0$, the character of $\En{}[\rho_{\pichar,k}]$ is the constant map choosing 1. 
           But by \cref{lem:character-of-Tm}, this character, on a connected component of an element different from $0$, is the $\B\ZZ/p^k$ action on $\dim(\En\shift{1})$ coming from the $\TT$ action.

           Assume now that $p = 2$ and $\chrHeight \le 3$. Then the same argument, shows that the character of $\En{}[\rho_{\pichar,k}]$ agrees with the character of $(\Sigma \En)^{\otimes 2^k} \in (\ModEn)^{\B\ZZ/2^k}\textit{}$. The claim now follows from \cref{cor:no-action-on-dim-En}.   
        \end{proof}

        \begin{corollary}\label{cor:chromatic-T-action-zeta-trivial}
            Let 
            \begin{equation*}
                \zeta \colon \pinD \to\Sigma^2\In\to  \Mod_{\ModEn}
            \end{equation*}
            as in \cref{def:zeta}. Then the restricted $\ZZ/p^k$-action on the dimension of $\En\shift{\pichar} \in \ModEn{}[\omega^{(0)}_{\SS}]$ is trivial for all $\pichar \in\pinD$.
        \end{corollary}

        \begin{proof}
            There is a symmetric monoidal functor $\Gr^{\pichar}_{\ZZ} \ModEn \to \ModEn{}[\omega^{(0)}_{\SS}]$, sending $\En\shift{1}$ to $\En\shift{\pichar}$.
        \end{proof}

        \begin{corollary}\label{cor:T-action-trivial-in-zeta}
            Let $V\in (\ModEn{}[\omega^{(0)}_{\SS}])\dbl$. Then the restricted $\ZZ/p^k$-action on $\dim(V)$ is trivial. 
        \end{corollary}

        \begin{proof}        
            Write $V$ as $V = \bigoplus_{\pichar \in \pinD} V_{\pichar}\shift{\pichar}$. 
            The $\ZZ/p^k$-action on $V$ can be understood degree-wise, i.e.\ in $\THH(\ModEn{}[\omega^{(0)}_{\SS}])^{\B\ZZ/p^k} = \ounit[\pinD]^{\B\ZZ/p^k}$. Therefore, it suffices to show that the $\ZZ/p^k$-action on $\dim(V_{\pichar}\shift{\pichar}) = \dim(V_{\pichar}) \cdot \dim(\En\shift{\pichar})$ is trivial. This follows by \cref{cor:no-action-on-dim-En} and \cref{cor:chromatic-T-action-zeta-trivial}.
        \end{proof}

        \begin{corollary}\label{cor:T-action-is-trivial-in-Z-graded}
            Let $V\in \Gr^{\pichar}_{\ZZ}(\ModEn)\dbl$. Then the restricted $\ZZ/p^k$-action on $\dim(V)$ is trivial. 
        \end{corollary}
        \begin{proof}
            The symmetric monoidal map $\Gr^{\pichar}_{\ZZ} \ModEn \to \ModEn{}[\omega^{(0)}_{\SS}]$ sends $V$ to a dualizable object, and the dimension to the dimension. As both categories share a unit, this follows from \cref{cor:T-action-trivial-in-zeta}.
        \end{proof}
        
        Putting everything together, from \cref{thm:braiding-depends-only-on-dim}, \cref{cor:T-action-is-trivial-in-Z-graded} and \cref{prop:truncated-units},  we get 

        \begin{theorem}\label{thm:chromatic-braiding-character}
            Let $\pichar\in \pinD$. Then
            \begin{enumerate}
                \item If $p=2$, $n \le 2$ and $\pichar$ is not 2-divisible, then the the braiding character of $\Gr^{\pichar}_{\ZZ} \ModEn$ is the braiding character of $\Sigma\En\in\ModEn$;
                \item Otherwise, the braiding character of $\Gr^{\pichar}_{\ZZ} \ModEn$ is the braiding character of $\En \in \ModEn$.
            \end{enumerate}
        \end{theorem}        

        % \begin{remark}
        %      Opening up some definition one can see that a variant of the braiding character of the form $\L\SS_p^{\B\Zp} \to \ounit[t^\pm]$ which comes from taking free loop after the group completion determines the symmetric monoidal structure of $\Gr_\ZZ(\ModEn)$. We call this the spherical braiding character. By \cref{cor:graded-trivial-iff-picard} it does so in the following way: the symmetric monoidal structure on $\Gr^{\pichar}_{\ZZ} \ModEn$ is equivalent to the day convolution if and only if there exists $V\in \ModEn$ such that $V\shift{1}\in \Gr_\ZZ(\ModEn)$ has the same spherical braiding character as $\En$ or as $\Sigma\En$. We see that \cref{thm:chromatic-braiding-character} says that after pre-composing to $\L_p(\MM_p) \to \L\SS_p^{\B\Zp}$ this is indeed the case. One can also identify the above map and show it is equivalent to $\colim_k \SS^{\B\ZZ/p^k} \to \SS^{\B\Zp}$.  
        % \end{remark}

\bibliographystyle{alpha}
\phantomsection\addcontentsline{toc}{section}{\refname}
\bibliography{references}

\end{document}